\documentclass[a4paper, twoside]{article}
\usepackage{a4}

\usepackage{amsfonts}
\usepackage{amsmath}
\usepackage{amssymb}
\usepackage{mathrsfs}
\usepackage{graphicx}
\usepackage{epstopdf}
\usepackage{wrapfig}
\usepackage{upref}
\usepackage[active]{srcltx}
\usepackage[dvips,colorlinks,citecolor=blue,linkcolor=blue]{hyperref}
\usepackage[dvipsnames]{color}
\usepackage{pgf,tikz}
\usepackage{subfig}
\usetikzlibrary{arrows}
\definecolor{zzttqq}{rgb}{0.27,0.27,0.27}
\definecolor{tttttt}{rgb}{0.2,0.2,0.2}
\definecolor{wwwwww}{rgb}{0.4,0.4,0.4}

\definecolor{blau}{rgb}{0.1,0.0,0.9}
\definecolor{gruen}{cmyk}{1.0,0.2,0.7,0.07}
\definecolor{mag}{cmyk}{0.0,0.9,0.3,0.0}

\newcounter{komcounter}
\numberwithin{komcounter}{section}

\newcount\minutes \newcount\hours
\hours=\time
\divide\hours 60
\minutes=\hours
\multiply\minutes -60
\advance\minutes \time
\newcommand{\klockan}{\the\hours:{\ifnum\minutes<10 0\fi}\the\minutes}
\newcommand{\tid}{\today\ \klockan}
\newcommand{\prtid}{\smash{\raise 10mm \hbox{\LaTeX ed \tid}}}
\renewcommand{\prtid}{}
\makeatletter
\def\sectionmark#1{} 
\def\subsectionmark#1{}
\newcommand{\sectnr}{\ifnum \c@secnumdepth >\z@
                 \thesection.\hskip 1em\relax \fi}
\def\@evenhead{\footnotesize\rm\thepage\hfil\leftmark\hfil\llap{\prtid}}
\def\@oddhead{\footnotesize\rm\rlap{\prtid}\hfil\rightmark\hfil\thepage}
\def\tableofcontents{\section*{Contents} 
 \@starttoc{toc}}
\makeatother
\makeatletter
\let\Thebibliography=\thebibliography
\renewcommand{\thebibliography}[1]{\def\@mkboth##1##2{}\Thebibliography{#1}
\addcontentsline{toc}{section}{References}
\frenchspacing 
\setlength{\@topsep}{0pt}
\setlength{\itemsep}{0pt}%
\setlength{\parskip}{0pt plus 2pt}%
}
\makeatother
\newcommand{\authortitle}[2]{\author{#1}\title{#2}\markboth{#1}{#2}}
\makeatletter
\def\@seccntformat#1{\csname the#1\endcsname.\quad}
\makeatother
\makeatletter
\long\def\@makecaption#1#2{%
  \vskip\abovecaptionskip
  \sbox\@tempboxa{ #1. #2}%
  \ifdim \wd\@tempboxa >\hsize
    #1. #2\par
  \else
    \global \@minipagefalse
    \hb@xt@\hsize{\hfil\box\@tempboxa\hfil}%
  \fi
  \vskip\belowcaptionskip}
\makeatother
\makeatletter
\renewcommand*\l@section[2]{%
  \ifnum \c@tocdepth >\z@
    \addpenalty\@secpenalty
    \setlength\@tempdima{2em}
    \begingroup
      \parindent \z@ \rightskip \@pnumwidth
      \parfillskip -\@pnumwidth
      \leavevmode \bfseries
      \advance\leftskip\@tempdima
      \hskip -\leftskip
      #1\nobreak\hfil \nobreak\hb@xt@\@pnumwidth{\hss #2}\par
    \endgroup
  \fi}
\def\numberline#1{\hb@xt@\@tempdima{#1.\hfil}} 
\makeatother
\DeclareMathOperator{\Mod}{Mod}
\newcommand{\Modp}{\Mod_p}
\DeclareMathOperator{\capa}{cap}
\newcommand{\capp}{\capa_p}
\DeclareMathOperator{\dist}{dist}
\DeclareMathOperator{\diam}{diam}
\DeclareMathOperator{\spt}{supp}
\DeclareMathOperator{\supp}{supp}
\DeclareMathOperator{\interior}{int}
\newcommand{\Cp}{{C_p}}

\RequirePackage{amsthm}
\theoremstyle{plain}
\newtheorem{theorem}{Theorem}[section]
\newtheorem{cor}[theorem]{Corollary}
\newtheorem{thm}[theorem]{Theorem}
\newtheorem{lemma}[theorem]{Lemma}
\newtheorem{lem}[theorem]{Lemma}
\newtheorem{prop}[theorem]{Proposition}

\theoremstyle{definition}
\newtheorem{definition}[theorem]{Definition}
\newtheorem{deff}[theorem]{Definition}
\newtheorem{ex}[theorem]{Example}
\newtheorem{example}[theorem]{Example}
\newtheorem{rem}[theorem]{Remark}
\newtheorem{remark}[theorem]{Remark}
\newtheorem{openprob}[theorem]{Open problem}

\numberwithin{equation}{section}
\newcounter{saveenumi}

\makeatletter
\let\Enumerate=\enumerate
\renewcommand{\enumerate}{\Enumerate%
\setlength{\@topsep}{0pt}
\setlength{\itemsep}{0pt}%
\setlength{\parskip}{0pt plus 1pt}%
\renewcommand{\theenumi}{\textup{(\alph{enumi})}}%
\renewcommand{\labelenumi}{\theenumi}%
}
\let\endEnumerate=\endenumerate
\renewcommand{\endenumerate}{\endEnumerate\unskip}
\makeatother
\newcommand{\art}[6]{{\sc #1, \rm #2, \it #3\/ \bf #4 \rm (#5), \mbox{#6}.}}
\newcommand{\auth}[2]{{#2. #1}}
\def\idxauth{\auth}
\newcommand{\artprep}[3]{{\sc #1, \rm #2, #3.}}
\newcommand{\artin}[3]{{\sc #1, \rm #2, in #3.}}
\newcommand{\book}[3]{{\sc #1, \it #2, \rm #3.}}
\newcommand{\AND}{{\rm and }}

\newcommand{\eqv}{\ensuremath{
\mathchoice{\quad \Longleftrightarrow \quad}{\Leftrightarrow}
                {\Leftrightarrow}{\Leftrightarrow}} }
\newcommand{\imp}{\ensuremath{\Rightarrow} }

{\catcode`p =12 \catcode`t =12 \gdef\eeaa#1pt{#1}}      
\def\accentadjtext#1{\setbox0\hbox{$#1$}\kern   
                \expandafter\eeaa\the\fontdimen1\textfont1 \ht0 }
\def\accentadjscript#1{\setbox0\hbox{$#1$}\kern 
                \expandafter\eeaa\the\fontdimen1\scriptfont1 \ht0 }
\def\accentadjscriptscript#1{\setbox0\hbox{$#1$}\kern   
                \expandafter\eeaa\the\fontdimen1\scriptscriptfont1 \ht0 }
\def\accentadjtextback#1{\setbox0\hbox{$#1$}\kern       
                -\expandafter\eeaa\the\fontdimen1\textfont1 \ht0 }
\def\accentadjscriptback#1{\setbox0\hbox{$#1$}\kern     
                -\expandafter\eeaa\the\fontdimen1\scriptfont1 \ht0 }
\def\accentadjscriptscriptback#1{\setbox0\hbox{$#1$}\kern 
                -\expandafter\eeaa\the\fontdimen1\scriptscriptfont1 \ht0 }
\def\itoverline#1{{\mathsurround0pt\mathchoice
        {\rlap{$\accentadjtext{\displaystyle #1}
                \accentadjtext{\vrule height1.593pt}
                \overline{\phantom{\displaystyle #1}
                \accentadjtextback{\displaystyle #1}}$}{#1}}
        {\rlap{$\accentadjtext{\textstyle #1}
                \accentadjtext{\vrule height1.593pt}
                \overline{\phantom{\textstyle #1}
                \accentadjtextback{\textstyle #1}}$}{#1}}
        {\rlap{$\accentadjscript{\scriptstyle #1}
                \accentadjscript{\vrule height1.593pt}
                \overline{\phantom{\scriptstyle #1}
                \accentadjscriptback{\scriptstyle #1}}$}{#1}}
        {\rlap{$\accentadjscriptscript{\scriptscriptstyle #1}
                \accentadjscriptscript{\vrule height1.593pt}
                \overline{\phantom{\scriptscriptstyle #1}
                \accentadjscriptscriptback{\scriptscriptstyle #1}}$}{#1}}}}

\newcommand{\setm}{\setminus}
\renewcommand{\subsetneq}{\varsubsetneq}
\renewcommand{\emptyset}{\varnothing}
\def\vint{\mathop{\mathchoice%
          {\setbox0\hbox{$\displaystyle\intop$}\kern 0.22\wd0%
           \vcenter{\hrule width 0.6\wd0}\kern -0.82\wd0}%
          {\setbox0\hbox{$\textstyle\intop$}\kern 0.2\wd0%
           \vcenter{\hrule width 0.6\wd0}\kern -0.8\wd0}%
          {\setbox0\hbox{$\scriptstyle\intop$}\kern 0.2\wd0%
           \vcenter{\hrule width 0.6\wd0}\kern -0.8\wd0}%
          {\setbox0\hbox{$\scriptscriptstyle\intop$}\kern 0.2\wd0%
           \vcenter{\hrule width 0.6\wd0}\kern -0.8\wd0}}%
          \mathopen{}\int}

\newcommand{\PP}{{\cal P}}
\newcommand{\R}{{\mathbb R}}
\newcommand{\oR}{{\overline{\mathbb R}}}
\newcommand{\N}{{\mathbb N}}
\newcommand{\C}{{\mathbb C}}
\newcommand{\D}{{\mathbb D}}
\newcommand{\p}{{$p\mspace{1mu}$}}
\newcommand{\pp}{{$p\mspace{1mu}$}}
\newcommand{\Np}{N^{1,p}}
\newcommand{\de}{\delta}
\newcommand{\End}{E}
\newcommand{\g}{\gamma}
\newcommand{\ga}{\gamma}
\newcommand{\gat}{\tilde{\gamma}}
\newcommand{\G}{\Gamma}
\newcommand{\Ga}{\Gamma}
\newcommand{\Om}{\Omega}
\newcommand{\eps}{\varepsilon}
\newcommand{\la}{\lambda}
\newcommand{\dmu}{d\mu}
\newcommand{\loc}{_{\rm loc}}
\newcommand{\al}{\alpha}
\newcommand{\bs}{\setminus}
\newcommand{\bdry}{\partial}
\newcommand{\clOm}{{\overline{\Om}}}
\newcommand{\clOmm}{{\overline{\Om}\mspace{1mu}}^M}
\newcommand{\clOmP}{{\overline{\Om}\mspace{1mu}}^P}
\newcommand{\Gj}[2]{G_{#1}(#2)}
\newcommand{\Hr}[1]{H{(#1)}}

\newcommand{\xij}{{x_{i,j}}}
\newcommand{\xijj}{{x_{i,j+1}}}
\newcommand{\xixjj}{{x_{i_x,j+1}}}
\newcommand{\xiio}{{x_{i+1,0}}}
\newcommand{\xoj}{{x_{0,j}}}
\newcommand{\Bij}{{B_{i,j}}}
\newcommand{\Bixj}{{B_{i_x,j}}}
\newcommand{\Bijj}{{B_{i,j+1}}}
\newcommand{\Bixjj}{{B_{i_x,j+1}}}
\newcommand{\Biio}{{B_{i+1,0}}}
\newcommand{\Bio}{{B_{i,0}}}
\newcommand{\Boo}{{B_{0,0}}}
\newcommand{\Boj}{{B_{0,j}}}
\newcommand{\dM}{d_M}
\newcommand{\bdy}{\partial}
\newcommand{\bdySP}{\partial_{\rm SP}}
\newcommand{\bdyM}{\partial_{M}}
\newcommand{\Gjr}{\Gj{j}{r}}
\newcommand{\Ct}{\widetilde{C}}
\newcommand{\Omt}{\widetilde{\Om}}
\newcommand{\CPI}{C_{\textup{PI}}}

\newenvironment{ack}{\medskip{\it Acknowledgement.}}{}

\begin{document}

\authortitle{Tomasz Adamowicz, Anders Bj\"orn, Jana Bj\"orn
    and Nageswari Shanmugalingam}
{Prime ends for domains in metric spaces}
\author{
Tomasz Adamowicz
\\
\it\small Department of Mathematics, Link\"opings universitet, \\
\it\small SE-581 83 Link\"oping, Sweden\/{\rm ;}
\it \small tadamowi@gmail.com 
\\
\\
Anders Bj\"orn \\
\it\small Department of Mathematics, Link\"opings universitet, \\
\it\small SE-581 83 Link\"oping, Sweden\/{\rm ;}
\it \small anders.bjorn@liu.se
\\
\\
Jana Bj\"orn \\
\it\small Department of Mathematics, Link\"opings universitet, \\
\it\small SE-581 83 Link\"oping, Sweden\/{\rm ;}
\it \small jana.bjorn@liu.se
\\
\\
Nageswari Shanmugalingam
\\
\it \small  Department of Mathematical Sciences, University of Cincinnati, \\
\it \small  P.O.\ Box 210025, Cincinnati, OH 45221-0025, U.S.A.\/{\rm ;}
\it \small shanmun@uc.edu
}

\date{April 27, 2012}
\maketitle

\noindent{\small
{\bf Abstract}.
In this paper we propose a new definition of prime ends for domains in metric spaces 
under rather general assumptions. We compare our
prime ends to those of Carath\'eodory and N\"akki. 
Modulus ends and prime ends, defined by means of the \p-modulus of curve 
families, are also discussed and related to the prime ends. 
We provide 
characterizations of singleton prime ends and 
relate them to the notion of accessibility 
of boundary points, and introduce a topology on the prime  end boundary. 
We also study relations between the prime end boundary and the Mazurkiewicz boundary.
Generalizing the notion of John domains, we introduce  almost John domains,
and we investigate
prime ends in the settings of John domains, almost
John domains and domains which are finitely connected at the boundary. 
}

\bigskip
\noindent
{\small \emph{Key words and phrases}:
Accessibility,
almost John domain,
capacity, 
doubling measure,
end,
finitely connected at the boundary,
John domain,
locally connected,
Mazurkiewicz distance,
metric space, 
\p-modulus, 
Poincar\'e inequality,
prime end,
uniform domain.
}


\tableofcontents

\section{Introduction}
The classical Dirichlet boundary value problem associated with 
a differential operator $L$ consists in
finding a function $u$
 which  satisfies the equation $L u = 0$ on  a domain $\Om$ 
and the boundary condition $u=f$ on $\partial\Om$ for given boundary
data $f:\partial\Om\to\R$.  
This problem has been studied extensively for various elliptic differential
operators, including the Laplacian $\Delta$ and
its nonlinear counterpart
the \p-Laplacian $\Delta_p$. 
Perhaps the most general method for solving the Dirichlet problem
for these equations is the Perron method introduced independently by 
Perron~\cite{per} and Remak~\cite{remak},
and further refined in the linear case
by Wiener and Brelot
(and therefore often called the PWB method in the linear case).
For the nonlinear case see
Heinonen--Kilpel\"ainen--Martio~\cite{hkm} and the notes therein,
and Bj\"orn--Bj\"orn--Shanmugalingam~\cite{BBS2}.

The Dirichlet problem, as posed above with $f$ defined on
the topological boundary $\bdry\Om$, 
is in many cases unnecessarily restrictive. 
For example, in the slit disk (see Example~\ref{ex-slit-disc})
one boundary value 
is prescribed for each point in the slit,
even though it would be more natural 
to have two boundary values 
at those points (except for the tip),
obtained by approaching the slit from either side.
On the other hand, in some domains with complicated boundary there may
be nontrivial parts of the boundary which are essentially invisible for
the solutions and therefore should be treated 
accordingly in the Dirichlet problem.

For linear operators such as $\Delta$, 
this drawback has been earlier addressed on $\R^n$ and Riemannian manifolds
using the Martin boundary, see Martin~\cite{martin}, Ancona~\cite{Anc2}, \cite{Anc3}
and Anderson--Schoen~\cite{AnSc}. The minimal Martin kernel functions,
which compose the Martin boundary of the domain, are analogs of 
Poisson kernels for more irregular domains and
provide us with integral representations for the solutions of the corresponding Dirichlet problem. In the 
slit disk one can see that there are two distinct minimal Martin kernels corresponding to each point in the slit
(except for the  tip).
Although, as shown by e.g.\ Holopainen--Shanmugalingam--Tyson~\cite{HST}
and Lewis--Nystr\"om~\cite{LNy}, 
a Martin boundary can be defined even for nonlinear operators such as
the \p-Laplacian 
and its generalizations to metric spaces,
we cannot hope to use the Martin boundary as a 
kernel for solving
the corresponding Dirichlet problem in the nonlinear case. 

The goal of this
paper is to instead develop an alternative notion of boundary, 
called the prime end boundary, which can give rise  to
a more comprehensive potential theory suitable for the Perron method
and taking the above geometrical concern into account. 
Prime ends were introduced 
by Carath\'eodory~\cite{car}  in 1913  for simply connected planar domains.
His approach is suitable also for finitely connected planar domains,
but to be able to treat more general domains satisfactorily we propose
a different definition. 
Even though we work in metric spaces, 
our results are new also for simply connected planar domains,
and our prime ends are different
from Carath\'eodory's also in this case.

Roughly speaking, a prime end corresponds to a part of the topological
boundary which can be reached from the domain in a connected way.
In very special domains, such as uniform domains
and domains which are locally connected at the boundary, the prime end boundary 
coincides with the topological boundary, but in more general domains
it need not even satisfy the T2 separation condition.
In between, there is a large class of domains for which the prime end boundary
behaves well and
provides more flexibility in the Dirichlet problem than the topological
boundary.

We introduce a natural topology on the prime end boundary, 
with the aim at solving the Dirichlet problem with respect to 
the prime end boundary.
Such a Dirichlet problem, and the related potential theory, are studied 
in a companion paper to this one, 
Bj\"orn--Bj\"orn--Shanmugalingam~\cite{BBSdir}.
Even the standard 
Dirichlet problem with respect to the topological boundary
benefits from the study of prime ends,
see A.~Bj\"orn~\cite{ABcomb}.

As already mentioned, the notion of prime ends was first proposed by 
Cara\-th\'eo\-dory for simply connected planar domains.
It was later used successfully by e.g.\  Beurling~\cite{Beu}, 
Ohtsuka~\cite{oh}, Ahlfors~\cite{Ah2},  N\"akki~\cite{na}  and 
Minda--N\"akki~\cite{Min} 
to study problems related to boundary regularity of
conformal and quasiconformal mappings in Euclidean domains. 
Others who have formulated versions of prime ends include
Kaufmann~\cite{Kau}, Mazurkiewicz~\cite{Maz},  Epstein~\cite{Ep}
and Karmazin~\cite{Ka3}, \cite{Ka}.
More recently, prime ends have been used by e.g.\  Ancona~\cite{Anc} and
Rempe~\cite{Re} in various settings to study problems related to
potential theory and dynamical systems.
These studies are set in Euclidean domains or domains in a topological 
manifold (as in~\cite{Maz}), and generally require
that the domain be a simply connected planar domain,
or that it is locally connected at the boundary.
The literature on prime ends is quite substantial, and we cannot hope 
to provide an exhaustive
list of references here; we recommend the interested reader to also 
consider papers cited in the above references.

The exposition of this paper is as follows. 
Section~\ref{sect-prelim} is devoted to the preliminary 
notions and definitions needed in the paper.  
Our aim is to develop prime ends in metric spaces under rather general
assumptions.
We use the standard assumptions that the metric space is complete 
and equipped with a doubling measure supporting a Poincar\'e inequality,
but often these assumptions can be substantially weakened; this
is pointed out at the end of Section~\ref{sect-prelim}.
Examples of spaces satisfying the standard assumptions mentioned above include
various manifolds, Heisenberg and Carnot groups, Carnot--Carath\'eodory spaces,
certain fractals and some
closed subsets of $\R^n$, see Appendix~A in
Bj\"orn--Bj\"orn~\cite{BBbook}.

In Section~\ref{sect-Car} we give a brief description of Carath\'eodory's
notion of prime ends, the prototype 
for most subsequent notions of prime ends, including ours.
The discussion in Section~\ref{sect-Car} provides a framework for 
our definition of ends 
and prime ends introduced 
in Section~\ref{ends-prime-e}. 
We illustrate our notion of ends and prime ends 
through various examples in Section~\ref{comp-PEvsCar}, where we also
compare our ends and prime ends with Carath\'eodory's and discuss some 
advantages of our approach.

To tie in the nonlinear potential theory, in Section~\ref{sect-Modp-ends}
we further refine the notions of ends and prime ends by introducing 
$\text{Mod}_p$-ends and $\text{Mod}_p$-prime ends, which 
through the \p-modulus of curve 
families take into account the geometry of the domain.
This part is somewhat motivated by the works
of Ahlfors~\cite{Ah2},  Beurling~\cite{Beu} and
N\"akki~\cite{na}.
We also relate our prime ends to \p-parabolicity and point out 
some geometrical differences 
arising from the conformal and nonconformal \p-modulus.

In~\cite{CL}, Collingwood and Lohwater provided classifications 
of Carath\'eodory prime ends. 
For us,  the class of singleton prime ends is 
the most interesting because of its connection to (path)accessibility and
finite connectedness at the boundary. 
Hence in Section~\ref{sect-access} we study ends with
singleton impressions, 
provide geometric conditions for an end to be a prime end, and
present a simple way of constructing prime ends 
 at accessible boundary points. 
For these results our modification of Carath\'eodory's definition
is essential.

The goal of Section~\ref{sect-top} is to construct a topology
on the collection of (prime) ends, or rather
on the (prime) end closure of the domain.  
Here we also
characterize the convergence of a sequence of ends to an end in terms of the 
Mazurkiewicz distance (sometimes called the inner diameter distance).
Section~\ref{sect-Mazur} is devoted to studying the topology on the prime end boundary.
We show that the set of singleton prime ends is homeomorphic to the Mazurkiewicz boundary 
obtained by completing the domain with respect to its Mazurkiewicz distance,
and this homeomorphism extends in a natural way to the domain itself. 

Section~\ref{sect-finconn} focuses on studying the prime end boundary 
of domains which are finitely connected at the boundary. 
We show that in such domains 
the prime end closure of the domain is metrizable, and that
all prime ends have singleton impressions and
can be obtained as connected
components of small neighborhoods of boundary points. 

The final section of this paper, Section~\ref{sect-John}, 
focuses on a more special class of domains, namely John domains.
For a more extensive theory,
we introduce almost John domains,
which e.g.\ makes it possible to include some cusp domains
in our discussion.
Here we show that such domains are finitely connected at the boundary, 
have only singleton prime ends, and that each boundary point is 
accessible and corresponds to the impression of at least one prime end.
For certain values of $p$ we also show that prime ends 
in almost John domains are exactly the
$\Modp$-ends, which reflects the use of extremal length in the construction
of Carath\'eodory prime ends due to Schlesinger~\cite{sc}.
For the more special class of uniform domains some additional results are 
obtained.

Auxiliary results related to modulus and capacitary estimates used in 
this paper are collected in an appendix. 
Some of them are also of independent interest.
Examples are given throughout the paper to illustrate various 
geometric situations and properties of prime ends.
In fact, all the examples given
in this paper are Euclidean domains, and indicate the possibilities 
which can occur even in the Euclidean setting when the
domains are not as nice as the ones considered in the works of Ahlfors, 
Carath\'eodory and N\"akki. 
The theorems, on the other hand, are formulated under least possible
assumptions to emphasize the generality of our theory.

\begin{ack}
This research began in 2008  when T.~A. was a postdoc at
the University of Cincinnati.
Part of the research was done during the visit of A.~B. and J.~B to
the University of Cincinnati in 2010,
and during the visit of N.~S. to Link\"opings universitet in 2011.
The authors wish to thank these institutions for their kind hospitality.
We also wish to thank David Herron, David Minda and Raimo N\"akki for
helpful discussions related to this research.

A.~B. and J.~B. were supported by the Swedish Science Research Council.
A.~B. was also a Fulbright scholar during his
visit to the University of Cincinnati, supported by the Swedish
Fulbright Commission,
while J.~B. was a Visiting Taft Fellow
during her visit to the University of Cincinnati,
supported by the Charles Phelps Taft Research Center at the University
of Cincinnati.
N.~S. was supported by the Taft Research Center
and by the Simons Foundation grant \#200474.

A.~B. and J.~B. belong to the European Science
Foundation Networking Programme \emph{Harmonic and Complex Analysis and
Applications} and to the NordForsk network \emph{Analysis and Applications}.
\end{ack}

\section{Preliminaries}
\label{sect-prelim}

Let $(X,d,\mu)$ be a metric measure space equipped with
a metric $d$ and a measure $\mu$ (and containing more than one point).
We will
assume that  $\mu$ is a Borel measure such that
$0<\mu(B)<\infty$ for all balls $B$ in $X$.

We also let $1 \le p  <\infty$
be fixed. We shall sometimes impose additional assumptions on $p$.
Throughout the paper, $\Om\subsetneq X$ will be a bounded domain in
$X$, i.e.\ a  
bounded nonempty connected open subset of $X$ that is not $X$ itself.

A wide class of metric measure spaces of current interest, including
weighted and unweighted Euclidean spaces, Riemannian manifolds,
Heisenberg groups, and other Carnot--Carath\'eodory spaces, all have
locally doubling measures that support a Poincar\'e inequality
locally. Since we are interested in unifying potential theory on all
these spaces, we will assume these properties for the metric spaces
considered in this paper. Because the domain under consideration is
assumed to be bounded, there is no loss of generality in assuming
the doubling condition and the Poincar\'e inequality as global
properties, with only simple modifications needed to go from spaces
with globally held properties to spaces with locally held
properties.

A measure $\mu$ is said to be \emph{doubling\/} if there is a constant $C_\mu>0$
such that for all balls $B=B(x,r)=\{y \in X : d(x,y)<r\}$,
\[
   \mu(2B)\le C_\mu \mu(B),
\]
where $\la B(x,r)=B(x,\la r)$. If $\mu$ is doubling, then $X$ is
complete if and only if it is proper (i.e.\ every closed bounded set
is compact), see Proposition~3.1 in Bj\"orn--Bj\"orn~\cite{BBbook}.

A consequence of the doubling condition is the following lower mass
bound. There exist $C, Q>0$ such that for all $x\in X$, $0<r\le R$
and $y\in B(x,R)$,
\begin{equation}\label{lower-mass-bound}
  \frac{\mu(B(y,r))}{\mu(B(x,R))}
\ge  \frac{1}{C} \left(\frac{r}{R}\right)^Q.
\end{equation}
Indeed,  $Q=\log_2 C_\mu$ and $C=C_\mu^2$ will do, 
see Lemma~3.3 in Bj\"orn--Bj\"orn~\cite{BBbook},
but there may be a better, i.e.~smaller, exponent $Q$.
Note also that \eqref{lower-mass-bound} implies that
$\mu$ is doubling, i.e.~$\mu$ is doubling if and only if
there is an exponent $Q$ such that \eqref{lower-mass-bound} holds.

If $X$ is also connected then there exist constants $C>0$ and $q>0$
such that for all $x\in X$, $0<r\le R$ and $y\in B(x,R)$,
\begin{equation}\label{upper-mass-bound}
    \frac{\mu(B(y,r))}{\mu(B(x,R))} \le  C \left(\frac{r}{R}\right)^{q}.
\end{equation}
Note that we always have $0<q\le Q$ and that any $0<q'<q$ and $Q'>Q$ will
do as well.

We say that $X$ is \emph{Ahlfors $Q_0$-regular} if there is a constant $C$
such that
\[
           \frac{1}{C} r^{Q_0} \le \mu(B(x,r)) \le C r^{Q_0}
\]
for all balls $B(x,r) \subset X$ with $r < 2 \diam X$.
In this case, the best choices for $q$ and $Q$
in \eqref{lower-mass-bound} and \eqref{upper-mass-bound}
are to let
$q=Q=Q_0$.
We emphasize that in this paper we do \emph{not} restrict ourselves to
Ahlfors regular metric spaces.

Garofalo--Marola~\cite{GaMa} introduced the \emph{pointwise dimension}
$q_0(x)$ (called $Q(x)$ in~\cite{GaMa})
as the supremum of all $q>0$ such that
\begin{equation}\label{upper-mass-bound-q}
    \frac{\mu(B(x,r))}{\mu(B(x,R))} \le  C_q \left(\frac{r}{R}\right)^{q}
\end{equation}
for some $C_q>0$ and all $0<r\le R\le\diam X$. Since the
analysis considered in this paper is local, we do not need 
the above inequality for all $R>0$. 

\begin{deff}\label{ptwise-dim}
Given $x \in X$
we consider the \emph{pointwise dimension set} $Q(x)$ of all possible $q>0$
for which there are constants $C_q>0$ and $R_q>0$ such that
\eqref{upper-mass-bound-q} holds
for all $0<r\le R\le R_q$. 
\end{deff}

Observe that $Q(x)$  
is a bounded interval. Indeed,
$Q(x)=(0,q_0]$ or $Q(x)=(0,q_0)$
for some nonnegative  number $q_0$,
as in the following examples.

\begin{example}
Let $X=\R^n$ with the measure $d\mu=w\,dx$, where
\[
w(x)=\max\{1,\log(1/|x|)\}.
\]
Then for all $0<r\le1/e$ we have
\[
\mu(B(0,r))=\al_n r^n \biggl( \frac1n + \log \frac1r \biggr),
\]
where $\al_n$ is the Lebesgue measure of the unit ball in $\R^n$.
It follows that for $x=0$, \eqref{upper-mass-bound-q} holds for all $q<n$
but not for $q=n$, i.e.\ $Q(0)=(0,n)$.
On the other hand, $Q(x)=(0,n]$ for $x\ne0$.
\end{example}

\begin{example}
Let $X= \R^n$ be equipped with the doubling measure
$d\mu(x)=|x|^\al\, dx$  for some fixed $\al>-n$.
 Then $\mu(B(0,r))$ is comparable to $r^{n+\al}$, while for $x\ne0$,
$\mu(B(x,r))$ is comparable to $r^n$ with comparison constant
depending on $|x|$. Thus $Q(0)=(0, n+\alpha]$ whereas
$Q(x)=(0, n]$ if $x\not =0$. It follows that
\eqref{lower-mass-bound} and~\eqref{upper-mass-bound} hold with
$q=\min\{n,n+\al\}$ and $Q=\max\{n,n+\al\}$.
Note that for $\al$ close to $-n$ we have $q$ close to 0.
\end{example}

Note that $q\le q_0  \le Q$, where $q$ and $Q$ are as
in~\eqref{lower-mass-bound} and~\eqref{upper-mass-bound}.
If the measure $\mu$ is Ahlfors $Q_0$-regular, then
$q_0=Q_0$ and $Q(x)=(0,Q_0]$ for all $x$. The pointwise
dimension $Q(x)$ will appear in some of our results in connection
with the capacity and the modulus of curve families.

\begin{deff}
A nonnegative Borel  function $g$ on $X$
is an \emph{upper gradient}
 of an extended real-valued function $u:X\to[-\infty,\infty]$
if for all nonconstant rectifiable curves $\gamma:[0,l_\ga]\to X$,
parameterized by arc length $ds$,
 we have
 \begin{equation}\label{eq:upperGrad}
    |u(\ga(0))-u(\ga(l_\ga))|\le \int_\gamma g\, ds 
 \end{equation}
whenever both $u(\g(0))$ and $u(\g(l_\g))$ are finite, and
$\int_\g g\, ds=\infty $ otherwise.
If $g$ is a nonnegative measurable function (not necessarily Borel) on $X$
and if (\ref{eq:upperGrad}) holds for \p-a.e.~nonconstant
rectifiable curve,
then $g$ is a \emph{\p-weak upper gradient} of~$u$.

By saying that a property holds for \p-a.e.\ rectifiable curve, we
mean that it fails only for a curve family $\Ga$ with zero 
\p-modulus,
i.e\ that there is a Borel function $\rho \in L^p(X)$
such that $\int_\ga \rho \, ds=\infty$ for all $\ga \in \Ga$.
This is consistent with the definition of \p-modulus
in~\eqref{eq-deff-modulus} below.
Since the \p-weak upper gradient $g$ can be modified on a set of measure
zero to obtain a Borel function, it can be shown that
$\int_{\gamma} g\,ds$ is defined (with a
value in $[0,\infty]$) for \p-a.e.\ rectifiable curve.
\end{deff}

Here and throughout the paper we require curves to be nonconstant,
unless otherwise stated explicitly.

Upper gradients were introduced by Heinonen
and Koskela~\cite{HeKo96}, \cite{hk} (where they were called
very weak gradients),
whereas \p-weak upper gradients were first defined in
Koskela--MacManus~\cite{KoMc}.
In \cite{KoMc} it was also shown
that if $g \in L^p(X)$ is a \p-weak upper gradient of $u$,
then one can find a sequence $\{g_j\}_{j=1}^\infty$
of upper gradients of $u$ such that $g_j \to g$ in $L^p(X)$.
If $u$ has an upper gradient in $L^p(X)$, then
it has a \emph{minimal \p-weak upper gradient} $g_u \in L^p(X)$
in the sense
 that for every \p-weak upper gradient $g \in L^p(X)$ of $u$, $g_u \le g$ a.e.,
see Corollary~3.7 in Shanmugalingam~\cite{Sh-harm}
and Theorem~7.16 in  Haj\l asz~\cite{Ha-cont}.

Next we define a version of Sobolev spaces on the metric space
$X$ due to Shanmugalingam~\cite{Sh-rev}.

\begin{deff} \label{def:Np}
        Whenever $u:X \to [-\infty,\infty]$ is a measurable function,
let
\[
        \|u\|_{\Np(X)} = \biggl( \int_X |u|^p \, \dmu
                + \inf_g  \int_X g^p \, \dmu \biggr)^{1/p},
\]
where the infimum is taken over all upper gradients $g$ of $u$.
The \emph{Newtonian space} on $X$ is
\[
        \Np (X) = \{u: \|u\|_{\Np(X)} <\infty \}/\sim,
\]
where two functions $u$ and $v$ are said to be equivalent, denoted $u\sim v$, if
$\|u-v\|_{\Np(X)}=0$.
\end{deff}

 We say that $X$ supports a \emph{\p-Poincar\'e
  inequality} if
there exist constants $\CPI>0$ and $\lambda \ge 1$
such that for all balls $B \subset X$ and
all $u\in \Np (X)$,
\begin{equation} \label{PI-ineq}
        \vint_{B} |u-u_B| \,\dmu
        \le \CPI (\diam B) \biggl( \vint_{\lambda B} g_u^{p} \,\dmu \biggr)^{1/p},
\end{equation}
where $ u_B :=\vint_B u \, d\mu :=\mu(B)^{-1}\int_B u \,\dmu$.
Such Poincar\'e inequalities are sometimes called \emph{weak} since
we allow for $\la >1$.

\medskip

\emph{From now on we assume that the space $X$ is complete
and supports a \p-Poincar\'e inequality, and that the measure $\mu$
is doubling.}

\medskip

A consequence of the above standing assumptions is that
$X$ is \emph{$L$-quasi\-convex}, i.e.\ for every
$x,y \in X$ there is a rectifiable curve with length at most
$L d(x,y)$ connecting $x$ and $y$, where $L$ only depends on the
doubling constant and the constants in the \p-Poincar\'e inequality.
This result is due to Semmes, see Theorem~17.1 in
Cheeger~\cite{Cheeg}, Proposition~4.4 in Haj\l
asz--Koskela~\cite{hak} or Theorem~4.32 in
Bj\"orn--Bj\"orn~\cite{BBbook} for a proof.
Theorem~4.32 in
\cite{BBbook} contains an explicit estimate for $L$.

A direct consequence of the quasiconvexity is that $X$ is locally 
connected.
Many of the results in this paper hold under the weaker assumption
that $X$ is a locally connected proper metric space.
This is true for all the results in
Sections~\ref{ends-prime-e}--\ref{sect-finconn},
except for the results concerning $\Modp$-{\rm(}prime\/{\rm)}ends.
The results in Section~\ref{sect-John}, apart from the $\Modp$-results,
hold under the assumptions that $X$ is a quasiconvex proper metric space
and $\mu$ is doubling.

\begin{remark} \label{rmk-locconn}
Recall that $X$ is \emph{locally\/ \textup{(}path\/\textup{)}connected}
if every neighborhood of a point $x \in X$ contains a (path)connected
neighborhood.
If $X$ is a locally connected proper metric space, then
$X$ is locally pathconnected
by the Mazurkiewicz--Moore--Menger theorem, see
Theorem~1, p.~254, in Kuratowski~\cite{kuratowski2}.
In particular every component of an open set is open and pathconnected, 
see Theorem~2, p.~253,
in \cite{kuratowski2}.
\end{remark}

\section{Carath\'eodory ends and prime ends}
\label{sect-Car}

In this section we give a brief overview of Carath\'eodory's definitions of ends
and prime ends from his 1913 paper \cite{car}
for simply connected planar domains.

 Let $\Om \subset \R^2$ be a simply connected domain.
A \emph{cross-cut} of
 $\Om$ is a 
 Jordan arc in $\Om$ with endpoints on the boundary
 of $\Om$. A sequence $\{c_k\}_{k=1}^\infty$
of cross-cuts is called a \emph{chain}
 if for every $k$, (1) $\overline{c}_k \cap \overline{c}_{k+1}=\emptyset$,
and
 (2) every cross-cut $c_k$ separates $\Om$ into exactly
 two subdomains, one containing $c_{k-1}$ and the other containing
 $c_{k+1}$;
let $D_k$ be the latter subdomain.
The \emph{impression} of the chain is $\bigcap_{k=1}^\infty \itoverline{D}_k$,
which is a nonempty connected compact set.

Carath\'eodory then defined the concept of division of a chain by
another chain
and says that two chains are equivalent if they divide each other.
This leads to an equivalence relation for which the equivalence
classes are called \emph{ends}. The impression is  
independent of the representative chain of an equivalence class.

The ends are naturally partially ordered by division, and 
a \emph{prime end} is an end divisible only by itself,
or in other terms, is minimal with respect to the partial ordering.
The impression of a prime end is always a subset of $\bdy \Om$,
see Theorem~9.2 in Collingwood--Lohwater~\cite{CL}.

Later it was realised that if one imposes an extremal length
condition on the chains then the corresponding ends are
automatically minimal, and they are therefore called prime
ends. When this approach is used there are no ends (other
than prime ends) and no need for weeding out ends which are
not prime ends. This approach leads to the same class of prime ends as in
Carath\'eodory's approach.
 According to our investigations, the first use of
extremal length in connection with prime ends is
due to Schlesinger~\cite{sc}.

One of the main motivations for Carath\'eodory was the (nowadays) well-known
correspondence between points on the unit circle
and prime ends of the image $\Phi(\D)$ of
a conformal mapping $\Phi: \D \to \C$, where $\D$ is the unit disk
in the complex plane.
This is one reason why prime ends
are important tools
in various situations, and
the theory works very well for simply (and finitely) connected
planar domains.
For infinitely connected domains,
and for general domains in higher dimensions,
the theory does not work quite so well,
cf. Kaufmann~\cite{Kau}, Mazurkiewicz~\cite{Maz} and N\"akki~\cite{na}.
However, when restricted to certain domains it
 has proved useful also in higher dimensions,
see N\"akki~\cite{nakki70},
Ohtsuka~\cite{oh} and Karmazin~\cite{Ka3}, \cite{Ka}.

We want to study prime ends in quite general metric spaces,
with a view towards a theory that lends itself to the study of
Dirichlet problems.
We therefore give two approaches.
In the first one
we start by defining ends and then say that the prime ends are the
ends which are minimal (with respect to the partial order).
In the other approach we require the ends initially to satisfy
a \p-modulus condition, and to distinguish these ends from the earlier
ones we call them $\Modp$-ends.
Here we have a choice of $p\ge1$, leading
us to different notions.
Since extremal length is connected with the $2$-modulus in $\R^2$, the
\p-modulus condition seems to be a natural generalization to consider.
In our generality $\Modp$-ends need not be minimal,
and we therefore also introduce $\Modp$-prime ends.

\section{Ends and prime ends}\label{ends-prime-e}

We are now ready to give our definition of  ends and prime ends.

\begin{definition}\label{def-accset}
 A bounded connected set $E\subsetneq\Omega$
is an \emph{acceptable} set if
$\itoverline{E}\cap \partial \Omega$ is nonempty.
\end{definition}

Since an acceptable set $E$ is bounded and connected, we know
that $\itoverline{E}$ is compact and connected. Moreover, $E$ is
infinite, as otherwise we would have $\itoverline{E}=E \subset \Om$.
Therefore, $\itoverline{E}$ is a continuum. Recall that a
\emph{continuum} is a connected compact set containing more than one
point.

\begin{deff}\label{deff-chain}
A sequence $\{E_k\}_{k=1}^\infty$ of acceptable sets is a \emph{chain} if
\begin{enumerate}
\item \label{it-subset}
$E_{k+1}\subset E_k$ for all $k=1,2,\ldots$;
\item \label{pos-dist}
$\dist(\Omega\cap\bdry E_{k+1},\Omega\cap \bdry E_k )>0$
for all $k=1,2,\ldots$;
\item \label{impr}
The \emph{impression}
$\bigcap_{k=1}^\infty \itoverline{E}_k \subset \bdry\Om$.
\end{enumerate}
\end{deff}

\begin{remark} \label{rmk-interior}
As $\{\itoverline{E}_k\}_{k=1}^\infty$ is a decreasing sequence of continua, the
impression is either a point or a continuum.
Moreover,
\ref{it-subset} and \ref{pos-dist} above imply that
$E_{k+1}\subset \interior E_{k}$. In particular, $\interior E_{k} \ne \emptyset$.
\end{remark}

\begin{deff}
We say that a chain $\{E_k\}_{k=1}^\infty$ \emph{divides} the chain
$\{F_k\}_{k=1}^\infty$ if for each $k$ there exists $l$
such that $E_{l}\subset F_k$.
(Here we implicitly assume that $k$ and $l$ are positive integers.
We make similar implicit assumptions throughout the paper to enhance readability.)
Two chains are \emph{equivalent} if they divide each other.
A collection of all mutually equivalent chains is called an \emph{end} and denoted $[E_k]$, where
$\{E_k\}_{k=1}^\infty$ is any of the chains in the equivalence class.
The \emph{impression of} $[E_k]$, denoted $I[E_k]$, is defined as the impression of
any representative chain.

The collection of all ends is called the \emph{end boundary} and is
denoted $\partial_E\Omega$.
\end{deff}

Note that the impression of an end is independent of the choice
of representative chain.
Indeed, if $\{E_k\}_{k=1}^\infty$ divides $\{F_k\}_{k=1}^\infty$ then
$I[E_k] \subset I[F_k]$, and the opposite inclusion holds similarly
if the two chains are equivalent. 
Hence the above definition of impression of an end makes sense.

Note also that if a chain $\{F_k\}_{k=1}^\infty$ divides $\{E_k\}_{k=1}^\infty$,
then it divides every chain
equivalent to $\{E_k\}_{k=1}^\infty$. Furthermore, if $\{F_k\}_{k=1}^\infty$ divides $\{E_k\}_{k=1}^\infty$,
then every chain equivalent to $\{F_k\}_{k=1}^\infty$ also divides $\{E_k\}_{k=1}^\infty$.
Therefore, the relation of division extends in a natural way from chains to ends, giving a partial order on ends.

\begin{remark} \label{rmk-open}
Let $\{E_k\}_{k=1}^\infty$ be a chain.
By Remark~\ref{rmk-interior}, $E_{k+1} \subset \interior E_k$.
Unfortunately, $\interior E_k$ is not necessarily connected. 
However, let $G_k$ be the component of $\interior E_k$ containing $E_{k+1}$. 
Since $X$ is locally connected, we see that $G_k$ is open, 
and hence is an acceptable
set. As $\bdy G_k \subset \bdy \interior E_k \subset \bdy E_k$,
we get that $\{G_k\}_{k=1}^\infty$ is a chain.
Since $E_{k+1} \subset G_k$ for each $k$, 
we see that $\{G_k\}_{k=1}^\infty$ is
divisible by $\{E_k\}_{k=1}^\infty$.
On the other hand, $\{G_k\}_{k=1}^\infty$ clearly divides $\{E_k\}_{k=1}^\infty$,
and thus they are equivalent
and $[G_k]= [E_k]$. 
As a consequence, we could have required that acceptable sets are open
to obtain an equivalent definition of ends. This observation is used in the proofs of
Propositions~\ref{prop1A-A}, \ref{prop-top-pe} and Theorem~\ref{thm-b-John-cor}.
On the other hand, with our definition we have a bit more
freedom when constructing examples.
\end{remark}

Let us next show that
there is a certain redundancy in the collection of ends.

\begin{example}\label{needPrime1}
Let $\Omega=(0,1)^2$ 
be the unit square in the plane and  let
$E_k=(0,1)\times(0,1/k)$ and 
$F_k=\Omega\cap B\bigl(\bigl(\frac{1}{2}, 0\bigr), 1/2k\bigr)$,
$k=1,2,\ldots$.
Then the chain $\{E_k\}_{k=1}^\infty$ is
divisible by the chain $\{F_k\}_{k=1}^\infty$, but $\{F_k\}_{k=1}^\infty$
is not divisible by $\{E_k\}_{k=1}^\infty$.
\end{example}

Such a redundancy might not cause a problem in some applications
(see e.g.\ Miklyukov~\cite{Mik}, where the
analogues of acceptable sets are not even required to be connected), but
since one of our aims is to use the notion of ends to construct
a more general boundary of a domain,
such a redundancy creates a difficulty in using the collection of
all ends as  a boundary.
To overcome this  type of redundancy, we consider the minimal ends
in the following sense.

\begin{deff}\label{prime-end}
An end $[E_k]$ is
a \emph{prime end} if it is not divisible by any other end.
The collection of all prime ends is called the \emph{prime end boundary} and is
denoted $\partial_P\Omega$.
\end{deff}

The following is a natural problem about existence of prime ends. 
We have only been able to solve it in a special case, 
see Proposition~\ref{prop-end-divisible}.
The difficulty in the more general setting is that we do not know 
if a totally ordered uncountable collection
of ends (ordered by division) has a maximal end, 
and hence we are unable to use Zorn's lemma. (Zorn's lemma
was actually first stated by Kuratowski,
see the discussion on p.\ 30 in Bourbaki~\cite{bourbaki}.)

\begin{openprob}
Is it true that every end is divisible by some prime end?
\end{openprob}

The question has a positive answer in the setting of Carath\'eodory ends. 
Indeed,  Carath\'eodory's theorem ensures a correspondence between
points on the unit circle and prime ends of a given simply connected
domain $\Om \subsetneq \R^2$. Let $\{c_n\}_{n=1}^{\infty}$ be (a chain
corresponding to) a Carath\'eodory end in $\Om$ and $\Phi$ be a conformal 
map from
$\Om$ to the unit disk $\D$. 
For each point $x \in \bdy \D$ in the impression of the Carath\'eodory end
 determined by $\{\Phi(c_n)\}_{n=1}^{\infty}$ in the unit disk, we
define the singleton impression prime end for $\D$ by taking a sequence of
circles centered at $x$ intersected with $\D$. The preimage
of this prime end under $\Phi$ gives a Carath\'eodory prime end in $\Om$ dividing 
the end $\{c_n\}_{n=1}^{\infty}$.

\section{Examples and comparison with Carath\'eodory's definition}\label{comp-PEvsCar}

We shall see later that in certain domains,
every boundary point corresponds to at least one prime end.
However, the following example illustrates that in some situations
one may need to also consider ends which are not prime ends.

\begin{example}\label{comb}
(The topologist's comb, see Figure~\ref{fig1}.)
Let $\Om$ be the unit square $(0,1)^2 \subset\R^2$ with the segments
$S_k=\bigl[\tfrac12,1\bigr)\times\{2^{-k}\}$ removed for each $k$.
Let $x_0=\bigl(\tfrac12,0\bigr)$ and let $I=\bigr(\tfrac12,1\bigr]\times\{0\}$
be the set
of inaccessible points, see Definition~\ref{deff-access-pt}.
Then the sets
\[
 E_k=\bigl(\bigl(\tfrac12-2^{-k}, 1\bigr)\times (0, 2^{-k})\bigr)\cap \Om,\quad k=1,2,\ldots,
\]
define an end with the impression $I\cup\{x_0\}$.
However, this is not a prime end, as it is divisible by the chain
$\{B(x_0,2^{-k})\cap \Om \}_{k=1}^\infty$,
which defines a prime end with impression $\{x_0\}$.
Note that there is no prime end with impression containing a point from $I$.
We point out here that by Proposition~\ref{prop-single-Modp}
the prime ends of this domain are also $\Modp$-prime ends 
for $1 \le p\le 2$, in the sense of Definition~\ref{deff-Modp-end}.

Note that with
Carath\'eodory's prime ends the situation is different.
In this case $\{x_0\}$ is not the impression of any Carath\'eodory
prime end, but
instead $I \cup \{x_0\}$ is the impression of a Carath\'eodory prime end.
\end{example}

\begin{figure}[t]
\centering
\subfloat[\empty]
{
\begin{tikzpicture}[line cap=round,line join=round,>=triangle 45, scale =1.55 ]
\clip(-3.23,1.57) rectangle (1.26,5.74); 
\fill[line width=0.4pt,dash pattern=on 4pt off 4pt,color=zzttqq,fill=zzttqq,fill opacity=0.05] (-3,3.6) -- (-3,2.1) -- (0,2.1) -- (0,3.6) -- cycle;
\fill[line width=0.4pt,dash pattern=on 4pt off 4pt,color=zzttqq,fill=zzttqq,fill opacity=0.1] (-2.25,2.85) -- (-2.25,2.1) -- (0,2.1) -- (0,2.86) -- cycle;
\fill[line width=0.4pt,dash pattern=on 4pt off 4pt,color=zzttqq,fill=zzttqq,fill opacity=0.1] (-1.89,2.45) -- (-1.89,2.1) -- (0,2.1) -- (0,2.46) -- cycle;
\draw  [line width=0.6pt] (-3,2.1)-- (-3,5.1);
\draw [line width=0.6pt] (0,5.1)-- (-3,5.1);
\draw [line width=0.6pt] (0,2.1)-- (0,5.1);
\draw [line width=0.6pt] (0,2.1)-- (-3,2.1);
\draw [line width=0.6pt] (-1.5,3.6)-- (0,3.6);
\draw [line width=0.6pt] (-1.5,2.85)-- (0,2.85);
\draw [line width=0.6pt] (-1.5,2.46)-- (0,2.46);
\draw [line width=0.6pt] (-1.5,2.29)-- (0,2.29);
\draw [line width=2pt] (-1.51,2.1)-- (0,2.1);
\draw [line width=0.4pt,dash pattern=on 2pt off 2pt,color=zzttqq] (0,3.6)-- (-3,3.6);
\draw [line width=0.4pt,dash pattern=on 2pt off 2pt,color=zzttqq] (-2.25,2.85)-- (-2.25,2.1);
\draw [line width=0.4pt,dash pattern=on 2pt off 2pt,color=zzttqq] (0,2.85)-- (-2.25,2.85);
\draw [line width=0.4pt,dash pattern=on 2pt off 2pt,color=zzttqq] (-1.89,2.45)-- (-1.89,2.1);
\draw [line width=0.4pt,dash pattern=on 2pt off 2pt,color=zzttqq] (0,2.46)-- (-1.89,2.46);
\draw[color=black] (0.13,3.6) node {${\frac12}$};
\draw[color=black] (0.14,2.85) node {${\frac14}$};
\draw[color=black] (0.15,2.46) node {${\frac18}$};
\draw[color=tttttt] (-1.46,1.89) node {${x_0=(\frac12,0)}$};
\draw[color=black] (-0.25,1.89) node {${I[E_k]}$};
\draw[color=zzttqq] (-2.81,3.43) node {${E_1}$};
\draw[color=zzttqq] (-2.06, 2.68) node {${E_2}$};
\draw[color=black] (-2.25,1.89) node {${\frac14}$};
\draw[color=zzttqq] (-1.71, 2.30) node {${E_3}$};
\end{tikzpicture}
}
\subfloat[\empty]
{
\begin{tikzpicture}[line cap=round,line join=round,>=triangle 45,x=0.59cm,y=0.59cm]
\clip(-2.21,-4.32) rectangle (6.77,5.44);
\fill[color=zzttqq,fill=zzttqq,fill opacity=0.1] (-1,1) -- (-1,-3) -- (5,-3) -- (5,1) -- cycle;
\fill[line width=0.4pt,color=zzttqq,fill=zzttqq,fill opacity=0.1] (-0.48,-1) -- (-0.48,-3) -- (4.52,-3) -- (4.52,-1) -- cycle;
\fill[color=zzttqq,fill=zzttqq,fill opacity=0.1] (-0.24,-2) -- (-0.24,-3) -- (4.26,-3) -- (4.26,-2) -- cycle;
\draw [line width=0.6pt] (-2,5)-- (-2,-3);
\draw [line width=0.6pt] (-2,-3)-- (6,-3);
\draw [line width=0.6pt] (6,-3)-- (6,5);
\draw [line width=0.6pt] (6,5)-- (-2,5);
\draw [line width=0.6pt] (-2,1)-- (4,1);
\draw [line width=0.6pt] (-2,-1)-- (4,-1);
\draw [line width=0.6pt] (-2,-2)-- (4,-2);
\draw [line width=0.6pt] (0,0)-- (6,0);
\draw [line width=0.6pt] (0,-1.5)-- (6,-1.5);
\draw [line width=0.6pt] (0,-2.30)-- (6,-2.30);
\draw [line width=1.3pt] (-2,0)-- (0,0);
\draw [line width=1.3pt] (-2,-1.5)-- (0,-1.5);
\draw [line width=1.3pt] (-2,-2.30)-- (0,-2.30);
\draw [dash pattern=on 2pt off 2pt,color=zzttqq] (-1,1)-- (-1,-3);
\draw [dash pattern=on 2pt off 2pt,color=zzttqq] (-1,-3)-- (5,-3);
\draw [dash pattern=on 2pt off 2pt,color=zzttqq] (5,-3)-- (5,1);
\draw [dash pattern=on 2pt off 2pt,color=zzttqq] (5,1)-- (-1,1);
\draw [line width=0.4pt,dash pattern=on 2pt off 2pt,color=zzttqq] (-0.48,-1)-- (-0.48,-3);
\draw [line width=0.4pt,dash pattern=on 2pt off 2pt,color=zzttqq] (-0.48,-3)-- (4.52,-3);
\draw [line width=0.4pt,dash pattern=on 2pt off 2pt,color=zzttqq] (4.52,-3)-- (4.52,-1);
\draw [line width=0.4pt,dash pattern=on 2pt off 2pt,color=zzttqq] (4.52,-1)-- (-0.48,-1);
\draw [dash pattern=on 2pt off 2pt,color=zzttqq] (-0.24,-2)-- (-0.24,-3);
\draw [dash pattern=on 2pt off 2pt,color=zzttqq] (-0.24,-3)-- (4.26,-3);
\draw [dash pattern=on 2pt off 2pt,color=zzttqq] (4.26,-3)-- (4.26,-2);
\draw [dash pattern=on 2pt off 2pt,color=zzttqq] (4.26,-2)-- (-0.24,-2);
\draw[color=black] (-1.6, -0.3) node {$c_1$};
\draw[color=black] (-1.6, -1.8) node {$c_2$};
\draw[color=black] (-1.6, -2.6) node {$c_3$};
\draw[color=zzttqq] (-0.55, 0.64) node {$E_1$};
\draw[color=zzttqq] (-0.12,-1.25) node {\tiny{$E_2$}};
\draw[color=black] (0,-3.55) node {$\frac14$};
\draw[color=black] (4.1,-3.55) node {$\frac34$};
\draw [line width=3pt,color=black] (0,-3)-- (4,-3);
\draw[color=black] (2,-3.55) node {$I[E_k]$};
\end{tikzpicture}
}
\caption{\label{fig1}%
Examples~\ref{comb} (left) and \ref{ex-double-equilat-comb} (right).}
\end{figure}
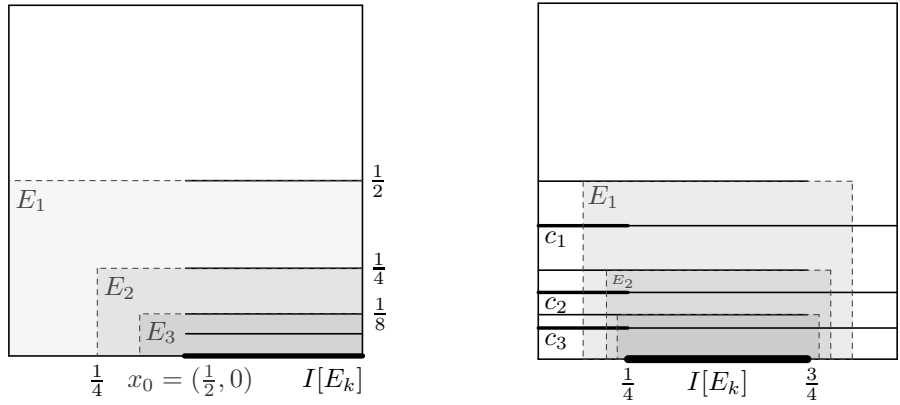
The following example is a major motivation for us.

\begin{example} \label{ex-slit-disc}
Let $\Om$ be the \emph{slit disk} $B((0,0),1) \setm ((-1,0]\times \{0\})
\subset  \R^2$. 
Then for each $x \in [-1,0)\times \{0\}$ there are two prime ends
with the impression $\{x\}$
(one from the upper half-plane and one from the lower
half-plane).
For $x \in \bdy \Om \setm ([-1,0)\times \{0\})$
there is exactly one prime end
with the impression $\{x\}$. These
are the only prime ends associated with $\Om$.
In this example the Carath\'eodory prime ends are the same
as our prime ends.
\end{example}

\begin{prop}\label{prop-car-our-end}
 Let $\Om$ be a bounded simply connected domain in the plane. If
 $\{c_k\}_{k=1}^\infty$ is a Carath\'eodory 
end with impression in the boundary, in particular if it is 
a Carath\'eodory
prime end, then
$[D_k]$ is an end in our sense, 
where $D_k$ are defined as in Section~\ref{sect-Car}.
\end{prop}

This follows directly from the definitions
and the fact that the impression of 
a Carath\'eodory
prime end is always a subset of the boundary, see Section~\ref{sect-Car}.
(Of course, if $\{c_k\}_{k=1}^\infty$ is a Carath\'eodory end with impression
containing some point in $\Om$, then $[D_k]$ is not an end in our sense.)

Observe that not every Carath\'eodory prime end is a prime end
in our sense, see Example~\ref{comb}.
This is due to the fact that we have more ends in some cases,
which again depends on the fact that we only require that an acceptable set
$E$ is connected, not that its relative boundary $\Om \cap \bdy E$ is connected.

That we do not recover Carath\'eodory's prime ends in the simply connected
planar case is a drawback in our theory, and in some situations
our theory is not as useful as 
Carath\'eodory's.
On the other hand, it is well known that Carath\'eodory's theory
is limited to simply and finitely connected planar domains.
We will see in Section~\ref{sect-access} that there is a close connection
between our prime ends and accessibility of boundary points, a connection
lost with Carath\'eodory's definition as shown by Example~\ref{comb}
($x_0$ is an accessible point but there is no Carath\'eodory prime end with
impression equal to $\{x_0\}$).
This connection is crucial for our results in
Sections~\ref{sect-Mazur}--\ref{sect-John}.

We now give one more example showing that
Carath\'eodory's prime ends and our prime ends
need not  be the same in general.

\begin{example}(Double equilateral comb, see Figure~\ref{fig1}.)
\label{ex-double-equilat-comb}
Let $\Om\subset \R^2$ be the domain obtained from the unit
square $(0,1)^2$ by removing the collection of segments
$\bigl(0, \tfrac{3}{4}\bigr]\times \{2^{-n}\}$
and $\bigl[\tfrac14, 1\bigr)\times \{3\cdot 2^{-n-2}\}$
for $n=1,2, \ldots$.  From the point of view of Carath\'eodory's
theory
introduced in the beginning of
Section~\ref{sect-Car} we take the  chain of cross-cuts
\[
 c_{k}=\bigl(0,\tfrac14\bigr)\times \{3\cdot 2^{-k-2} \}
\quad \text{for } k=1,2, \ldots,
\]
which gives a Carath\'eodory prime end with impression $[0,1]\times \{0\}$.

To obtain a prime end in our sense
we define the acceptable sets
\[
 E_{k}=\Om\cap \bigl(\bigl(\tfrac14-2^{-k-2}, \tfrac34+2^{-k-2}\bigr)
   \times \bigl(0, 2^{-k}\bigr)\bigr)
\quad \text{for } k=1,2, \ldots.
\]
Then $[E_k]$ is a prime end with impression
 $I[E_k]=\bigl[\tfrac14,\tfrac34\bigr]\times\{0\}$.
Thus the Carath\'eodory prime end above is not a prime end in our sense, as it is divisible by
$[E_k]$.

  We point out here that $[E_k]$ is also a $\Modp$-prime end for all 
  $p\geq 1$, in the sense of Definition~\ref{deff-Modp-end}.
\end{example}

\begin{remark} \label{rmk-differences}
Our definition of prime ends differs from  earlier definitions.
The boundaries of our
acceptable sets correspond to Cara\-th\'eo\-dory's \emph{cross-cuts},
and our acceptable sets correspond to the components $D_k$ in
Cara\-th\'eo\-dory's definition.
The following are the main differences between our definition and 
Cara\-th\'eo\-dory's definition:
\begin{enumerate}
\item \label{cc-1} Carath\'eodory's  cross-cuts are connected,
while the boundaries of
acceptable sets need not be.
\item\label{cc-2} Cross-cuts break the domain into exactly two components, 
whereas the boundaries of acceptable sets break the underlying domain 
into at least two components.
\end{enumerate}

There are several reasons for these differences.
First,  the topology of a metric space
is more complicated than that of $\R^n$.
(The reader could think of $\R^n$ with a number of holes removed as a
particular example of a metric space under consideration.)
But even in the simply connected planar case 
it would have been more restrictive had we required 
the boundaries of acceptable sets to be connected,
as in Carath\'eodory's definition, see Example~\ref{comb}.
In more complicated geometries the difference is even larger,
as e.g.\ not even boundaries of balls need to be connected.

Our modification of the definition of ends and prime ends is essential
for many of the results in this paper.
In particular,  in Section~\ref{sect-access} we obtain a close connection between
singleton prime ends and accessibility, 
a connection which fails for Carath\'eodory ends, as 
is again demonstrated by Example~\ref{comb}.

In Remark~\ref{importante}
we discuss N\"akki's definition of prime ends on $\R^n$ from~\cite{na}.
N\"akki, following Carath\'eodory,
requires cross-sets to be connected, and so his definition
has the same  drawback as Carath\'eodory's in connection with the
results in this paper, the main difference being again
\ref{cc-1} and \ref{cc-2}.

Among the many definitions of prime ends given by Karmazin~\cite{Ka3}, \cite{Ka}
is probably
the one  closest to ours.
Karmazin  however studies only prime ends on $\R^n$ and 
with different applications of the theory than ours.
\end{remark}

\section{Modulus ends and modulus prime ends}
\label{sect-Modp-ends}

The notion of ends and prime ends discussed in the previous sections does
not take into account the
potential theory associated with the domain. 
Using the following notion of \p-modulus, in
this section we give a
subclass of ends and prime ends
associated with the potential theory. 
Here, $1\le p<\infty$ is fixed.

Let $\Gamma$ be a family of (nonconstant)
rectifiable curves in $X$.
The \emph{\p-modulus} of the family $\Gamma$ is 
\begin{equation}   \label{eq-deff-modulus}
  \Modp(\Gamma):=\inf_{\rho} \int_X\rho^p\, d\mu,
\end{equation}
where the infimum is taken over all nonnegative Borel  functions
$\rho$ on $X$ such that $\int_\gamma\rho\, ds\ge 1$ for every
$\gamma\in\Gamma$ and $ds$ denotes the arc length measure. 
(As usual $\inf \emptyset := \infty$.)
It is straightforward to verify that $\Mod_p$ is an outer measure
on the collection of all rectifiable curves on $X$. In particular, if $\G_1$ and $\G_2$ are two families
of (nonconstant)
rectifiable curves in $X$ such that $\Gamma_1\subset\Gamma_2$,
then $\Mod_p(\G_1)\le\Mod_p(\G_2)$. This
monotonicity property will be useful in this paper.
For more on \p-modulus we refer the interested reader to
Heinonen~\cite{he} and V\"ais\"al\"a~\cite{vaisala}.

The $n$-modulus in $\R^n$ can be used to define and investigate extremal length as well as 
(quasi)conformal and quasiregular maps. Further applications of the \p-modulus include its relation to
capacities and Loewner spaces.
See V\"ais\"al\"a~\cite{vaisala}, 
Heinonen--Koskela~\cite{hk}, Kallunki--Shanmugalingam~\cite{KaSh}, 
Vuorinen~\cite{vu}, and Lemma~\ref{MP}.

For nonempty sets $U \subset X$, $E \subset U$ and  $F \subset U$,
 we let $\G(E,F,U)$ denote the family of all
(nonconstant) rectifiable curves
$\ga:[0,l_\ga]\to U$
such that $\ga(0)\in E$ and $\ga(l_\ga)\in F$.
As in \cite{vaisala}, the
modulus of the curve family $\G(E,F,U)$ is
\[
    \Modp(E,F,U):= \Modp(\G(E,F,U)).
\]

\begin{deff}
A chain $\{E_k\}_{k=1}^\infty$
(see Definition~\ref{deff-chain})
is a \emph{$\Modp$-chain} if
  \begin{equation}\label{deff-Modp-end-cond}
    \lim_{k\to \infty} \Modp(E_k, K, \Omega)=0
  \end{equation}
for every compact set $K\subset \Omega$.
 \label{deff-Modp-end}
An end $[E_k]$ is a \emph{$\Modp$-end} if
there is a $\Modp$-chain representing it.
A $\Modp$-end $[E_k]$ is a \emph{$\Modp$-prime end} if
the only $\Modp$-end dividing it is $[E_k]$ itself.
\end{deff}

For $p>1$ it is equivalent to assume that~\eqref{deff-Modp-end-cond} 
holds for some compact $K\subset \Om$ with positive
capacity, see Lemma~\ref{lem-cpt-equiv}.
We do not know if this is true for $p=1$.

\begin{lem} \label{lem-enum-prime-end}
\begin{enumerate}
\item \label{item-div-chain}
A chain dividing a $\Modp$-chain is also a $\Modp$-chain.
\item \label{item-deff-Modp-end}
Any chain representing a  $\Modp$-end is a  $\Modp$-chain.
\item \label{it-div-Modp}
An end dividing a $\Modp$-end is also a $\Modp$-end.
\item \label{it-prime-end-eq}
A $\Modp$-end is a prime end if and only if it is a $\Modp$-prime end.
\end{enumerate}
\end{lem}

\begin{proof}
\ref{item-div-chain}
Let $\{F_k\}_{k=1}^\infty$ be a chain dividing the
$\Modp$-chain $\{E_k\}_{k=1}^\infty$.
Then for each $k$ there exists $n_k$ such that
$F_{n_k} \subset E_k$.
The monotonicity of $\Modp$ then implies that for every compact $K\subset \Om$,
\[
     \lim_{k \to \infty} \Modp(F_{k},K,\Om)
    = \lim_{k \to \infty} \Modp(F_{n_k},K,\Om)
    \le \lim_{k \to \infty} \Modp(E_k,K,\Om)
     =0,
\]
and hence $\{F_k\}_{k=1}^\infty$ is  a $\Modp$-chain.

\ref{item-deff-Modp-end}
Let $\{E_k\}_{k=1}^\infty$ be a chain representing a $\Modp$-end $[E_k]$.
By definition there is a $\Modp$-chain $\{F_k\}_{k=1}^\infty$
representing the end $[E_k]$.
As $\{E_k\}_{k=1}^\infty$ and $\{F_k\}_{k=1}^\infty$ are
equivalent chains,
it follows from
\ref{item-div-chain} that $\{E_k\}_{k=1}^\infty$ is also a $\Modp$-chain.

\ref{it-div-Modp}
This also follows from~\ref{item-div-chain}.

\ref{it-prime-end-eq}
Let $[E_k]$ be a $\Modp$-end.
If $[E_k]$ is also a prime end, then there is no other end dividing it,
let alone any other $\Modp$-end dividing it.
Thus $[E_k]$ must be a $\Modp$-prime end.
Conversely, if $[E_k]$ is a $\Modp$-prime end and $[F_k]$ is an end dividing
$[E_k]$, then \ref{it-div-Modp} shows that $[F_k]$ is a $\Modp$-end. 
Hence $[F_k]=[E_k]$, and thus $[E_k]$ is a prime end.
\end{proof}

\begin{example}\label{needPrime}
Let $\Om$ be the unit ball in $\R^n$, $n\ge3$, with a radius
removed.
Then for every boundary point $x \in \bdy \Om$ there is a prime end $[F_k^x]$
with $\{x\}$ as its impression, see Corollary~\ref{access1}.

Let $I$ be a closed subsegment of the removed radius and let
\[
E_k=\{x\in\Om: \dist(x,I)<2^{-k}\}.
\]
Then $[E_k]$ is an end
with $I$ as its impression.
This is not a prime end as it is divisible by $[F_k^x]$ for
every $x \in I$.
If $p\le n-1$, then
$\Modp( E_k,K, \Omega)\to 0$ as $k\to\infty$,
and  thus $[E_k]$ is a $\Modp$-end but
not a  $\Modp$-prime end.
\end{example}

Under some conditions, see e.g.\ Section~\ref{John},
all $\Mod_p$-ends are $\Mod_p$-prime ends,
and in this case one does not need to do the further subdivision.

Recall that if $\{E_k\}_{k=1}^\infty$ is a chain,
then $\{\itoverline{E}_k\}_{k=1}^\infty$ is a decreasing sequence of
continua, and so the
impression is either a point or a continuum.
Lemmas~\ref{lem-imp-subset-bdry} and~\ref{lem-cpt-equiv} imply that
if a sequence $\{E_k\}_{k=1}^\infty$ of open acceptable sets satisfies the
conditions \ref{it-subset} and \ref{pos-dist}
 of Definition~\ref{deff-chain} and
$\lim_{k\to\infty}\Mod_p(E_k,K,\Omega)=0$ for some compact set $K\subset\Om$ with
positive \p-capacity, $p>Q-1$, then $\{E_k\}_{k=1}^\infty$
satisfies all the conditions
of Definition~\ref{deff-chain},
and is therefore a chain.
Thus, in view of Remark~\ref{rmk-open}, $\Modp$-ends could be equivalently
defined using only \ref{it-subset} and \ref{pos-dist}
in Definition~\ref{deff-chain} together with \eqref{deff-Modp-end-cond},
when $p>Q-1$.

The notion of $\Modp$-prime end is similar 
to the concept of \p-parabolic prime
ends discussed in Miklyukov~\cite{Mik} and Karmazin~\cite{Ka2}.
The name \p-parabolicity has been used in 
the  literature to
denote spaces where there is not enough room out at infinity
in the sense that the collection of all
curves that start from a fixed ball and leave every compact
subset of the space has \p-modulus zero.
A prime end
is a $\Modp$-prime end if there is insufficient room close to the impression of the
prime end.
In this sense one could think of  $\Modp$-prime ends as
\p-parabolic ends of the domain.
See \cite{CHS}, \cite{Db}, \cite{Gr1}, \cite{GrSa}, 
 \cite{Ho}, \cite{HoK}, \cite{HS},  \cite{Kakut}, \cite{LT}, 
\cite{LW} and \cite{NR}
for some applications of parabolic ends.

The condition $\lim_{k\to \infty} \Modp(E_k, K, \Omega)=0$
depends heavily on $p$. For example, if $p>Q$, where
$Q$ is from~\eqref{lower-mass-bound}, then the collection of all curves
in $X$ passing through $x$ has positive \p-modulus.
For Ahlfors $Q$-regular $X$ this follows from Theorem~4.3 in
Adamowicz--Shanmugalingam~\cite{adsh},
and the proof therein holds also in our case.
Hence in general
there are no $\Modp$-chains with $x$ in their impressions.
However, it can happen that for some $x\in\bdry\Om$,
and every compact $K\subset\Om$ we have
\begin{equation}
\lim_{r\to 0} \Modp (B(x,r)\cap\Om,K,\Om) = 0,
\end{equation}
even if $p>Q$.
This is the case e.g.\  if $\Om\subset\R^n$ (unweighted)
has an outward polynomial cusp at $x$
of degree $m$ and $p\le m+n-1$, see
Durand-Cartagena--Shanmugalingam--Williams~\cite[Example~2.2]{dsw}.

\begin{rem}\label{importante}
Based on the notion of $n$-modulus, 
N\"akki~\cite{na} introduced another variant of prime ends in $\R^n$: 
A connected subset $A$ of a
domain $\Om\subset \R^n$  is called a \emph{cross-set} if (1) it is relatively
closed in $\Om$, (2) $\itoverline{A}\cap \partial \Om\not=\emptyset$,
and (3) $\Om\setminus A$ consists of two components whose boundaries
intersect $\bdry \Om$. A collection $\{A_k\}_{k=1}^\infty$ of cross-sets is called a
\emph{N\"akki chain} if $A_k$ separates $A_{k-1}$ and $A_{k+1}$ (within $\Om$) for
all $k$. A N\"akki chain is a \emph{N\"akki prime chain} if (a$'$)
$\Mod_n(A_{k+1}, A_{k}, \Om)<\infty$, and (b$'$) for any continuum
$K\subset \Om$ we have that 
$\lim_{k\to \infty} \Mod_n(A_k, K, \Om)=0$. The equivalence classes 
with respect to division of N\"akki prime chains define N\"akki prime ends. 

 If $A_k$ is a cross-set, then the component of $\Om\setminus A_k$ containing  
$A_{k+1}$ is an acceptable
set in our sense. Denote this component by $E_k$.

In the domains $\Om\subset \R^n $ considered
by N\"akki (so-called quasiconformally collared domains), condition
(a$'$) is equivalent to $\dist(A_{k+1},A_k)>0$, by Lemma~2.3 in~\cite{na}. 
Similarly, for such domains with $p=n$, 
condition~\ref{pos-dist} of Definition~\ref{deff-chain}
is equivalent to
\begin{equation*}   
 \Mod_p(\Om\cap \partial E_{k+1}, \Om\cap \partial E_{k}, \Om)<\infty.
\end{equation*}
The same holds if $p=Q$ and $\Om$ is Ahlfors $Q$-regular and Loewner,
by~(3.9) in Heinonen--Koskela~\cite{hk}.
(For  definition and discussion of Loewner spaces
see \cite{hk} and Heinonen~\cite{he}.)
However, $\Om$ is in general not Ahlfors $Q$-regular and Loewner,
even if $X$ happens to be.

In the nonconformal case $p \ne Q$, Example~2.7 
in Adamowicz--Shanmugalingam~\cite{adsh}
and Example~\ref{rem_p_n} below show that the corresponding equivalence
can fail even in nice domains, and
in  more general metric spaces there is usually no value
of $p$ for which it is true.
We therefore explicitly require that
chains $\{E_k\}_{k=1}^\infty$ satisfy
\[
\de_k:=\dist(\Omega \cap \partial E_{k+1}, \Omega\cap\partial E_{k})>0.
\]
This modification automatically implies that
$\Modp(\Omega \cap \partial E_{k+1},\Om\cap\partial E_k, \Om)<\infty$, 
since the function 
$\rho=\chi_\Om/\de_k$
is admissible in the definition of
$\Modp(\Omega \cap \partial E_{k+1},\Om\cap\partial E_k, \Om)$.
\end{rem}

\begin{ex}\label{rem_p_n}
Let $\Om=B((0, 0), 2)\subset \R^2$,
$E=[-1,0]\times\{0\}$ and $F=[0,1]\times \{0\}$.
If $1\le p<2$, then $\Modp(E,F,\Om)<\infty$
even though $\dist(E,F)=0$, as we shall next see. In fact $E\cap F\not = \emptyset$.

Let $\Ga_0$ be the family of (nonconstant) rectifiable
curves in $\Om$ passing through the origin.
Since singletons have zero \p-capacity in $\R^2$, we have
$\Modp(\Ga_0)=0$.
We shall therefore in this example only consider curves which
do not pass through the origin.
Let $\ga:[0,l_\ga]\to\Om$  
be such a rectifiable curve connecting $E$ to $F$ in $\Om$.
Joining $\ga$ with its reflection in the real axis makes a closed curve $\gat$
in $\Om$ around the origin.
The residue theorem now yields that
\[
\int_{\gat} \frac{dz}{z} = 2\pi i n,
\]
when $\gat$ is positively oriented and $n \ge 1$ is an integer.
Using symmetry we obtain that
\[
   n\pi = \frac{1}{2}\biggl|\int_{\gat} \frac{dz}{z}\biggr|
 = \biggl|\int_\ga \frac{dz}{z}\biggr|
 \le \int_\ga \frac{|dz|}{|z|}
  = \int_\ga \frac{ds}{|\ga(s)|},
\]
where $ds$ denotes the arc length measure.
It follows that the function $\rho(z)=1/\pi|z|$ is admissible in the
definition of $\Modp(E,F,\Om)$ and hence
\[
\Modp(E,F,\Om) \le \int_\Om \rho^p\,dx\,dy
   = 2\pi^{1-p} \int_0^2 r^{1-p}\,dr = \frac{2^{3-p} \pi^{1-p}}{2-p} <\infty.
\]
\end{ex}

\section{Singleton ends and accessibility}
\label{sect-access}

It is useful to have criteria identifying ends
which are prime ends and $\Modp$-prime ends.
Ends are naturally divided into two classes, those with singleton impressions and
those with larger (continuum) impressions.
The former are simpler to handle,
and the main focus in the later sections will be on singleton ends.
A \emph{singleton end} is an end with a singleton impression.

The classification of ends is a classical topic considered initially
by Cara\-th\'eo\-dory~\cite{car}. See Sections~9.7 and~9.8 in
Collingwood--Lohwater~\cite{CL} for an
extensive classification of prime ends for simply connected planar domains.
For us it is enough to distinguish between singleton ends
and nonsingleton ends.

\begin{prop}  \label{prop-end-single}
If an end has a singleton impression,
then it  is a prime end.
\end{prop}

Note, however, that there are prime ends with nonsingleton impressions,
see Example~\ref{ex-double-equilat-comb}.
Proposition~\ref{prop-end-single} follows directly from the following two lemmas.

\begin{lem} \label{lem-single-char}
Let $[E_k]$ be an end. Then
\[
\diam I[E_k] = \lim_{k\to\infty} \diam E_k.
\]
In particular, $[E_k]$ is a singleton end if and only if
$\diam E_k\to0$ as $k\to\infty$.
\end{lem}

\begin{proof}
Since $E_{k+1}\subset E_k$, it is clear that the limit on the right-hand
side exists.
As $I[E_k] = \bigcap_{k=1}^\infty \itoverline{E}_k$ and
$\diam \itoverline{E}_k = \diam E_k$, one inequality is obvious.

For the converse inequality, let $\eps>0$ be arbitrary and choose
$x_k, y_k \in E_k$ so that
$d(x_k,y_k)\ge (1-\eps) \diam E_k$  for $k=1,2,\ldots$.
By compactness, both $\{x_k\}_{k=1}^\infty$ and $\{y_k\}_{k=1}^\infty$
have converging subsequences $x_{k_j}\to x_0\in I[E_k]$
and $y_{k_j}\to x_0\in I[E_k]$.
It follows that
\begin{align*}
\diam I[E_k] &\ge d(x_0,y_0) = \lim_{j\to\infty} d(x_{k_j}, y_{k_j}) \\
&\ge (1-\eps) \lim_{j\to\infty} \diam E_{k_j}=(1-\eps)\lim_{k\to\infty}\diam E_k.
\end{align*}
Since $\eps>0$ was arbitrary, this completes the proof.
\end{proof}

\begin{lem}\label{lem-identify-prime-end}
Let $[E_k]$ and $[F_k]$ be two ends such that $[E_k]$ divides
$[F_k]$ and assume that $\lim_{k\to\infty}\diam F_k=0$.
Then
$[E_k]= [F_k]$.
\end{lem}

The following observation will be used in the proof
of Lemma~\ref{lem-identify-prime-end} and
also later in the paper.

\begin{remark}  \label{rem-connected-diam}
If a connected set $F\subset\Om$ intersects both $A$
and $\Om\setm A$,
then $F\cap (\Om\cap\bdry A)$ is nonempty.

A direct consequence is that if
$E_k$, $E_{k+1}$ and $F$ are connected subsets of $\Om$ with
$E_{k+1}\subset E_k$, $E_{k+1}\cap F \ne \emptyset$ and
$F\setm E_k \ne \emptyset$, then $F$ meets
both $\Om\cap\bdry E_{k+1}$ and $\Om\cap\bdry E_k$, which implies in turn that
$\dist(\Om\cap\bdry E_{k+1}, \Om\cap\bdry E_{k}) \le \diam F$.
\end{remark}

\begin{proof}[Proof of Lemma~\ref{lem-identify-prime-end}]
Assume that $[E_k] \ne [F_k]$,
i.e.\ that $[F_k]$ does not divide $[E_k]$.
Then it follows that there exists 
$l$ such that for each  
$n$ we can find 
$m_n\geq n$ with $F_{m_n}\setminus E_l\ne\emptyset$.
By the nested property of the chain
$\{F_k\}_{k=1}^\infty$ we get that
$F_k\setminus E_l\ne \emptyset$ for all $k$. From this we infer that
for all $k$ there is a point $y_k\in F_k\setm E_l$.

As $[E_k]$ divides $[F_k]$, for every  $k$ there exists
$j_k\ge l+1$ such that $E_{j_k}\subset F_k$.
Let $x_{k}\in E_{j_k}$ be arbitrary.
Then $x_{k}\in F_k\cap E_{l+1}$ and $y_k\in F_k \setm E_l$.
As $F_k$ is connected, Remark~\ref{rem-connected-diam} implies that
\[
   \dist(\Om\cap\bdry E_{l+1}, \Om\cap\bdry E_{l})
    \le \diam F_k  \to 0\quad \text{as }k \to \infty,
\]
contradicting the fact that $\{E_k\}_{k=1}^\infty$ is a chain.
\end{proof}

For $\Modp$-prime ends we have the following result.

\begin{prop} \label{prop-single-Modp}
If $[E_k]$ is a singleton end with impression
$I[E_k]=\{x\}$ and $1\le p \in Q(x)\ne (0,1],$  
then $[E_k]$ is a $\Modp$-prime end.
\end{prop}

\begin{proof}
By Lemma~\ref{lem-single-char}, $\diam E_k \to 0$ as $k \to \infty$,
and thus $[E_k]$ is a $\Modp$-end by Lemma~\ref{lem-cap-0-mod-0}.
Moreover, Proposition~\ref{prop-end-single} shows that  $[E_k]$ is a prime end,
and hence it is a $\Modp$-prime end
by Lemma~\ref{lem-enum-prime-end}\,\ref{it-prime-end-eq}.
\end{proof}

\begin{deff}    \label{deff-access-pt}
We say that a point $x\in\partial\Om$ is an \emph{accessible} boundary point
if there is a {\rm(}possibly nonrectifiable\/{\rm)} curve
$\gamma:[0,1]\to X$ such that $\gamma(1)=x$ and $\gamma([0,1))\subset\Omega$.
Moreover, if $[E_k]$ is an end and there
is
a curve $\ga$ as above
such that for every $k$ there is $0 < t_k <1$ 
with
\(
\gamma([t_k,1))\subset E_k
\),
then $x\in\bdry\Om$ is \emph{accessible through $[E_k]$}.
\end{deff}

Note that $x\in\bdry\Om$ can be accessible through $[E_k]$ only if
$x\in I[E_k]$.

In the following lemma we use curves to construct prime ends
at accessible points.
A similar construction has been used by Karmazin~\cite{Ka3},~\cite{Ka}.

\begin{lem}  \label{lem-curve-imp-prime-end}
Let $\ga:[0,1]\to X$ be a curve such that $\ga([0,1))\subset\Om$ and
$\ga(1)=x\in\bdry\Om$.
Let also  $\{r_k\}_{k=1}^\infty$ be a strictly decreasing sequence
converging to zero as $k\to\infty$.
Then there exist 
 a sequence $\{\de_k\}_{k=1}^\infty$
of positive numbers smaller than $1$ and a prime end $[F_k]$ such that
$I[F_k]=\{x\}$, $\ga([\de_k,1))\subset F_k$
and $F_k$ is a component of $\Om\cap B(x,r_k)$ for all $k=1,2,\ldots$.
In particular, $x$ is accessible through $[F_k]$.
If  $1\le p\in Q(x) \ne(0,1]$, then this prime end is also
a $\Mod_p$-prime end.
\end{lem}

\begin{proof}
Note first that by the continuity of $\ga$, for each $k=1,2,\ldots,$
there exists $0 < \de_k<1$ such that
\[
\ga([\de_k,1))\subset \Om \cap B(x,r_k).
\]
Let $F_k$ be the component of $\Om \cap B(x,r_k)$ containing
$\ga(\de_k)$.
It follows directly that
$\ga([\de_k,1))\subset F_k$ and hence
$x\in\itoverline{F}_k$, showing that $F_k$ is an acceptable set.
Also, by construction, $F_{k+1}\subset F_k$ for each $k=1,2,\ldots.$

Since $\Om\cap\bdry F_k \subset \bdry B(x,r_k)$, it follows that
for all $k=1,2,\ldots,$
\[
   \dist(\Om\cap\partial F_{k+1},\Om\cap\partial F_{k}) 
   \ge   r_k - r_{k+1} > 0.
\]
Also, as ${F}_k \subset B(x,r_k)$, we have that $I[F_k]=\{x\}$.

Finally, Proposition~\ref{prop-end-single} implies that $[F_k]$ is
a prime end.
Moreover, if $1\le p\in Q(x)\ne(0,1]$, then it is also a $\Mod_p$-prime end by
Proposition~\ref{prop-single-Modp}.
\end{proof}

\begin{cor}\label{access1}
Let $x\in \partial \Om$ be an accessible boundary point.
Then there is a prime end $[F_k]$ with $I[F_k]=\{x\}$.
If $1\le p\in Q(x)\ne(0, 1]$
then this prime end is a $\Mod_p$-prime end.
\end{cor}

\begin{prop}  \label{prop-prime-end-iff-diam-0}
 Let $[E_k]$ be an end and $x\in I[E_k]$ be
accessible through $[E_k]$.
Then the following are equivalent\/\textup{:}
\begin{enumerate}
\item \label{it-prime}
$[E_k]$ is a prime end\/\textup{;}
\item \label{it-x}
$[E_k]$ is a singleton end.
\setcounter{saveenumi}{\value{enumi}}
\end{enumerate}

If\/ $1\le p\in Q(x)\ne(0, 1]$
then  the following
statement is also equivalent to the statements above\/\textup{:}
\begin{enumerate}
\setcounter{enumi}{\value{saveenumi}}
\item \label{it-Modp}
$[E_k]$ is a $\Modp$-prime end.
\end{enumerate}
\end{prop}

The assumption of accessibility is essential in Proposition~\ref{prop-prime-end-iff-diam-0}.
That \ref{it-prime} $\imp$ \ref{it-x} fails without this assumption follows from Example~\ref{ex-double-equilat-comb}.

\begin{proof}
\ref{it-x} $\imp$ \ref{it-prime}
This follows from Proposition~\ref{prop-end-single}.

\ref{it-prime} $\imp$ \ref{it-x} 
As $x$ is accessible through $[E_k]$, there exists
a curve $\ga:[0,1]\to X$ and an increasing sequence of positive numbers
$t_k\to1$ 
as $k\to\infty$, such that $\ga(1)=x$
and for $k=1,2\ldots$, $\ga([t_k,1))\subset E_k$.
Lemma~\ref{lem-curve-imp-prime-end} with e.g.\ $r_k=2^{-k}$
provides us with a prime end $[F_k]$
such that $I[F_k]=\{x\}$ and $\ga([\de_k,1))\subset F_k$ for some $0 <\de_k<1$,
$k=1,2,\ldots.$

We shall show that $[F_k]$ divides $[E_k]$. If not, then there
exists $k$ such that for every $l\ge k+1$ there is a point $x_l\in
F_l\setm E_k$. Since $t_j\to 1$ as $j\to\infty$, for every
$l \ge k+1$ we can find  $j_l\ge l+1$ such that $t_{j_l}\ge \de_l$ and
hence $y_l:=\ga(t_{j_l})\in E_{j_l} \subset E_{k+1}$
and $y_l \in F_l$.
As $x_l\notin E_k$ and
$y_l\in E_{k+1}$,  Remark~\ref{rem-connected-diam} yields
\[
\dist(\Om\cap\bdry E_{k+1},\Om\cap\bdry E_{k}) \le \diam F_l \to0
\quad \text{as }l\to\infty,
\]
which contradicts the definition of chains.
Hence, $[F_k]$ divides $[E_k]$, and as $[E_k]$ is a prime end, it
follows that $[E_k]=[F_k]$, and in particular $I[E_k]=\{x\}$.

Let us finally assume that $1\le p\in Q(x)\ne(0,1]$.

\ref{it-x} $\imp$ \ref{it-Modp}
This follows from Proposition~\ref{prop-single-Modp}.

\ref{it-Modp} $\imp$ \ref{it-prime}
This follows from Lemma~\ref{lem-enum-prime-end}\,\ref{it-prime-end-eq}.
\end{proof}

\begin{prop}\label{prop1A-A}
If $[E_k]$ is an end and $I[E_k]=\{x\}$,
then $x$ is accessible through $[E_k]$.
\end{prop}

\begin{proof}
By Remark~\ref{rmk-open}, we can assume that each $E_k$ is open.
As $X$ is locally connected and $E_k$ is connected, it follows that 
 $E_k$ is pathconnected, see Remark~\ref{rmk-locconn}. 
Choose $x_k\in E_k\setm E_{k+1}$ for $k=1,2,\ldots.$
Since both $x_k$ and $x_{k+1}$ belong to the pathconnected set $E_k$,
there exists a curve  $\ga_k:[1-1/k,1-1/(k+1)] \to E_k$
connecting $x_k$ to $x_{k+1}$.
Let $\ga$ be the union of all these curves.
More precisely, let $\ga:[0,1] \to X$ be given
by $\ga(t)=\ga_k(t)$ if $t \in [1-1/k,1-1/(k+1)]$, $k=1,2,\ldots$,
and $\ga(1)=x$.
Because $\diam E_k \to0$, we know that $\ga$ is continuous
at $1$, and hence $x$ is
accessible along $\ga$ through $[E_k]$.
\end{proof}

The following two corollaries summarize some of the results in
this section.

\begin{cor} \label{cor1A}
A prime end $[E_k]$ is a singleton end if and only if its impression $I[E_k]$
contains a point which is accessible through $[E_k]$.
\end{cor}

\begin{proof}
This follows directly from Propositions~\ref{prop-prime-end-iff-diam-0}
and~\ref{prop1A-A}.
\end{proof}

\begin{cor}  \label{cor-access-equiv-end}
Let $x \in \bdy\Om$.
Then the following are equivalent\/\textup{:}
\begin{enumerate}
\item \label{i2-acc}
$x$ is accessible\textup{;}
\item \label{i2-end}
there is an end $[E_k]$ with $I[E_k]=\{x\}$\textup{;}
\item \label{i2-prime-end}
\setcounter{saveenumi}{\value{enumi}}
there is a prime end $[E_k]$ with $I[E_k]=\{x\}$.
\end{enumerate}

If $1\le p\in Q(x)\ne(0,1]$, then also the following 
statements are equivalent to the statements above\/\textup{:}
\begin{enumerate}
\setcounter{enumi}{\value{saveenumi}}
\item \label{i2-Modp}
there is a $\Modp$-end $[E_k]$ with $I[E_k]=\{x\}$\textup{;}
\item \label{i2-Modp-prime}
there is a $\Modp$-prime end $[E_k]$ with $I[E_k]=\{x\}$.
\end{enumerate}
\end{cor}

\begin{proof}
\ref{i2-acc} \imp \ref{i2-prime-end}
This follows from Corollary~\ref{access1}.

\ref{i2-prime-end} \imp \ref{i2-end}
This is trivial.

\ref{i2-end} \imp \ref{i2-acc}
This follows from Proposition~\ref{prop1A-A}.

Let us finally assume that $1\le p\in Q(x)\ne(0,1]$.

\ref{i2-end} \imp \ref{i2-Modp-prime}
This follows from Proposition~\ref{prop-single-Modp}

\ref{i2-Modp-prime} \imp \ref{i2-Modp} \imp \ref{i2-end}
These implications are trivial.
\end{proof}

\section{The topology on ends and prime ends}
\label{sect-top}

We would like to find
homeomorphisms between the prime end boundary
 $\bdy_P \Omega$ and other boundaries.
To do so we need a topology on $\bdy_P \Om$,  and in fact
on the prime end closure $\clOmP:=\Om \cup \bdy_P \Om$.
We will introduce a topology on the larger set $\Om \cup \bdy_E \Om$,
where $\bdy_E \Om$ is the end boundary.
It then naturally induces a topology on
$\clOmP$ and also on the boundaries connected
with $\Modp$-prime ends.

\begin{definition}
\label{seqtoprime}
We say that a \emph{sequence of points} $\{x_n\}_{n=1}^\infty$ in $\Omega$
\emph{converges to the end} $[E_k]$,
and write $ x_n \to [E_k]$ as $n \to \infty$,
if for all $k$ there exists $n_k$
such that $x_n\in E_k$ whenever $n \ge n_k$.
\end{definition}

If $x_n \to [E_k]$ as $n\to\infty$, and $[E_k]$ divides $[F_k]$,
then $x_n$ also converges to $[F_k]$.
Thus the limit of a sequence need not be unique,
and we therefore avoid writing $\lim_{n \to \infty} x_n$.
It is less obvious that this problem remains even if we restrict our
attention to prime ends, see Example~\ref{ex-Jana-two-limits} below.

\begin{definition}\label{conv}
A \emph{sequence of ends} $\{[E_k^n]\}_{n=1}^\infty$ \emph{converges
to the end} $[E_k^\infty]$ if
for every $k$ there is $n_{k}$
such that for each
$n\ge n_{k}$ there exists $l_{n, k}$ with
$E_{l_{n, k}}^n \subset E_{k}^\infty$.
\end{definition}

Note that the integers $n_k$ and $l_{n,k}$ in Definitions~\ref{seqtoprime}
and~\ref{conv}
depend on the
representative chain  of the corresponding ends.
However, both  notions of convergence are
independent of the choice of representative chain.

\begin{deff}  \label{def-closed}
Convergence of points and ends defines a topology on
$\Om \cup \bdy_E \Om$. In this topology,
a collection  $C \subset \Om \cup \bdy_E \Om$ of points and ends is \emph{closed}
if whenever (a point or an  end) $y\in \Om \cup \bdy_E \Om$
is a limit of a sequence in $C$,
then $y \in C$.
\end{deff}

In this topology, a sequence $\{x_n\}_{n=1}^\infty$ of points in $\Om$
converges to a point $y \in \Om$ as given by the metric
topology, and no sequence of ends converges to a point in $\Om$.

\begin{prop}
The topology defined above is indeed a topology on $\Om \cup \bdy_E \Om$.
\end{prop}

\begin{proof}
(1) The empty set  and $\Om \cup\bdy_E \Om$ are clearly closed.

(2) Let $C_1$ and $C_2$ be closed subsets of $\Om \cup \bdy_E \Om$.
Assume that
$\{y_n\}_{n=1}^\infty$
is a sequence in $C_1\cup C_2$
such that $y_n\to y_{\infty}$. Then there is a subsequence
$\{y_{n_k}\}_{k=1}^{\infty}$ lying entirely either in $C_1$ or else in $C_2$.
Since a subsequence of a convergent
sequence converges to the same limit,
it follows that $y_\infty\in C_1$ or $y_\infty\in C_2$.
Hence  $y_\infty\in C_1\cup C_2$.
By induction, for any positive integer $N$ we have that
$\bigcup_{n=1}^N C_n$ is closed whenever $C_1,\ldots, C_N$ are closed.

(3) Now, let  $\{C_i\}_{i\in \mathcal{I}}$ be a collection of
closed
subsets of $\Om \cup \bdy_E \Om$.
Consider a sequence $\{y_n\}_{n=1}^\infty\subset \bigcap_{i\in\mathcal{I}}C_i$.
If $y_n\to y_\infty$ as $n \to \infty$, then $y_\infty\in C_i$ for all $i\in \mathcal{I}$, since the
$C_i$ are closed.
Therefore, $y_\infty  \in \bigcap_{i\in \mathcal{I}}C_i$ and the
intersection is closed.
\end{proof}

In the rest of this section we discuss the
topology on $\Om \cup \bdy_E \Om$
and the induced topology on the prime end closure $\clOmP$
to gain a better understanding for them.
Given an open set $G \subset \Om$, let $G^\End$
be the union of $G$ and all the ends $[E_k]$ such that
$E_k \subset G$ for some $k$. The letter E in the superscript
stands for ``ends''.

\begin{prop}\label{prop-top-pe}
The collection of sets
\[
    \mathcal{C}:=\{G,G^\End: G \subset \Om \text{ is open}\}
\]
forms a basis for our topology.
\end{prop}

\begin{proof}
We first prove that $G^\End$ is open in our topology
if $G \subset \Om$ is open.
For this, we show that $F=(\Om \cup \bdy_E\Om) \setm G^\End$ is
closed in the sense of Definition~\ref{def-closed}.
Since $F \cap \Om=\Om \setm G$ is closed in $\Om$,
if a  sequence $\{x_n\}_{n=1}^\infty \subset F \cap \Om$ converges to
$x_\infty\in\Om$, then $x_\infty\in F$.
Next, if $\{x_n\}_{n=1}^\infty \subset F \cap \Om$ converges to an end  $[E_k]$,
then for each $k$ we can find $n$ such that $x_n \in E_k$.
In particular, $E_k \cap (F \cap \Om) \ne \emptyset$,
or equivalently $E_k \not\subset G$.
As this holds for all $k$, we see that $[E_k] \notin G^\End$,
i.e.\ $[E_k] \in F$.

Similarly, if $\{[E_k^n]\}_{n=1}^\infty \subset F \cap \bdy_E \Om$
converges to an end  $[E_k^\infty]$, then for
every $k$ there are $n$ and $l$ such that $E_l^n \subset E_k^\infty$.
Since $[E_k^n] \notin G^\End$ we must have that $E_l^n \not \subset G$,
and in particular $E_k^\infty \not\subset G$, showing that
$[E_k^\infty] \in F$.

Thus all the
sets in $\mathcal{C}$ are open.
Since $\mathcal{C}$ contains all the open subsets of $\Om$ it is enough
to show that if $H \subset \Om \cup \bdy_E \Om$ is open in our topology
and $[E_k] \in H$, then there is an open set $G \subset \Om$
such that $[E_k] \in G^\End \subset H$.
By Remark~\ref{rmk-open} we may assume that the sets $E_k$ are open.
We shall show that $E_k^\End \subset H$ for some $k$,
i.e.\ $G=E_k$ will do.

Assume that this is false.
Then by passing to a subsequence if necessary, either
\begin{enumerate}
\item \label{top-aa}
there are points $x_n \in (E_n^\End \cap \Om) \setm H = E_n \setm H$
for all $n$; or
\item \label{top-bb}
there are ends $[F_k^n] \in (E_n^\End \cap \bdy_E \Om) \setm H $
for all $n$.
\end{enumerate}

In case~\ref{top-aa}, $x_n \to [E_k]$ as $n \to \infty$,
by definition.
As $[E_k]\in H$, this contradicts the fact that $H$ is open.
Also in case~\ref{top-bb} we see that
$[F_k^n] \to [E_k]\in H$ as $n \to \infty$,
by definition, contradicting the openness of $H$ again.

We thus conclude that indeed $E_k^\End \subset H$ for some $k$,
and thus $\mathcal{C}$ is a basis for our topology.
\end{proof}

One may ask if the collection
\begin{equation} \label{eq-G1G2}
   \{G_1 \cup G_2^E : G_1,G_2 \subset \Om \text{ are open}\}
\end{equation}
may contain all open sets in our topology.
Example~\ref{ex-Ge} below shows that this is not true.

When restricting to prime ends it directly follows from
Proposition~\ref{prop-top-pe} that the collection
\[
   \{G, G^P: G \subset \Om \text{ is open}\},
\]
where $G^P=G^\End \cap \clOmP$, forms a basis
for our topology on $\clOmP$.
This time it follows from Example~\ref{ex-Gp} 
that not all open sets can be written in the form
$G_1 \cup G_2^P$.

\begin{example} \label{ex-Ge}
Let $\Om=(0,3)^2 \subset\R^2$,
$I=[0,3] \times \{0\}$ and  
$E_k=(0,3)\times(0,1/k)$. 
Then $[E_k]$ is an end with impression $I$.
Furthermore, let $G_1=(0,2)^2$ and $G_2=(1,3) \times (0,2)$.
If
\begin{equation} \label{eq-G}
         G_1^E \cup G_2^E \subset G_3^E \cup G_4
\end{equation}
 for some open sets
$G_3,G_4\subset \Om$, then $[E_k] \in G_3^E$.
But $[E_k] \notin G_1^E \cup G_2^E$.
Thus we cannot have equality in \eqref{eq-G} and hence
the collection \eqref{eq-G1G2}
does not contain all open sets.
\end{example}

\begin{example} \label{ex-Gp}
Let $\Omt \subset \R^2$ be the double equilateral comb from
Example~\ref{ex-double-equilat-comb}, and let $\Om=\Omt \times (0,1)$.
Note that $\Om\subset\R^3$ is simply connected and homeomorphic to a ball.
For $a,b \in [0,1]$, let $I^{a,b}$ be the closed line segment with end points
$\bigl(\tfrac{1}{4},0,a\bigr)$ and
$\bigl(\tfrac{3}{4},0,b\bigr)$.
Let further
\[
    E^{a,b}_k=\{x \in \Om : \dist(x,I^{a,b}) < 1/k\}.
\]
Then $[E_k^{a,b}]$ is a prime end with impression $I^{a,b}$.
Let next
\[
   G_1=\Omt \times \bigl(0,\tfrac{2}{3}\bigr)
   \quad \text{and} \quad
   G_2=\Omt \times \bigl(\tfrac{1}{3},1\bigr).
\]
If
\begin{equation} \label{eq-GP}
         G_1^P \cup G_2^P \subset G_3^P \cup G_4
\end{equation}
 for some open sets
$G_3, G_4 \subset\Om$, then $[E_k^{0,1}] \in G_3^P$.
But $[E_k^{0,1}] \notin G_1^P \cup G_2^P$. 
Thus we cannot have equality in \eqref{eq-GP} and hence
the collection
\[
   \{G_1 \cup G_2^P : G_1,G_2 \subset \Om \text{ are open}\}
\]
does not contain all open sets in $\clOmP$.

Observe also that
$\{[E_k^{a,1-a}] : 0 \le a \le 1\}$ is an uncountable collection
of prime ends such that no pair of them can be separated
as in the T2 separation condition below, since the sequence
$\bigl\{\bigl(\tfrac12,\tfrac1{2n+1},\tfrac12\bigr)\bigr\}_{n=1}^\infty$
converges to all of them.
\end{example}

Recall that a topological space $Y$ satisfies the \emph{T1 separation condition} if any
two distinct points can be separated,
i.e.\  if each point lies in an open set which
does not contain the other point. (An equivalent way to formulate the T1 separation condition is to require
that every singleton set is closed.)
If the two open sets can be chosen to be disjoint, then
$Y$ satisfies the \emph{T2 separation condition}.
A topological space is \emph{Hausdorff} if it satisfies the T2 separation condition.

If $[E_k]$ and $[F_k]$ are two distinct ends such that $[E_k]$ divides $[F_k]$, then
any neighborhood of $[F_k]$ contains $[E_k]$, and thus
the topology on $\Om \cup \bdy_E\Om$ does not satisfy 
the T1 separation condition.
If we however restrict ourselves to prime ends, i.e.\
to $\clOmP$, then the T1 separation condition is satisfied.

\begin{prop}
The topology on $\clOmP$ satisfies the T1 separation condition.
\end{prop}

\begin{proof}
If $x \in \Om$, then $\{x\}$ is closed in our topology.
Thus to verify the T1 separation condition it suffices
 to show that a singleton set
$\{[E_k]\}$ is closed for any prime end $[E_k]$.
We thus need to consider the sequence $\{[E_k^n]\}_{n=1}^\infty$,
with $[E_k^n]=[E_k]$ for all $n$.
Assume that $\{[E_k^n]\}_{n=1}^{\infty}$ converges to a prime end $[E_k^{\infty}]$.
As the sequence is constant it is not hard to see that
$[E_k]$ must divide $[E_k^{\infty}]$.
Since $[E_k^{\infty}]$ is a prime end, we thus must have $[E_k]=[E_k^{\infty}]$.
Hence the set $\{[E_k]\}$ is closed.
\end{proof}

The topology obtained on $\clOmP$ does not need to satisfy the T2 separation condition,
and can thus be nonmetrizable,
as shown by Example~\ref{ex-Gp} and
the following example.
In Corollary~\ref{metric-lemma-5-12} we will show that this topology is metrizable 
if $\Om$ is finitely connected at the boundary.

\begin{figure}[t]
\centerline{
\begin{tikzpicture}[line cap=round,line join=round,>=triangle 45,x=1.0cm,y=1.0cm]
\clip(-3.3,0.54) rectangle (5.98,6.3);
\fill[color=zzttqq,fill=zzttqq,fill opacity=0.1] (-3,3) -- (-3,1) -- (3,1) -- (3,3) -- cycle;
\fill[color=zzttqq,fill=zzttqq,fill opacity=0.1] (-3,2) -- (-3,1) -- (2,1) -- (2,2) -- cycle;
\fill[color=zzttqq,fill=zzttqq,fill opacity=0.1] (-3,1.52) -- (-3,1) -- (1.5,1) -- (1.5,1.5) -- cycle;
\draw [line width=0.6pt] (-3,5)-- (-3,1);
\draw [line width=0.6pt] (-3,1)-- (5,1);
\draw [line width=0.6pt] (5,1)-- (5,5);
\draw [line width=0.6pt] (5,5)-- (-3,5);
\draw [line width=0.6pt] (-3,3)-- (-1,3);
\draw [line width=0.6pt] (3,3)-- (5,3);
\draw [line width=0.6pt] (-3,2)-- (0,2);
\draw [line width=0.6pt] (2,2)-- (5,2);
\draw [line width=0.6pt] (-3,1.5)-- (0.52,1.5);
\draw [line width=0.6pt] (5,1.5)-- (1.46,1.5);
\draw [line width=0.6pt] (-1,4)-- (3,4);
\draw [line width=0.6pt] (-1.98,2.5)-- (4,2.5);
\draw [line width=0.6pt] (-2.48,1.76)-- (4.56,1.76);
\draw [line width=0.6pt] (-2.74,1.33)-- (4.78,1.33);
\draw [dash pattern=on 2pt off 2pt,color=zzttqq] (-3,3)-- (-3,1);
\draw [dash pattern=on 2pt off 2pt,color=zzttqq] (-3,1)-- (3,1);
\draw [dash pattern=on 2pt off 2pt,color=zzttqq] (3,1)-- (3,3);
\draw [dash pattern=on 2pt off 2pt,color=zzttqq] (3,3)-- (-3,3);
\draw [dash pattern=on 2pt off 2pt,color=zzttqq] (-3,2)-- (-3,1);
\draw [dash pattern=on 2pt off 2pt,color=zzttqq] (-3,1)-- (2,1);
\draw [dash pattern=on 2pt off 2pt,color=zzttqq] (2,1)-- (2,2);
\draw [dash pattern=on 2pt off 2pt,color=zzttqq] (2,2)-- (-3,2);
\draw [dash pattern=on 2pt off 2pt,color=zzttqq] (-3,1)-- (1.5,1);
\draw [dash pattern=on 2pt off 2pt,color=zzttqq] (1.5,1)-- (1.5,1.5);
\draw [dash pattern=on 2pt off 2pt,color=zzttqq] (1.5,1.5)-- (-3,1.5);
\draw [line width=2.4pt,color=black] (-3,1)-- (1,1);
\draw [line width=1.6pt,dash pattern=on 1pt off 2pt on 5pt off 4pt,color=black] (1,1)-- (5,1);
\draw[color=zzttqq] (-2.72,2.78) node {$E_1$};
\draw[color=zzttqq] (-2.72,1.79) node {$E_2$};
\draw[color=black] (-0.7,0.7) node {$I[E_k]$};
\draw[color=black] (3.28,0.7) node {$I[F_k]$};
\end{tikzpicture}
}
\caption{\label{fig2}%
Example~\ref{ex-Jana-two-limits}.}
\end{figure}
\begin{example} \label{ex-Jana-two-limits}
(See Figure~\ref{fig2}.)
Let $\Om\subset\R^2$ be obtained from the rectangle $(-1,1)\times(0,1)$
by removing the segments
\[
   (-1,-2^{-k}] \times\{2^{-k}\},\quad [2^{-k},1) \times \{2^{-k}\}
 \quad \text{and} \quad
[-1+2^{-k},1-2^{-k}] \times \{3\cdot 2^{-k-1}\},
\]
$k=1,2,\ldots$.
Then the sets
\[
E_k = \Om \cap ((-1,2^{-k}) \times (0,2^{-k}))
\quad \text{and} \quad
F_k = \Om \cap ((-2^{-k},1) \times (0,2^{-k}))
\]
define two prime ends with impressions
\[
I[E_k]=[-1,0]\times\{0\} \quad \text{and} \quad
I[F_k]=[0,1]\times\{0\}.
\]
These prime ends are clearly different but the sequence
$\{(0,2^{-n})\}_{n=1}^\infty$ converges to both of them.
It follows that any neighborhood of any of these two prime ends
contains all but a finite number of points from this sequence.
Hence these two prime ends do not have disjoint neighborhoods,
or in other terms the T2 separation condition fails.
It is easy to modify $\Om$ so that the impressions 
of the two prime ends
have a common interval and not just a common point. 

The domain  $\Om$ above is not simply connected.
To get a simply connected domain 
consider
\[
\Om' = (\Om \times(0,1]) \cup ((-1,1)\times(0,1)\times (1,2))
\]
or Example~\ref{ex-Gp}.
\end{example}

Definition~\ref{conv} implies that
if $[E_k^n] \to [E_k^\infty]$ as $n \to \infty$, then
there are sequences $\{x_{i}^n\}_{i=1}^\infty$, $n=1,2,\ldots$, and
$\{x_i^\infty\}_{i=1}^\infty$ in $\Omega$, which converge to
$[E_{k}^n]$ and $[E_{k}^\infty]$ respectively as $i\to\infty$,
and satisfy
\[
\lim_{n \to \infty} \limsup_{i \to \infty} d(x_i^n, x_i^\infty)=0.
\]

 However, even with the additional assumption that the diameters 
 of $[E_{k}^n]$ and $[E_{k}^\infty]$ converge to $0$, this sequential criterion
does not imply convergence of prime ends. Consider e.g.\
the slit disk (Example~\ref{ex-slit-disc})
and let $x_i^n$ converge to a point on the slit from one side and
$x_i^\infty$ from the other side.
Instead, one can  use the Mazurkiewicz distance associated with the connectedness properties of the domain.

\begin{deff} \label{deff-dM}
We define the \emph{Mazurkiewicz distance} $\dM$ on $\Om$ by
\[
     \dM(x,y) =\inf \diam E,
\]
where the infimum is over all connected sets $E \subset \Om$
containing $x,y \in \Om$.
\end{deff}

Clearly,  $d_M$ is a metric on $\Omega$.
When $x,y\in\Omega$, we have $d_M(x,y)\ge d(x,y)$.

\begin{lemma}   \label {lem-dM-sequential}
Let $\{[E_k^n]\}_{n=1}^\infty$, $n=1,2,\ldots$, and $[E_k^\infty]$ be ends.
Then the following are equivalent\/\textup{:}
\begin{enumerate}
\item \label{lem-dM-a}
The sequence of ends $\{[E_k^n]\}_{n=1}^\infty$ converges to the end $[E_k^\infty]$ and
$[E_k^\infty]$ is a singleton end.

\item \label{lem-dM-b}
Whenever
$\{x_{i}^n\}_{i=1}^\infty$, $n=1,2,\ldots$, and $\{x_i^\infty\}_{i=1}^\infty$
are sequences in $\Omega$, which converge to
$[E_{k}^n]$ and $[E_{k}^\infty]$, respectively, as $i\to\infty$,
we must have
\begin{equation}  \label{eq-dM-to0}
    \lim_{n \to \infty} \limsup_{i \to \infty} d_M(x_i^n, x_i^\infty)=0.
\end{equation}
\end{enumerate}
\end{lemma}

Recall that by Lemma~\ref{lem-single-char} an end $[E_k]$ has a singleton impression if and only if
$\lim_{k \to \infty}\diam E_k =0$.

\begin{proof}
\ref{lem-dM-a} $\imp$ \ref{lem-dM-b}
For all $k$ there exists $n_k$ such that for each $n\ge n_k$ we can
find $l_{n,k}$ such that
$E_{l_{n,k}}^n\subset E_k^\infty$.
Let  $\{x_{i}^n\}_{i=1}^\infty$ and $\{x_i^\infty\}_{i=1}^\infty$
converge to $[E_k^n]$, $n=1,2,\ldots$, and $[E_k^\infty]$, respectively.
Fix $k$, $n$ and $l_{n,k}$ as above for a moment.
Then there exists $m_{n,k}$ such
that for all $i\geq m_{n,k}$ we have $x_i^n\in E_{l_{n,k}}^n\subset E_k^\infty$.
Similarly, there is
$m_k$ such that for $i\geq m_k$ we have that $x_i^\infty\in E_k^\infty$.
Hence for $i\ge\max\{m_{n, k}, m_{k}\}$ we have that
\[
d_M(x_i^n, x_i^\infty)\leq \diam E_k^\infty.
\]
Since $\diam E_k^\infty \to 0$ we conclude that
\[
0\leq\limsup_{n \to \infty} \limsup_{i \to \infty} d_M(x_i^n,
x_i^\infty) \leq \diam E_k^\infty\to0.
\]

\ref{lem-dM-b} $\imp$ \ref{lem-dM-a}
Assume first that
$\{[E_k^n]\}_{n=1}^\infty$ does not converge to the
end $[E_k^\infty]$. Then there exists $k_0$ such that for all $n$
there is $m_n\geq n$ with the property that $E_i^{m_n}\setminus
E_{k_0}^\infty\ne\emptyset$ for all $i=1,2,\ldots$.
For each $n=1,2,\ldots$ we define a sequence
$\{x_i^n\}_{i=1}^\infty$ by choosing $x_i^n\in E_i^{n}\setminus E_{k_0}^\infty$
if this set is nonempty and $x_i^n\in E_i^{n}$ otherwise. By construction,
$x_i^n\to [E_{k}^n]$, as $i\to \infty$.
Let also $x_i^\infty\in E_i^\infty$ be arbitrary, $i=1,2,\ldots$.
Then $x_i^{m_n}\notin E_{k_0}^\infty$ and $x_i^\infty\in E_{k_0+1}^\infty$
whenever $i>k_0$.
Remark~\ref{rem-connected-diam} yields
\[
 d_M(x_i^{m_n}, x_i^\infty)\geq \dist (\Om\cap\bdry E_{k_0+1}^\infty, \Om\cap\bdry
 E_{k_0}^\infty)>0.
\]
Letting $i\to\infty$ and then $n\to\infty$ implies that
\[
    \limsup_{n \to \infty} \limsup_{i \to \infty} d_M(x_i^n, x_i^\infty)>0,
\]
i.e.\ \eqref{eq-dM-to0} fails. Thus $[E_k^n]\to [E_k^\infty]$.

To see that $[E_{k}^\infty]$ is a singleton end, choose
$x_k^\infty, y_k^\infty \in E_{k}^\infty$ so that
$d(x_k^\infty, y_k^\infty)\geq \frac12 \diam E_{k}^\infty$, $k=1,2,\ldots$.
Let $z_k^n\in E_{k}^n$, $n,k=1,2,\ldots$, be arbitrary.
Observe that $x_j^\infty\to [E_{k}^\infty]$, $y_j^\infty\to [E_{k}^\infty]$
and $z_j^n\to [E_{k}^n]$ for $n=1,2,\ldots$, as $j\to\infty$.
We then have by the triangle inequality that
\[
\diam E_k^\infty \le 2 d(x_k^\infty,y_k^\infty)
\le 2d_M(x_k^\infty,y_k^\infty)
\le 2 ( d_M(x_k^\infty,z_k^n) + d_M(z_k^n,y_k^\infty) ).
\]
Letting $k\to\infty$ and then $n\to\infty$ together with \eqref{eq-dM-to0}
(used twice) completes the proof.
\end{proof}

\section{Prime ends and the Mazurkiewicz boundary}
\label{sect-Mazur}

We now focus on describing embeddings and homeomorphisms between the
prime end boundary and two other boundaries, the topological
boundary and the Mazur\-kiewicz boundary. 
Our investigations are motivated by the fact
that such mappings allow us to discuss the correspondence between prime ends 
and their impressions, with a view towards boundary value problems.

In Bj\"orn--Bj\"orn--Shanmugalingam~\cite{BBSdir} the Dirichlet problem
for \p-harmonic functions, with boundary data defined on the
Mazurkiewicz boundary, is
studied in domains which are finitely connected at the boundary
(see Definition~\ref{def-finite-conn} 
for the notion of finite connectedness at the boundary).
By Theorem~\ref{thm-fin-con-homeo} this is equivalent to studying
the Dirichlet problem with respect to the prime end boundary for such domains.
We refer to \cite{BBSdir} for further details on the Dirichlet
problem, but this is another important motivation for this and the next section.

We saw in Section~\ref{sect-access} that
accessibility of a
boundary point determines whether there is a
prime end with a singleton impression
at this point. Furthermore, Lemma~\ref{lem-dM-sequential}
tells us that there is a strong link between the Mazurkiewicz
distance on $\Om$ and the topology
of $\partial_P\Om$.
Motivated by these, we consider the boundary of
$\Om$ with respect to the Mazurkiewicz
distance in this section.
Recall that the Mazurkiewicz distance was introduced in
Definition~\ref{deff-dM}. 

\begin{remark}\label{preserveLength}
Because $X$ is locally connected, $d_M$ and $d$ define the same topology on $\Om$.
\end{remark}

The completion of the metric space $(\Omega, d_M)$ is denoted $\clOmm$,
and $d_M$ extends in the standard way to $\clOmm$:
For $\dM$-Cauchy
sequences $\{x_n\}_{n=1}^\infty, \{y_n\}_{n=1}^\infty\in \Om$ we define the
equivalence relation
\[
 \{x_n\}_{n=1}^\infty \sim \{y_n\}_{n=1}^\infty
 \quad \text{if} \quad \lim_{n\to \infty} d_M(x_n, y_n) =0.
\]
Note that every Cauchy sequence is trivially equivalent to any
of its subsequences.

The collection of all equivalence classes of $\dM$-Cauchy sequences can be
formally considered to be $\clOmm$, but
we will identify equivalence classes of $\dM$-Cauchy sequences having a
limit in $\Omega$ with that limit point.
By considering equivalence classes of $\dM$-Cauchy sequences without
limits in $\Om$ we define the boundary of $\Om$ with respect to $d_M$
as $\partial_M\Om=\clOmm\setminus \Om$.
Since $X$ is proper, we know that $\Om$ is locally compact
with respect to $d_M$, and it
follows that $\Om$ is an open subset of
$\clOmm$.
We extend the original metric $d_M$ on $\Om$  to $\clOmm$
by setting
\[
 d_M(x^*, y^*)=\lim_{n\to\infty}d_M(x_n, y_n),
\]
if $x^*=\{x_n\}_{n=1}^\infty \in \clOmm$ and
$y^*= \{y_n\}_{n=1}^\infty   \in \clOmm$.
This is well defined and an extension of $\dM$.

By the construction of $\clOmm$ and Remark~\ref{preserveLength}, every point in  $\Om$
can be identified with exactly one equivalence class of
$\dM$-Cauchy sequences in $\Om$.
This is, of course, not true on the boundary of $\Om$ in general, as illustrated by the following example.

\begin{example}\label{ex-maz-top-slit}
Consider the planar slit disk in Example~\ref{ex-slit-disc}.
To every point $x\in [-1,0)\times\{0\}$ there correspond
exactly two points in the Mazurkiewicz boundary given by sequences approaching
$x$ from the upper and lower half-planes, respectively.
For instance the sequence
 $x_n=(x, (-1)^n/n)$ for $n=2, 3, \ldots$ converges to $x$
in the Euclidean metric but is not a $d_M$-Cauchy sequence
 as
\[
d_M(x_{n}, x_{n+1})\geq 2|x| \quad \text{for } n=2, 3, \ldots.
\]
\end{example}

In the next example we show that a point in the topological boundary need
not correspond to a limit of any $d_M$-Cauchy sequence
in the Mazurkiewicz boundary.

\begin{example}{(The topologist's comb)}\label{ex-maz-top-comb}
 Let $\Om$ be the topologist's comb as in Example~\ref{comb}.
Then no point in the bottom segment $I=\bigl(\frac12, 1\bigr]\times \{0\}$
corresponds to an element of $\partial_M \Om$.
Namely, any sequence of points in $\Om$ converging to a point in $I$
fails the Cauchy condition with respect to the Mazurkiewicz distance.
\end{example}

\begin{lemma}\label{lem-M-bdry-to-bdry}
There is a continuous map $\Psi:\clOmm \to \clOm$ such that $\Psi|_{\Om}$ is the identity map
and $\Psi|_{\bdyM \Om}:\partial_M\Omega \to \partial \Omega$.
\end{lemma}

This mapping need not be injective nor surjective in general, as demonstrated by the slit disk and the
topologist's comb in Examples~\ref{ex-maz-top-slit} and~\ref{ex-maz-top-comb}, respectively.

\begin{proof}
 Let $\{x_n\}_{n=1}^\infty$ be a $d_M$-Cauchy sequence in $\Om$ representing
a point in $\clOmm$.
 Since $d(x_i,x_j)\le d_M(x_i,x_j)$, it follows that $\{x_n\}_{n=1}^\infty$
is a Cauchy sequence in
the given metric $d$
 as well, and so by the completeness of $X$, we can set
\[
   \Psi\left(\{x_n\}_{n=1}^\infty\right)=\lim_{n\to\infty} x_n\in \clOm.
\]
The map $\Psi$ is well defined, since every sequence representing the same
point in $\clOmm$  converges to the same limit in the given metric $d$.

To prove the continuity of $\Psi$,  consider
$\{x_n\}_{n=1}^\infty, \{y_n\}_{n=1}^\infty \in \clOmm$
and let $x=\Psi\left(\{x_n\}_{n=1}^\infty\right)$
and $y=\Psi\left(\{y_n\}_{n=1}^\infty\right)$.
Then by definition we have that
\[
d_M(\{x_n\}_{n=1}^\infty,\{y_n\}_{n=1}^\infty)=\lim_{n\to\infty} d_M(x_n, y_n)\geq \lim_{n\to\infty} d(x_n, y_n)=d(x,y).
\]
 Therefore $d(\Psi(\{x_n\}_{n=1}^\infty), \Psi(\{y_n\}_{n=1}^\infty))\!\le d_M(\{x_n\}_{n=1}^\infty,\{y_n\}_{n=1}^\infty)$, that is,
 $\Psi$ is $1$-Lipschitz continuous.
\end{proof}

Next, we show that under rather general assumptions, the prime end boundary
and the Mazurkiewicz boundary coincide.

\begin{theorem}\label{thm-homeo-primeends}
Assume that every prime end in $\Om$ has a singleton impression.
Then there is a homeomorphism
$\Phi: \partial_P\Omega\rightarrow \partial_M\Omega$.
\end{theorem}

This is a special case of the following result. Recall from
Proposition~\ref{prop-end-single} that every singleton end is a prime end.

\begin{theorem}\label{thm-homeo-primeends-2}
Let $\bdySP \Om$ be the set of all singleton ends.
Then there is a homeomorphism
$\Phi: \Om \cup \bdySP \Om\to \clOmm$
such that $\Phi|_{\Om}$ is the identity map
and
$\Phi|_{\bdySP \Om}: \bdySP \Om\to \bdyM \Om$.
\end{theorem}

Recall that by Lemma~\ref{lem-single-char}
an end $[E_k]$ has a singleton impression if and only if
$\lim_{k \to \infty}\diam E_k =0$.
We will use this fact (implicitly) several times in the proof below.

\begin{proof}
\emph{Step} 1. \emph{Definition of $\Phi$.}
Let $[E_k]\in\bdySP\Om$.
For each $k$ choose $x_k\in E_k$.
Then for $l\geq k$ we have that $x_k, x_l\in E_k$
and as $E_k$ is connected, this implies that
\[
   d_M(x_k, x_l)\leq \diam E_k \to 0, \quad \text{as } k\to\infty.
\]
Thus, $\{x_k\}_{k=1}^\infty$ is a $d_M$-Cauchy sequence
and corresponds to a point  $y \in \clOmm$.
If $y$ belonged to $\Om$, then we would have
$y \in \bigcap_{k=1}^\infty \itoverline{E}_k\cap\Om=I[E_k]\cap\Om =\emptyset$,
which is a contradiction.
Thus $y \in \bdyM \Om$,
and we define
\[
 \Phi([E_k])=y.
\]
For $x \in \Om$ we, of course, set
$\Phi(x)=x$.

\emph{Step} 2. \emph{$\Phi$ is well defined.}
Assume that $\{E_k\}_{k=1}^\infty$ and $\{E'_k\}_{k=1}^\infty$
are equivalent chains, and let $x_k\in E_k$
and $x'_k\in E'_k$, $k=1,2,\ldots.$
Then for every $k$, there exists $l_k\ge k$ such that
$E_{l_k} \subset E'_k$.
Hence for all $l\ge l_k$, we have that
$x_l\in E_l\subset E_{l_k}\subset E'_k$ and
$x'_l\in E'_l\subset E'_k$.
Thus
\[
\lim_{l \to \infty} d_M(x_l,x'_l)\le \diam E'_k \to 0 \quad \text{as } k\to \infty,
\]
showing that
$\{x_k\}_{k=1}^\infty$ and $\{x'_k\}_{k=1}^\infty$ are equivalent as
$d_M$-Cauchy sequences. Hence $\Phi$ is well-defined.

\emph{Step} 3. \emph{$\Phi$ is surjective.}
Let $\{x_n\}_{n=1}^\infty$
be a $d_M$-Cauchy sequence
in $\Om$, corresponding to a point in $\bdry_M\Om$.
We can assume that for all $j,k\ge n$,
\begin{equation}
d(x_j,x_k)\le d_M(x_j,x_{k})< 2^{-n-1}.
\label{eq-surjective}
\end{equation}
It follows that $\{x_n\}_{n=1}^\infty$ is a $d$-Cauchy sequence
and converges to some $x\in\bdry\Om$, and moreover,
\begin{equation}
d(x_k,x)\le 2^{-k-1}.
\label{eq-x_k-x}
\end{equation}
For each $k=1,2,\ldots,$
let $E_k$ be the component of
$\Om\cap B(x,2^{-k})$ containing $x_{k}$.
Then for all $j\ge k$,
\eqref{eq-surjective} implies that there exists a connected set
$F_j\subset \Om$ such that $x_j, x_k\in F_j$ and $\diam F_j < 2^{-k-1}$.
From \eqref{eq-x_k-x} it follows that $F_j\subset \Om\cap
B(x,2^{-k})$.
Ass $E_k$ is a component of $\Om \cap B(x,2^{-k})$
and $x_k\in E_k$, we obtain
that the connected set
$F_j\subset E_k$ and thus $x_j \in E_k$ for all $j\geq k$.
Letting $j\to\infty$ shows that $x\in \itoverline{E}_k$ for $k=1,2,\ldots$.

This also shows that $x_{k+1}\in E_k$ and as $E_{k+1}$ is
connected, we obtain that $E_{k+1} \subset E_k$ for all $k$.
Since $\Omega\cap\bdry E_k \subset \bdry B(x,2^{-k})$, we see that
\[
   \dist(\Omega\cap\bdry E_{k+1},\Omega\cap \bdry E_k) \ge 2^{-k-1}>0.
\]
By construction we know that $\diam E_k\to 0$, and hence
$\{E_k\}_{k=1}^\infty$ is a chain
with impression $\{x\}$.
By Proposition~\ref{prop-end-single},
$[E_k]$ is a prime end.
Moreover, $\Phi([E_k])=\{x_n\}_{n=1}^\infty$.
Thus $\Phi$ is surjective.
(That $\Phi|_{\Om}$ is bijective is clear.)

\emph{Step} 4. \emph{$\Phi$ is injective.}
Let $[E_k]$ and $[F_k]$ be two
distinct  singleton prime ends. So $\{F_k\}_{k=1}^\infty$ does not
divide $\{E_k\}_{k=1}^\infty$. Hence, there exists $k$ such
that for each $l$ we can find a point $y_l\in
F_l\setm E_k$. We need to show that $\{y_l\}_{l=1}^\infty$ is not
equivalent to any sequence representing $\Phi([E_k])$.
Let $x_l\in E_l$ for each  $l$.
Since $x_l\in E_{k+1}$ and $y_l\notin E_k$ for $l>k$,
Remark~\ref{rem-connected-diam} yields that every connected set $A$
containing both $x_l$ and $y_l$ satisfies
\[
   \diam A \ge \dist (\Om\cap\bdry E_{k+1}, \Om\cap\bdry E_{k}).
\]
Hence, for each $l\ge k+1$ we have that
\[
d_M(x_l,y_l) \ge \dist (\Om\cap\bdry E_{k+1}, \Om\cap\bdry E_{k})>0.
\]
Thus the two sequences $\{x_l\}_{l=1}^\infty$ and $\{y_l\}_{l=1}^\infty$
are not equivalent,
and $\Phi$ is  injective.

\emph{Step} 5. \emph{$\Phi$ is continuous.}
We need to show that preimages of closed sets are closed.
Since the topologies on $\clOmm$ and $\Om \cup \bdySP$ are given by converging
sequences, it suffices to consider sequential continuity.
As $\Phi|_{\Om}$ is continuous, it is enough to show
that the image of every sequence with a limit in $\bdySP \Om$
has the correct limit. There are two such types of sequences
we need to consider.

Assume first that the sequence
of singleton prime ends
$\{[E_k^n]\}_{n=1}^\infty$ converges to a singleton prime end
$[E_k^\infty]$. 
Let $\Phi([E_k^n])= \{x_k^n\}_{k=1}^\infty$ and
$\Phi([E_k^\infty])= \{x_k^\infty\}_{k=1}^\infty$, where $x_k^n\in E_k^n$, $n=1, 2, \ldots$, and
$x_k^\infty \in E_k^\infty$ are provided by Steps~1 and~2.
Then it is clear that $\{x_k^n\}_{k=1}^\infty$ converges to $[E_k^n]$ for each $n$,
and that $\{x_k^\infty\}_{k=1}^\infty$ converges to $[E_k^\infty]$. 
By Lemma~\ref{lem-dM-sequential},
it follows that $\lim_{n\to\infty}\limsup_{k\to\infty}d_M(x_k^n,x_k^\infty)=0$,
i.e.\ $\lim_{n\to\infty} d_M(\Phi([E_k^n]),\Phi([E_k^\infty]))=0$.
This shows that $\{\Phi([E_k^n])\}_{n=1}^\infty$ converges in $d_M$ to
$\Phi([E_k^\infty])$ as $n\to\infty$.

Assume next that $\Om \ni y_n \to [E_k] \in \bdySP \Om$ as $n \to
\infty$.
Thus for each $k$ we can find $n_k$ such that $y_{n} \in E_k$
whenever $n \ge n_k$.
As $E_k$ is connected we see that
\[
     \dM(y_l, y_n) \le \diam E_k
     \quad \text{for } l,n \ge n_k.
\]
Since $\diam E_k \to 0$ as $k \to \infty$, this shows that
$\{y_n\}_{n=1}^\infty$ is a $\dM$-Cauchy sequence.
As $y_{n_k}\in E_k$, the sequence $\{y_{n_k}\}_{k=1}^\infty$
represents $\Phi([E_k])$ and is equivalent to $\{y_n\}_{n=1}^\infty$,
which is the limit of the sequence $\{y_n\}_{n=1}^\infty$ in $\clOmm$.

Thus we conclude that  $\Phi$ is
continuous.

\emph{Step} 6. \emph{$\Phi^{-1}$ is continuous.}
As in Step~5 there are two types of sequences we need to consider.

Assume first that
the sequence of singleton prime ends $\{[E_k^n]\}_{n=1}^\infty$
does not converge to the singleton prime end $[E_k^\infty]$.
Then by Lemma~\ref{lem-dM-sequential}, there are sequences
$\{x_k^n\}_{k=1}^\infty$ converging to $[E_k^n]$ for each $n$, and 
$\{x_k^\infty\}_{k=1}^\infty$ converging to $[E_k^\infty]$, such that
\[
  \limsup_{n\to\infty}\, \limsup_{k\to\infty}d_M(x_k^n,x_k^\infty)>0.
\]
Because the ends $[E_k^n]$ and $[E_k^\infty]$ are singleton ends, 
it follows that $\Phi([E_k^n])$ is represented by $\{x_k^n\}_{k=1}^\infty$
and $\Phi([E_k^\infty])$ is represented by $\{x_k^\infty\}_{k=1}^\infty$.
Therefore, it is not true that $\lim_{n\to\infty}d_M(\Phi([E_k^n]),\Phi([E_k^\infty]))=0$.

Assume next that $\{x_n\}_{n=1}^\infty$ is a sequence of points
in $\Om$ which does not converge to the singleton prime end $[E_k^\infty]$.
Then we can find $k$ and an increasing
sequence $n_i\to \infty$ 
for which $x_{n_i}\notin E_k^\infty$.
Let $y_l \in E_l^\infty$ for $l=1,2,\ldots$, i.e.\ 
$\Phi([E_l^\infty])=\{y_l\}_{l=1}^\infty$.
As the sequence $\{x_{n_i}\}_{i=1}^\infty$  lies entirely in $\Omega\setminus E_k^\infty$,
Remark~\ref{rem-connected-diam}
shows that
\[
     d_M(y_l,x_{n_i}) \ge \dist(\Om \cap\partial E_{k+1}^\infty,
      \Omega\cap\partial E_k^\infty)
      \quad \text{for } l>k \text{ and all } i.
\]
Thus
\[
     d_M(\Phi([E_k^\infty]),x_{n_i}) \ge \dist(\Om \cap\partial E_{k+1}^\infty,
      \Omega\cap\partial E_k^\infty) > 0,
\]
which shows that $\{x_{n_i}\}_{i=1}^\infty$ cannot
converge to $\Phi([E_k^\infty])$,
and hence neither can  $\{x_{n}\}_{n=1}^\infty$
converge to $\Phi([E_k^\infty])$.

This shows that $\Phi^{-1}$ is continuous and so
$\Phi$ is a homeomorphism.
\end{proof}

\section{Domains finitely connected at the boundary}
\label{sect-finconn}

In general not all prime ends have singleton impressions, as demonstrated by
Example~\ref{ex-double-equilat-comb}. In this section and the next section we
explore conditions under which prime
ends have this property (cf. Section~\ref{sect-access}).

Here we present a topological condition, finite connectedness
at the boundary, which guarantees that all
prime ends have singleton impressions.
Finite connectedness at the boundary
is equivalent to the compactness of $\Om \cup \bdySP \Om$,
where $\bdySP \Om$ is the set of all singleton prime ends,
see Theorem~\ref{thm-clOmm-cpt-new} and the comments after it.

\begin{deff}\label{def-finite-conn}
We say that $\Om$ is \emph{finitely connected\/} at a point
$x_0\in\partial\Om$ if for
every $r>0$ there is an open set $G$ (open in $X$) such that $x_0
\in G \subset B(x_0,r)$ and $G \cap \Om$ has only finitely many
components.
If $\Om$ is finitely connected at every boundary point,
then it is called \emph{finitely connected at the boundary}.
\end{deff}

This terminology follows N\"akki~\cite{nakki70},
who seems to have first used it  in print. 
(N\"akki~\cite{nakki-private} has informed us that
he learned about it from V\"ais\"al\"a, who
however first seems to have used it in print in~\cite{vaisala}.)
Beware that the notion of finitely connected domains is a 
completely different notion.

We now introduce some further notation.
Fix $x_0 \in \bdy \Om$
(we do \emph{not} assume that $\Om$ is finitely connected here).
For each $r>0$ let $\{\Gjr\}_{j=1}^{N(r)}$ be the family of components
of $B(x_0,r) \cap \Om$ which have $x_0$ in their boundary,
i.e.\ $x_0 \in \overline{\Gjr}$.
Here $N(r)$ is either a nonnegative integer or $\infty$.
Let 
\[
    H(r)=(B(x_0,r) \cap \Om) \setm
     \bigcup_{j=1}^{N(r)} \Gjr
\]
be the union of the remaining components (if any).
(The sets $\Gjr$ and $H(r)$ of course depend on $x_0$.)

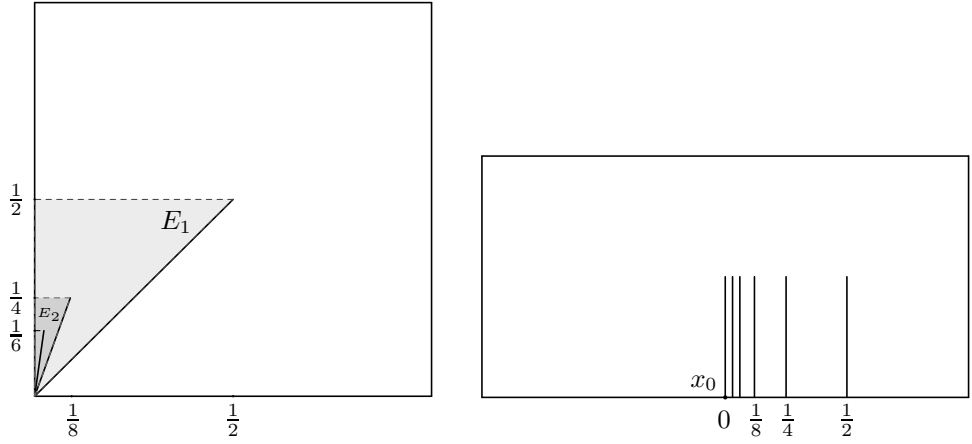
\begin{figure}[t]
\centering
\subfloat[\empty]
{
\begin{tikzpicture}[line cap=round,line join=round,>=triangle 45, scale = 0.87]
\clip(-3.37,-2.34) rectangle (3.4,5.2);
\fill[dash pattern=on 2pt off 2pt,color=zzttqq,fill=zzttqq,fill opacity=0.1] (-3,2) -- (-3,-1) -- (0,2) -- cycle;
\fill[dash pattern=on 2pt off 2pt,color=wwwwww,fill=wwwwww,fill opacity=0.2] (-3,0.5) -- (-3,-1) -- (-2.46,0.5) -- cycle;
\draw  [line width=0.6pt] (-3,5)-- (-3,-1);
\draw  [line width=0.6pt] (-3,-1)-- (3,-1);
\draw  [line width=0.6pt] (3,-1)-- (3,5);
\draw  [line width=0.6pt] (3,5)-- (-3,5);
\draw  [line width=0.6pt] (-3,-1)-- (0,2);
\draw  [line width=0.6pt] (-3,-1)-- (-2.46,0.5);
\draw  [line width=0.6pt] (-3,-1)-- (-2.86,0);
\draw [dash pattern=on 2pt off 2pt,color=zzttqq] (-3,2)-- (-3,-1);
\draw [dash pattern=on 2pt off 2pt,color=zzttqq] (-3,-1)-- (0,2);
\draw [dash pattern=on 2pt off 2pt,color=zzttqq] (0,2)-- (-3,2);
\draw [dash pattern=on 2pt off 2pt,color=wwwwww] (-3,0.5)-- (-3,-1);
\draw [dash pattern=on 2pt off 2pt,color=wwwwww] (-3,-1)-- (-2.46,0.5);
\draw [dash pattern=on 2pt off 2pt,color=wwwwww] (-2.46,0.5)-- (-3,0.5);
\draw [dash pattern=on 2pt off 2pt] (-3,0)-- (-2.9,0);
\fill [color=black] (0,-1) circle (0.5pt);
\draw[color=black] (0,-1.36) node {$\frac12$};
\fill [color=black] (-2.44,-1) circle (0.5pt);
\draw[color=black] (-2.42,-1.36) node {$\frac18$};
\fill [color=black] (-3,2) circle (0.5pt);
\draw[color=black] (-3.28,2) node {$\frac12$};
\fill [color=black] (-3,0.5) circle (0.5pt);
\draw[color=black] (-3.28,0.5) node {$\frac14$};
\fill [color=black] (-3,0) circle (0.5pt);
\draw[color=black] (-3.28,-0.07) node {$\frac16$};
\draw[color=black] (-0.86,1.65) node {$E_1$};
\draw[color=black] (-2.78,0.25) node {\tiny{$E_2$}};
\end{tikzpicture}
}
\subfloat[\empty]
{
\begin{tikzpicture}[line cap=round,line join=round,>=triangle 45, scale =0.8]
\clip(-3.1,-0.44) rectangle (5.98,5.8);
\draw  [line width=0.6pt] (-3,5)-- (-3,1);
\draw  [line width=0.6pt] (-3,1)-- (5,1);
\draw  [line width=0.6pt] (5,1)-- (5,5);
\draw  [line width=0.6pt] (5,5)-- (-3,5);
\draw  [line width=0.6pt] (1,3)-- (1,1);
\draw  [line width=0.6pt] (3,1)-- (3,3);
\draw  [line width=0.6pt] (2,3)-- (2,1);
\draw  [line width=0.6pt] (1.48,3)-- (1.48,1);
\draw  [line width=0.6pt] (1.24,3)-- (1.24,1);
\draw  [line width=0.6pt] (1.12,3)-- (1.12,1);
\fill [color=black] (1,1) circle (1.0pt);
\draw[color=black] (0.65,1.24) node {${x_0}$};
\draw[color=black] (0.98,0.63) node {${0}$};
\draw[color=black] (3,0.63) node {${\frac12}$};
\draw[color=black] (2.03,0.63) node {${\frac14}$};
\draw[color=black] (1.51,0.63) node {${\frac18}$};
\end{tikzpicture}
}
\caption{\label{fig3}%
Examples~\ref{ex-shrinking-pins} (left)  and \ref{ex-10.3} (right).}
\end{figure}

\begin{example}   \label{ex-shrinking-pins}
(See Figure~\ref{fig3}.)
Let $\Om$ be obtained from the unit square $(0,1)^2$ in $\R^2$ by removing
the segments
\[
S_k = \biggl\{(x,y): 0<x\le\frac{1}{2k^2}\text{ and } y=k x \biggr\}, 
\quad k=1,2,\ldots.
\]
Then $\Om$ is finitely connected at the boundary but there are infinitely
many prime ends with the origin as their impression.
For each ``wedge'' there is one such singleton prime end.  The
sequence consisting of these prime ends
converges to the singleton prime end defined by the acceptable sets
\[
E_k = \biggl\{(x,y): 0<x<\frac{1}{2k^2} \text{ and } kx<y<\frac{1}{2k}\biggl\}, \quad k=1,2,\ldots.
\]
\end{example}

\begin{example}
\label{ex-10.3}
(See Figure~\ref{fig3}.)
 Let
\[
\Om=((-1,1)\times (0,1))\setm
\bigl(\bigl\{0,\tfrac12, \tfrac14, \tfrac18, \ldots\bigl\}
\times \bigl(0,\tfrac12\bigl]\bigl)
\]
 (see the above figure).
Then for the point $x_0=(0,0)$ we have that $N(r)=1$ for all $0<r<1$ but
the domain is not finitely connected at $x_0$.
\end{example}

The following characterization of finite connectedness is useful, see
Bj\"orn--Bj\"orn--Shanmugalingam~\cite{BBSmbdy} for a proof.

\begin{prop} \label{prop1-fin}
The set $\Om$ is finitely connected at $x_0$
if and only if for each $r>0$ we have $N(r)<\infty$
and $x_0 \notin \itoverline{\Hr{r}}$.
\end{prop}

We next present two auxiliary results which will primarily be applied
to  ends, but we state them here for more general sets since it will be useful later
in the proofs of Theorem~\ref{thm-fin-con-homeo} 
and Proposition~\ref{prop-finn-conn-N-ends}.

\begin{lem}  \label{lem-ex-G-j}
Assume that $\Om$ is finitely connected at $x_0\in\bdry\Om$.
Let $A_k\subsetneq\Om$ be such that
$A_{k+1}\subset A_k$, $x_0\in \itoverline{A}_k$
and $\dist(x_0,\Om\cap\bdry A_k)>0$ for each $k=1,2,\ldots$.
Furthermore, let 
$0<r_k<\dist(x_0,\Om\cap\bdry A_k)$ be a sequence decreasing to zero.
Then for each $k=1,2,\ldots$ 
there is a component
$G_{j_k}(r_k)$ of $B(x_0, r_k) \cap \Om$ intersecting $A_l$ for each $l=1,2,\ldots$,
and such that
$x_0\in \itoverline{G_{j_k}(r_k)}$ and $G_{j_k}(r_k) \subset A_k$.
\end{lem}

\begin{proof}
 Consider the components $G_1(r_k), \ldots, G_{N(r_k)}(r_k)$ of
$B(x_0,r_k)\cap\Om$ which have $x_0$ in their boundary. Let
$H(r_k)=(\Om\cap B(x_0,r_k))\setminus\bigcup_{j=1}^{N(r_k)}G_j(r_k)$. As
$\Om$ is finitely connected at $x_0$, Proposition~\ref{prop1-fin} shows
that $x_0\notin \itoverline{\Hr{r_k}}$, so for each $l=1,2,\ldots$,
at least one of
$G_1(r_k), \ldots, G_{N(r_k)}(r_k)$ has a nonempty intersection with
$A_l$. 
Since there are only finitely many components
$G_j(r_k)$, $j=1,2,\ldots,N(r_k)$,
at least one of them intersects infinitely many (and thus all) $A_l$.
Call this component $G_{j_k}(r_k)$.
As it is connected and $r_k<\dist(x_0,\Om\cap\bdry A_k)$,
we must have $G_{j_k}(r_k) \subset A_k$.
\end{proof}

\begin{lem}  \label{lem-ex-chain-10}
Assume that $\Om$ is finitely connected at $x_0\in\bdry\Om$.
Let $A_k\subsetneq\Om$ and $r_k>0$ be as in the statement of Lemma~\ref{lem-ex-G-j}.

Then there exists a prime end $[F_k]$  such that  $I[F_k]=\{x_0\}$,
$F_k=G_{j_k}(r_k)$ for some
$1\le j_k\le N(r_k)$ and
$F_{k}\subset A_k$, $k=1,2,\ldots$.
\end{lem}

\begin{proof}
Consider the rooted tree whose vertices are
$G_j(r_k)$, $j=1,2,\ldots,N(r_k)$, $k=1,2,\ldots$,
and  where two vertices are connected by an edge
if and only if they
are $G_j(r_k)$ and $G_i(r_{k+1})$ for some $i$, $j$ and $k$
with 
$G_i(r_{k+1}) \subset G_j(r_k)$.

Consider the collection $\PP$ of all
descending paths in the tree starting from the root
(including finite ones).
We introduce a metric $t$ on $\PP$ by letting
$t(p,q)=2^{-n}$, where $n$ is the level where the paths $p$ and $q$ branch
(or end),
i.e.\ $n$ is the largest integer such that $p$ and $q$
have a common vertex $G_j(r_n)$.
Since $\Om$ is finitely connected at $x_0$,
for each $l=1,2,\ldots$ there are only finitely many vertices
in the first $l$ levels of the tree.
It follows that $\PP$ is totally bounded in the metric $t$.

For each $k=1,2,\ldots$, we consider the subcollection $\PP_k$
consisting of all paths $p\in\PP$ for which there exists a component
$G_j(r_{k})\subset A_k$ such that $p$ passes through the vertex
$G_j(r_{k})$. 
Lemma~\ref{lem-ex-G-j} guarantees that each $\PP_k$  is nonempty.
Let $p\in\PP_{k+1}$ and let $G_i(r_{k+1})$ be a vertex $p$ passes through.
Let $G_j(r_{k})$ be the component of $\Om\cap B(x_0,r_k)$ containing
$G_i(r_{k+1})$.
Since $A_{k+1}\subset A_k$, we see that $G_j(r_{k})\cap A_k$ is nonempty
and as $G_j(r_{k})$ is connected and $r_k<\dist(x_0,\Om\cap\bdry A_k)$,
we conclude that $G_j(r_{k})\subset A_k$.
Hence $\PP_{k+1}\subset\PP_k$ for $k=1,2,\ldots$.

We now verify that each $\PP_k$ is complete.
Indeed, if $\{p_n\}_{n=1}^\infty\subset \PP_k$ is a Cauchy sequence
in the metric $t$, then for every $l=1,2,\ldots$, there exists
$n_l$ such that the paths $p_n$ and $p_m$ have the first $l$ vertices
in common, whenever $n,m\ge n_l$.
This makes it possible to construct a path $p\in \PP_k$ which for every
$l=1,2,\ldots$ has the first $l$ vertices in common with all $p_n$,
$n\ge n_l$, i.e.\ $p_n\to p$ in the metric $t$.

As $\PP$ is totally bounded, it follows that all $\PP_k$, $k=1,2,\ldots$,
are compact.
Hence $\{P_k\}_{k=1}^\infty$ is a decreasing sequence of nonempty compact sets,
and thus
there exists an infinite path $q \in \bigcap_{k=1}^\infty \PP_k$.
The vertices through which it passes
define the end $[F_k]$ such that $F_k=G_{j_k}(r_k)\subset A_k$, $k=1,2,\ldots$.
Since $\diam F_k\le 2r_k \to 0$ as $k \to \infty$,
this end is a prime end by
Proposition~\ref{prop-end-single} (and Lemma~\ref{lem-single-char}).
\end{proof}

\begin{prop}   \label{prop-end-divisible}
Assume that $\Om$ is finitely connected at $x_0\in\bdry\Om$.
If $[E_k]$ is an end with $x_0\in I[E_k]$,
then there is a  prime end $[F_k]$ dividing $[E_k]$ such that
$I[F_k]=\{x_0\}$.
If moreover
$[E_k]$ is a prime end, then
$[E_k]=[F_k]$ and $I[E_k]=\{x_0\}$.
\end{prop}

\begin{proof}
As $\dist(\Om\cap\bdry E_{k+1}, \Om\cap\bdry E_k) >0$, at least one
of $\dist(x_0,\Om\cap\bdry E_k)$ and
$\dist(x_0,\Om\cap\bdry E_{k+1})$ must be positive.
We can therefore 
choose a subsequence of
$\{E_k\}_{k=1}^\infty$ 
to obtain an equivalent chain, also denoted $\{E_k\}_{k=1}^\infty$,
where all those distances are positive.
Hence, we can inductively construct a sequence $\{r_k\}_{k=1}^\infty$
decreasing to zero, such that $0<r_k<\dist(x_0,\Om\cap\bdry A_k)$.
Lemma~\ref{lem-ex-chain-10}, applied with $A_k=E_k$,
provides a prime end $[F_k]$ with the
desired properties.
If moreover $[E_k]$ is a prime end, then we must have $[E_k]=[F_k]$ and
thus $I[E_k]=\{x_0\}$.
\end{proof}

\begin{thm}  \label{thm-fin-con-homeo}
Assume that $\Om$ is finitely connected at the boundary.
Then all prime ends
have singleton impressions, and every $x\in\bdry\Om$ is the impression
of a prime end and is accessible.

Furthermore, if $1\le p\in Q(x)\ne (0,1]$ for each $x\in\bdry\Om$,
then $\partial_P\Om$ is also the $\Mod_p$-prime end boundary.
\end{thm}

\begin{proof}
That all prime ends have singleton impressions follows from
Proposition~\ref{prop-end-divisible}.
If $x\in\bdry\Om$, then applying Lemma~\ref{lem-ex-chain-10} to
$A_k=\Om\setm\{y\}$ for some $y\in\Om$ yields a prime end $[F_k]$ with $\{x\}$ as
its impression.
Proposition~\ref{prop1A-A} shows that $x$ is accessible.
Finally, Proposition~\ref{prop-single-Modp} shows that
all prime ends are also $\Mod_p$-prime ends if 
$p$ is as in the statement of the theorem.
\end{proof}

The next few results relate prime ends to the Mazurkiewicz boundary.
The conclusions about metrizability and compactness will be important
for future studies on Dirichlet problems with respect to
prime end boundaries. 

The following result follows directly from 
Theorems~\ref{thm-homeo-primeends-2} and~\ref{thm-fin-con-homeo}.

\begin{cor}\label{metric-lemma-5-12}
Assume that $\Om$ is finitely connected at the boundary. 
Then there is a homeomorphism $\Phi: \clOmP \to \clOmm$ such that $\Phi|_{\Om}$ is the 
identity map. Moreover,  the prime end
closure $\clOmP$ is metrizable with the metric
\(
  m_P(x,y):=d_M(\Phi(x), \Phi(y)).
\)
The topology on $\clOmP$ given by this metric
is equivalent to the topology given by
the sequential convergence discussed in Section~\ref{sect-top}.
\end{cor}

\begin{thm} \label{thm-clOmm-cpt-new}
The following are equivalent\/\textup{:}
\begin{enumerate}
\item \label{i3-fin}
$\Om$ is finitely connected at the boundary\/\textup{;}
\item \label{i3-P}
$\clOmP$ is compact and all prime ends have singleton impressions\/\textup{;}
\item \label{i3-SP}
$\Om \cup \bdySP \Om$ is compact\/\textup{;}
\item \label{i3-m}
$\clOmm$ is compact.
\end{enumerate}
\end{thm}

\begin{proof}
\ref{i3-fin} \eqv \ref{i3-m}
This is shown in Bj\"orn--Bj\"orn--Shanmugalingam~\cite{BBSmbdy}.

\ref{i3-SP} \eqv \ref{i3-m}
This follows directly from
Theorem~\ref{thm-homeo-primeends-2}.

\ref{i3-fin} \imp \ref{i3-P}
By  Theorem~\ref{thm-fin-con-homeo},
 all prime ends have singleton impressions.
Hence $\clOmP=\Om \cup \bdySP \Om$, which is compact
by the already shown implication \ref{i3-fin} \imp \ref{i3-SP}.

\ref{i3-P} \imp \ref{i3-SP}
Since all prime ends have singleton impressions,
we have that
$\Om \cup \bdySP \Om=\clOmP$, which is compact
by assumption.
\end{proof}

The fact that all prime ends have singleton impressions is on its own
not sufficient for $\Om$ to be finitely connected at the boundary, see e.g.\
the topologist's comb in Example~\ref{ex-maz-top-comb}
whose prime end closure $\clOmP$ is not compact.
On the other hand, the double  comb
below has a compact prime end closure 
but is not finitely connected at the boundary and has a nonsingleton prime
end. 
(Note that the double equilateral comb in Example~\ref{ex-double-equilat-comb} does not
have a compact prime end closure.)

\begin{figure}[t]
\centerline{
\begin{tikzpicture}[line cap=round,line join=round,>=triangle 45, scale =0.6]
\clip(-0.53,-1.4) rectangle (13,8);
\draw  [line width=0.6pt] (2,8)-- (2,0);
\draw  [line width=1.4pt] (2,0)-- (10,0);
\draw  [line width=0.6pt] (10,0)-- (10,8);
\draw  [line width=0.6pt] (10,8)-- (2,8);
\draw  [line width=0.6pt] (2,4)-- (6,4);
\draw  [line width=0.6pt] (2,2)-- (8,2);
\draw  [line width=0.6pt] (2,0.98)-- (9.02,0.98);
\draw  [line width=0.6pt] (2,0.46)-- (9.48,0.46);
\draw  [line width=0.6pt] (2,0.23)-- (9.75,0.23);
\draw  [line width=0.6pt] (5.97,3)-- (10,3);
\draw  [line width=0.6pt] (4,1.47)-- (10,1.47);
\draw  [line width=0.6pt] (3.05,0.72)-- (10,0.72);
\draw  [line width=0.6pt] (2.5,0.36)-- (10,0.36);
\fill [color=black] (2,4) circle (0.5pt);
\draw[color=black] (1.60,4) node {${\frac12}$};
\fill [color=black] (2,2) circle (0.5pt);
\draw[color=black] (1.61,2) node {${\frac14}$};
\fill [color=black] (2,0.98) circle (0.5pt);
\draw[color=black] (1.57,1) node {${\frac18}$};
\fill [color=black] (2,0.46) circle (0.5pt);
\fill [color=black] (10,3) circle (0.5pt);
\draw[color=black] (10.43,3) node {${\frac38}$};
\fill [color=black] (10,1.47) circle (0.5pt);
\draw[color=black] (10.50,1.47) node {${\frac{3}{16}}$};
\fill [color=black] (10,0.72) circle (0.5pt);
\draw[color=black] (10.53,0.66) node {${\frac{3}{32}}$};
\draw[color=black] (3.0,-0.45) node {${\frac18}$};
\draw[color=black] (4.0,-0.45) node {${\frac14}$};
\draw[color=black] (6.0,-0.45) node {${\frac12}$};
\draw[color=black] (8.0,-0.45) node {${\frac34}$};
\draw[color=black] (9.0,-0.45) node {${\frac78}$};
\end{tikzpicture}
}
\caption{\label{fig4}
Example~\ref{ex-double-comb}.}
\end{figure}
\begin{example}
\label{ex-double-comb}
(Double comb, see Figure~\ref{fig4}.)
Let $\Om\subset \R^2$ be the domain obtained from the unit
square $(0,1)^2$ by removing the collection of segments
$\bigl(0, 1-2^{-n}\bigr]\times \{2^{-n}\}$
and $\bigl[2^{-n}, 1\bigr)\times \{3\cdot 2^{-n-2}\}$
for $n=1,2, \ldots$. 
Then $\Om$ has a prime end with impression $[0,1] \times \{0\}$.
Note also that this prime end is a $\Mod_p$-prime end for all $p\geq 1$. 
\end{example}

We end this section by
providing more details on prime ends
at certain boundary points.
Note  that if $\Om$ is finitely connected at a boundary point then $N(r)\ge1$
at that point and $r\mapsto N(r)$ is nonincreasing,
see Bj\"orn--Bj\"orn--Shanmugalingam~\cite{BBSmbdy}.

\begin{deff} \label{deff-Nconn}
  Assume that $\Om$ is finitely connected at $x_0\in\partial\Om$
and let
\[
N=\lim_{r\to0}N(r).
\]
Then $\Om$ is \emph{$N$-connected at $x_0$} if $N<\infty$,
and \emph{locally connected at $x_0$} if $N=1$.

If $\Om$ is locally connected at every boundary point,
then $\Om$ is  said to be \emph{locally connected at the boundary}.
\end{deff}

\begin{prop} \label{prop-finn-conn-N-ends}
  Assume that $\Om$ is finitely connected at $x_0\in\partial\Om$.
Then there are exactly $N$ distinct prime ends with impression $\{x_0\}$,
where $N$ is as in Definition~\ref{deff-Nconn}.
Furthermore, there is no other prime end with
$x_0$ in its impression.
\end{prop}

\begin{proof}
Assume first that $N$ is finite.
Then there exists $r_0$ such that $N(r)=N$ for all $0<r\le r_0$.
For each $j=1,\ldots, N$ and $0<r<r_0$, consider the 
components $G_j(r)$ of $B(x_0,r)\cap\Om$
which  have $x_0$ in their boundaries.
We  label them in such a way that $G_j(r)\subset G_j(r_0)$.
It can be directly checked that for each $j=1,\ldots, N$,
the choice of $E_k^j=G_j(r_0/k)$, $k=1,\ldots$, gives us an end $[E_k^j]$
with impression $\{x_0\}$.
Clearly, these ends are distinct since they belong to different
components of $B(x_0,r_0)\cap\Om$.
By Proposition~\ref{prop-end-single}, they are prime ends.

To see that these are the only such prime ends,
let $[E_k]$ be a prime end with $x_0 \in I[E_k]$.
By Lemma~\ref{lem-ex-chain-10}, applied with $A_k=E_k$,
there are
positive numbers $r_k $ decreasing to $0$  and
a singleton prime end $[F_k]$ dividing
$[E_k]$ such that $F_k=G_{j_k}(r_k)$ for some
$1\le j_k\le N(r_k)=N$.
As $G_{j_k}(r_k)\subset G_{j_0}(r_0)$ we must have $j_k=j_0$, i.e.\
$[F_k]=[E_k^{j_0}]$.
Since $[E_k]$ is a prime end it follows that $[E_k]=[E_k^{j_0}]$,
showing that there are no more prime ends.

If $N$ is infinite, let $n$ be arbitrary and find $\rho_n$ such that
$N(\rho_n)\ge n$.
For each $j=1,2,\ldots,N(\rho_n)$ apply
Lemma~\ref{lem-ex-chain-10}
to the sets $A_k:=G_j(\rho_n)$, $k=1,2,\ldots$ to obtain
$N(\rho_n)\ge n$ distinct prime ends with $x_0$ as their impression.
Letting $n\to\infty$ shows that there are infinitely many  such distinct
prime ends. By Proposition~\ref{prop-end-divisible} there are no other prime ends
containing $x_0$ in their impressions.
\end{proof}

\begin{cor}  \label{cor-loc-conn}
If $\Om$ is locally connected at the boundary and $[E_k]$ is a prime end
in $\Omega$, then $I[E_k]=\{x\}$ for some $x\in\bdry\Om$ and
there exist radii $r_k^x>0$,  such that
\[
B(x, r_k^x)\cap \Om\subset E_k, \quad k=1,2,\ldots.
\]
Furthermore,
for each $x\in\bdry\Om$, the sets $F_k=G_1(1/k)$,
$k=1,2,\ldots$,
define the only prime end $[F_k]$
with $x$ in its impression.

Moreover, the mapping $\Upsilon: [E_k]\mapsto I[E_k]$
extended by identity in $\Om$ is a homeomorphism
between  the prime end closure $\clOmP$ and the topological
closure $\overline{\Om}$.
\end{cor}

\begin{proof}
The pairing between prime ends and boundary points
follows from Proposition~\ref{prop-finn-conn-N-ends}, 
which also shows that it is a bijection.

The continuity of $\Upsilon$ is straightforward since if
$[E^n_k]\to[E^\infty_k]$ as $n\to\infty$,
then for each $l$ there is $n_l$ such that
$I[E^n_k]\subset \itoverline{E^\infty_l}$
whenever $n\ge n_l$.
As $\diam E^\infty_l\to0$, this implies that $I[E^n_k]\to I[E^\infty_k]$.

To see that $\Upsilon^{-1}$ is continuous, assume that
$I[E^n_k]\to I[E^\infty_k]=\{x\}$ in the given metric.
We can assume that $B(x, r_k^x)\cap \Om\subset E^\infty_k$
for all $k$.
For each $l$ there exists $n_l$ such that $d(I[E^n_k],x)<r_l^x$
whenever $n\ge n_l$.
Since $\diam E^n_k\to0$ we get
$E^n_k\subset B(x,r_l^x)\cap\Om\subset E^\infty_l$
for sufficiently large $k$, i.e.\ $[E^n_k]\to [E^\infty_k]$.
\end{proof}

\section{(Almost) John and uniform domains}
\label{John}
\label{sect-John}

We saw in Theorem~\ref{thm-fin-con-homeo}
that under some conditions all
prime ends are $\Modp$-ends.
The aim of this section is to look at the converse,
i.e.\ when are $\Modp$-ends
automatically prime ends.
We will obtain this converse (for $p>Q-1$)
for uniform and John domains.
In order to also include outward cusps (which are not John domains) we
introduce \emph{almost John domains},
which to our best knowledge have not appeared earlier in the literature.

In this section $\de_\Om(x)$ stands for the distance of the point
$x\in \Om$ to $X\setminus\Om$ with respect to the given metric $d$.

\begin{deff}\label{def-John-uniform}
A domain $\Om\subset X$ is a \emph{John domain}
if there is a constant $C_\Om \ge 1$,
called a \emph{John constant}, and a point $x_0\in\Om$,
called a \emph{John center},
such that for every $x\in\Om$ there exists a rectifiable \emph{John curve}
$\gamma:[0,l_\ga]\to\Om$ parameterized by arc length, such that
$x=\ga(0)$, $x_0=\ga(l_\ga)$ and
\begin{equation}   \label{eq-def-John}
t \le C_\Om \de_\Om(\ga(t))
    \quad \text{for } 0 \le t \le l_\ga.
\end{equation}

A domain $\Om\subset X$ is a \emph{uniform domain} if there is a
constant $C_\Om \ge 1$, called a \emph{uniform constant},
 such that whenever $x,y\in\Om$ there is a
rectifiable curve
$\gamma:[0,l_\ga]\to\Om$, parameterized by arc length,
connecting $x$ to $y$ and
satisfying the following two conditions:
\[
      l_\ga \le C_\Om d(x,y),
\]
and
\[
      \min\{t,l_\ga - t \}\le C_\Om  \delta_\Om(\ga(t))
    \quad \text{for } 0 \le t \le l_\ga.
\]
\end{deff}

A slit disk or a bounded domain satisfying the interior cone
  condition are John domains, while for instance 
outward cusps,
such as
\begin{equation} \label{eq-outward-cusp}
     \Om= \{(x,y) \in \R^2 : 0 <y<x^3 <1\},
\end{equation}
fail condition (\ref{eq-def-John}). 
Among examples of uniform domains we mention quasidisks,
  bounded Lipschitz domains and domains with fractal boundary such as
  the von Koch snowflake. 
See
  Buckley--Stanoyevitch~\cite{bs01}, Heinonen~\cite{he},
Martio--Sarvas~\cite{MS}, N\"akki--V\"ais\"al\"a~\cite{nv91}
  and V\"ais\"al\"a~\cite{vaisala88} for more information on 
John and uniform  domains. 

Observe that uniform domains are necessarily John domains
and that they are locally connected at the boundary, 
see Proposition~\ref{prop-uniform} below.
Note however that there are plenty of John domains
which are locally connected at
the boundary, but not uniform,
e.g.\
inward
cusps in $\R^2$ such as
\[
    B((0,0),1) \setm \{(x,y): 0 \le y \le x^3 < 1\}.
\]  

\begin{prop} \label{prop-uniform}
If $\Om$ is a uniform domain, then it is locally connected at the boundary.
\end{prop}

\begin{proof}
Let $x_0 \in \bdy \Om$, $r>0$  and $x,y \in B(x_0,r/4C_\Om) \cap \Om$,
where $C_\Om \ge 1$ is a uniform constant of $\Om$.
Let $G$ be the component of $B(x_0, r)\cap \Om$ containing $x$.
Then $x$ and $y$ can be connected by a curve $\ga:[0,l_\ga] \to \Om$ with length
\[
    l_\ga \le C_\Om d(x,y) \le C_\Om \frac{r}{2C_\Om} = \frac{r}{2}
\]
from which it follows that $\ga \subset B(x_0,r)$ and hence $y \in G$. Thus, $G\cup B(x_0, r/4C_\Om)$ is 
a neighborhood  of $x_0$ whose intersection with $\Om$ is connected.
\end{proof}

 In this section we will show that under some assumptions all
$\Modp$-ends in John domains are prime ends. Let us however
first focus on connections with the results in the previous section.

\begin{thm} \label{thm-John-Ncon}
Let $\Om$ be a John domain.
Then there is a constant $N$ depending only on the doubling
constant $C_\mu$, the John constant $C_\Om$
and the quasiconvexity constant $L$,
such that
$\Om$ is at most $N$-connected at every boundary point.
\end{thm}

Recall that quasiconvexity was discussed 
at the end of Section~\ref{sect-prelim}.

This result can also be deduced from Lemma~4.3 in 
Aikawa--Shan\-mu\-ga\-lin\-gam~\cite{aish} by an argument similar
to the one at the beginning of the proof below.
Our proof is more direct and has been inspired by the proof of
 Theorem~2.18 in N\"akki--V\"ais\"al\"a~\cite{nv91} 
for domains in $\R^n$. 
It follows from the proof below that 
$N$ can be chosen as the integer part of
$C_\mu^2(3LC_\Om)^{\log_2 C_\mu}$.

\begin{proof}
Let $x_0$ be a John center (with 
 John constant $C_\Om$) and $x\in \bdry \Om$.
Let $B=B(x, r)$ be a ball such that $x_0\notin 3B$.
It is enough to prove that the
ball $B$ intersects at most $N$ components of
$3B\cap \Om$.
The union of these components together with $B$ then makes
the open neighborhood $G$ of $x$  as in Definition~\ref{def-finite-conn}
(for the radius $3r$).
Let $G_1,\ldots, G_k$ be some components of $3B\cap\Om$
which intersect $B$, 
and let $x_j\in G_j \cap B$ for $j=1,\ldots, k$.
Since $\Om$ is a John domain, there exist John curves $\gamma_j$ joining
$x_j$ to  $x_0$.
As
$x_0\notin 3B$, we see that $\gamma_j\cap G_j \cap \bdry 2B
\not=\emptyset$. Choose $y_j=\ga(t_j)\in \gamma_j\cap G_j \cap \bdry 2B$.
Since
$x_j$ is contained in $B$ we have $d(x_j, y_j)>r$, and by the John condition,
$\delta_\Om(y_j)>r/C_\Om$.
Let $B_j=B(y_j,r/LC_\Om)$.
Since $C_\Om\ge 1$, it follows that $LB_j\subset 3B$. If $B_i\cap B_j$ is nonempty
for some $i\ne j$, then there are a point $z\in B_i\cap B_j$  and
two curves
$\beta_i$ and $\beta_j$ connecting $z$ to $y_i$ and $y_j$ respectively, with lengths at most 
\[
\max\{Ld(z,y_i), Ld(z,y_j)\}<\frac{r}{C_\Om}.
\]
From this it follows that $\beta_i$ and $\beta_j$ are both contained in
$LB_i\cup LB_j\subset 3B\cap\Om$. 
Because both $\beta_i$ and $\beta_j$ have $z\in B_i\cap B_j$ in common,
$B_i$ should be contained in the
 same component of $3B\cap\Om$ as $y_j$, which is not possible
since $y_i\in B_i\subset G_i$. 
Hence the  balls $B_j$, $j=1,\ldots, k$, are pairwise disjoint. 
Thus by (\ref{lower-mass-bound}),
\[
    \mu(3B) \ge \sum_{j=1}^k \mu(B_j) 
   \ge \frac{k}{C_\mu^2(3LC_\Om)^{\log_2 C_\mu}} \mu(3B).
\]
Hence $k \le C_\mu^2(3LC_\Om)^{\log_2 C_\mu}$.
\end{proof}

By Theorem~4.32 in Bj\"orn--Bj\"orn~\cite{BBbook} we have an explicit
estimate $L \le 192 C_\mu^3 \CPI$. Hence 
the control over $N$ can be given solely in
terms of $C_\mu$, $C_\Om$, and the constant $\CPI$ associated with the 
Poincar\'e inequality, but not on the dilation 
constant $\la$ in the Poincar\'e inequality.

The above theorem makes it possible to employ the results from the previous section.
However, these conclusions and other results in this section hold for somewhat more general domains as well.
We therefore introduce the following notion.
Recall first that the \emph{$s$-dimensional Hausdorff content}
$\mathcal{H}_\infty^s(E)$ of a set $E\subset X$ is the number
\[
   \mathcal{H}_\infty^s(E):=\inf\biggl\{\sum_{j=1}^\infty r_j^s:
       E\subset\bigcup_{j=1}^\infty B(x_j,r_j) \biggr\}.
\]

\begin{deff}  \label{def-almost-John}
A domain $\Om\subset X$ is an \emph{almost John domain} if
for each $r>0$
there exists a closed set $F\subset \overline{\Om}$ such that
$\mathcal{H}_\infty^1(F) < r$ and $\Om\setm F$ is a John domain.
\end{deff}

Observe that the John constant and John center of $\Om \setm F$
are allowed to depend on $r$.
Typical examples of almost John domains which are not
John domains are outward cusps
such as \eqref{eq-outward-cusp},
and the domain in Example~\ref{ex-shrinking-pins}.
In both cases we can take
$F=\itoverline{B((0,0),r) \cap \Om}$.

\begin{thm} \label{thm-almost-John-finconn}
If $\Om$ is an almost John domain,
then it is finitely connected at the boundary.
\end{thm}

The converse is false as the domain
\[
     \Om:= \{(x,y,z) \in\R^3 : 0 < y< x^3 < 1 \text{ and } 0 <z<1\}
\]
shows. 
Note that $\Om$ is locally connected at the boundary.

To prove Theorem~\ref{thm-almost-John-finconn} we will use the following lemma.

\begin{lem} \label{lem-finconn-H}
Let
$
       A=\{x \in \bdy \Om : \Om \text{ is not finitely connected at } x\}.
$
Then either $A=\emptyset$, i.e.\ $\Om$ is finitely connected at the boundary,
or
$
    \mathcal{H}_{\infty}^1(A) >0.
$
\end{lem}

This is a special case of Lemma~2.1 in Herron--Koskela~\cite{herron-kosk}.
They prove their result in $\R^n$, but the proof is valid for
the metric spaces under consideration here.
For the reader's convenience we include a proof of
our weaker result since our proof is simpler
and more self-contained than the one in \cite{herron-kosk}.

\begin{proof}
Assume that $A \ne \emptyset$ and let $x_0 \in A$.
By Proposition~\ref{prop1-fin}, there is  $0<r<\diam \Om$ such that
either $N(r)=\infty$ or $x_0 \in \itoverline{H(r)}$.
In either case there is a sequence $\{U_j\}_{j=1}^\infty$
of distinct components of $\Om \cap B(x_0,r)$ such that
$\dist(U_j,x_0) \to 0$ as $j \to \infty$.
Since $\Om$ is connected, we must have
$\bdy U_j \cap (\Om \cap \bdy B(x_0,r)) \ne \emptyset$.

Let $0 < r' < \tfrac{1}{2} r$.
Then for $j$ large enough, we can find $x_j \in U_j$
such that $d(x_0,x_j)=r'$.
As $X$ is complete and hence proper, there is a
convergent subsequence $\{x_{j_k}\}_{k=1}^\infty$
with limit $x' \in X$.
Since the $U_j$ are distinct we see that $x' \in \bdy \Om$.
It also follows that $d(x_0,x')=r'$.

Now let $0 <r'' < \tfrac{1}{2}r$.
For each sufficiently large $k$ 
there is a component $V_k$ of $B(x',r'') \cap \Om$
such that $x_{j_k}\in V_k \subset U_{j_k}$ and
\[
     \dist(V_k,x') <\dist(U_{j_k},x') + 1/k \to 0
     \quad \text{as } k \to \infty.
\]
The components $V_k$ must be distinct. 
Therefore, either $x'$ is in the boundary of infinitely many of the sets 
$V_k$, or else $x'\in \itoverline{H(r'')}$.
Again using Proposition~\ref{prop1-fin}, we see that
$\Om$ cannot be finitely connected at $x'$.

We have thus shown that for every $0< r' < \tfrac{1}{2} r$
there is a point $x' \in A$ such that $d(x_0,x')=r'$.
It follows that
$\mathcal{H}_{\infty}^{1}(A) \ge \tfrac{1}{4} r>0$.
\end{proof}

\begin{proof}[Proof of Theorem~\ref{thm-almost-John-finconn}]
Assume 
that $\Om$ is not finitely connected.
Then by Lemma~\ref{lem-finconn-H} we have
$\mathcal{H}_{\infty}^{1}(A)>0$, where
\[
A=\{x \in \bdy \Om : \Om \text{ is not finitely connected at } x\}.
\]
Let $F \subset \overline{\Om}$ be a closed set
such that $\mathcal{H}_{\infty}^{1}(F) < \mathcal{H}_{\infty}^{1}(A)$.
Then there is  some $x \in A \setm F$.
As $F$ is closed, $\dist(x,F)>0$.
Since finite connectedness at a boundary point  is a local property it follows
that $\Om \setm F$ cannot be finitely connected at $x$.
Hence $\Om \setm F$ is not a John domain, by
Theorem~\ref{thm-John-Ncon}.
Thus $\Om$ cannot be an almost John domain.
\end{proof}

We can now collect the consequences of the results in the previous
section.

\begin{cor}  \label{cor-John-sec10}
Let $\Om$ be an almost John domain.
Then the following are true\/\textup{:}
\begin{enumerate}
\item  \label{John-a}
Every end is divisible by some prime end.
\item \label{John-b}
Every prime end has a singleton impression.
\item \label{John-f}
Every $x\in\partial\Om$ is accessible and there is at least one
prime end with impression $\{x\}$.
\item \label{John-c}
There is a  homeomorphism $\Phi: \clOmP \to \clOmm$
such that $\Phi|_{\Om}$ is the identity map.
\item \label{John-d}
If
\begin{equation} \label{eq-p-cond}
    1\le p\in Q(x)\ne(0,1] \quad \text{for all } x \in \bdy \Om,
\end{equation}
then $\partial_P\Om$ is also the $\Mod_p$-prime end boundary.
\item \label{John-e}
The prime end closure $\clOmP$ is metrizable and compact.
\end{enumerate}
\end{cor}

\begin{proof}
By Theorem~\ref{thm-almost-John-finconn},
$\Om$ is finitely connected at the boundary.

\ref{John-a} and \ref{John-b}
This follows from Proposition~\ref{prop-end-divisible}.

\ref{John-f}--\ref{John-d} This follows from Theorem~\ref{thm-fin-con-homeo}.

\ref{John-e} This follows from Corollary~\ref{metric-lemma-5-12}
and Theorem~\ref{thm-clOmm-cpt-new}.
\end{proof}

We also have the following consequence of a combination of
Theorem~\ref{thm-John-Ncon}, 
Propositions~\ref{prop-uniform} and~\ref{prop-finn-conn-N-ends}
and 
Corollary~\ref{cor-loc-conn}.

\begin{cor}
If $\Om$ is a John domain,
then there is a positive integer $N$, depending only on the doubling
constant, the John constant
and the quasiconvexity constant, such that
for every $x\in\partial\Om$ there is at least one, and at most $N$,
prime ends with impression $\{x\}$.

If $\Om$ is a uniform domain,
then
for every $x\in\partial\Om$ there is exactly one
prime end with impression $\{x\}$.
Moreover, there is a homeomorphism $\Upsilon: \clOmP \to \clOm$
such that $\Upsilon|_{\Om}$ is the identity map.
\end{cor}

We are now ready to formulate and prove 
the main result of this section.

\begin{thm} \label{thm-b-John-cor} \label{thm-singleton-gen}
If $\Om$ is an almost John domain and $p>Q-1$,
then every $\Mod_p$-end
is a prime end with singleton impression.
\end{thm}

Note that the conclusion of Theorem~\ref{thm-singleton-gen} fails if $[E_k]$
is merely  an end or if $p\leq Q-1$, see Examples~\ref{needPrime1} and~\ref{needPrime} .
If \eqref{eq-p-cond} holds, then
the existence of $\Mod_p$-ends at every $x\in\bdry\Om$ follows
from \ref{John-f} and \ref{John-d} in
Corollary~\ref{cor-John-sec10}.

To prove Theorem~\ref{thm-singleton-gen} we need the following lemma about chains of balls in John domains.
This lemma is a variant of a chain condition first formulated by Boman, 
see Boman~\cite{bo}, Haj\l asz--Koskela~\cite{hak}, and the 
references therein.
In this paper we use the following chain condition.

\begin{deff}\label{chainB}
We say that a set $E\subset\Om$ is 
\emph{chain-connected} to $B(x_0,\rho_0)\Subset\Om$ 
if there exists $M>0$ 
such that every $x\in E$
can be connected to the ball
$\Boo=B(x_0,\rho_0)$ by a chain of balls
\[
    \{\Bij: i=0,1,\ldots \text{ and } j=0,1,\ldots,m_i\}
\]
with the following properties\/\textup{:}
\begin{enumerate}
\item \label{first}
For all balls $B$ in the chain, we have $3\la B \subset \Om$.

\item \label{second}
For all $i$ and $j$, the ball $\Bij$ has radius
$\rho_i=2^{-i}\rho_0$ and center $\xij$ such that
$d(\xij,x)\le M\rho_i$.

\item \label{third}
For all $i$, we have $m_i \le M$.

\item \label{fourth}
For large $i$, we have $m_i=0$ and the balls $\Bio$
are centered at $x$.

\item \label{last}
The balls $\Bij$ are ordered lexicographically, i.e.\ $\Bij$ comes before
$B_{i',j'}$ if and only if $i<i'$ or $i=i'$ and $j<j'$.
If $\Bij$ and $B_{i',j'}$ are two neighbors with respect to this ordering,
then $\Bij \cap B_{i',j'}$ is nonempty.
\end{enumerate}
\end{deff}

\begin{lem}   \label{lem-John-chain}
Let $\Om$ be a John domain with a John center $x_0$ and a John constant $C_\Om$.
Let $\rho_0\le \de_\Om(x_0)/4\la$
and $A=C_\Om \de_\Om(x_0)/\rho_0 \ge 4C_\Om \la$.
Then $\Om$ is chain-connected to 
$B_{0,0} := B(x_0,\rho_0)$ in the sense
of Definition~\ref{chainB} with $M=2A$.
\end{lem}

\begin{proof}
For $x \in \Om$, let
$\ga:[0,l_\ga]\to \Om$ be a John curve,
parameterized by arc length, connecting $x=\ga(0)$ to $x_0=\ga(l_\ga)$.
Choose the smallest possible $i_x\in\N$ such that
$4\la C_\Om \rho_{i_x} \le \tfrac{1}{2}\de_\Om(x)$.
Recall that $\rho_{i_x}= 2^{-i_x}\rho_0$.

The first ball $B_{0,0}=B(x_0,\rho_0)$ in the chain clearly satisfies
$4\la B_{0,0} \subset \Om$.
Also, by \eqref{eq-def-John},
\[
d(x_0,x) \le l_\ga \le C_\Om \de_\Om(x_0) = A \rho_0.
\]

Suppose that the ball  $\Bij$ has already been constructed and
that it satisfies \ref{first}.
Let
\[
c=\inf\{t\in[0,l_\ga]: \ga(t)\in\Bij\}.
\]
Assume first that $i<i_x$.
If $c\ge 4\la C_\Om \rho_i$, then
let $\Bijj=B(\xijj,\rho_i)$ with $\xijj = \ga(c)$
be the successor of $\Bij$.
Note that by construction and by \eqref{eq-def-John},
\begin{equation}  \label{eq-de-Om-j}
4\la  \rho_i \le \frac{c}{C_\Om} \le  \de_\Om(\xijj),
\end{equation}
i.e.\ $4\la\Bijj\subset\Om$.

If $c<4\la C_\Om \rho_i$, then let $m_i=j$ and let
$\Biio=B(\xiio,\rho_{i+1})$ with $\xiio=\ga(c)$
be the successor of $\Bij$.
Note that \eqref{eq-de-Om-j} implies
\[
\de_\Om(\xiio) \ge \de_\Om(x_{i,m_i}) - \rho_i
\ge 4\la \rho_i -\rho_i \ge 4\la \rho_{i+1}
\]
and hence $4\la\Biio\subset\Om$.

For $i=i_x$ and $c>0$,  let $\Bixjj=B(\xixjj,\rho_{i_x})$
with $\xixjj = \ga(c)$ be the successor of $\Bixj$.
Note that if $c\ge 4\la C_\Om \rho_{i_x}$, then \eqref{eq-de-Om-j} implies
that $4\la\Bixjj\subset\Om$.
On the other hand, if $0<c<4\la C_\Om \rho_{i_x}$,
then the same conclusion follows from the
fact that
\[
d(\xixjj,x) \le c < 4\la C_\Om \rho_{i_x} \le \tfrac12 \de_\Om(x)
\]
and hence
\[
\de_\Om(\xixjj) \ge \de_\Om(x) - d(\xixjj,x) \ge \tfrac12 \de_\Om(x)
\ge 4\la \rho_{i_x}.
\]
If $i=i_x$ and $c=0$ or if $i>i_x$, then let $\Biio=B(x,\rho_{i+1})$
be the successor of $\Bij$.
Then clearly
\[
4\la \rho_i \le 4\la \rho_{i_x} \le \tfrac12 \de_\Om(x)
\]
and thus $4\la\Biio\subset\Om$.

The balls $\{\Bij: i=0,1,\ldots \text{ and } j=0,1,\ldots,m_i\}$
cover $\ga$ in the direction
from $x_0$ to $x$ and neighboring balls
always have nonempty intersection.
Thus \ref{last} is satisfied.
Also \ref{first} is satisfied by construction and the comments above.

As for the other properties,
note first that if $i>i_x$, then there is only one ball with radius $\rho_i$ and
that ball is centered at $x$.
This proves \ref{fourth}, so it remains to prove \ref{second} and \ref{third}.

For $i=0$ and all $j\le m_0$  we have that
\[
0\le d(x_{0,j},x)  < l_\ga - j\rho_0 \le C_\Om \de_\Om(x_0) -j\rho_0
= (A - j) \rho_0,
\]
showing that $m_0\le A$ and $d(x_{0,j},x) \le A \rho_0$.

Similarly, for $0<i\le i_x$ we have by construction that
\[
0\le d(\xij,x) < 4\la C_\Om \rho_{i-1} - j\rho_i = (8\la C_\Om - j) \rho_i
\]
and hence $j<8\la C_\Om \le 2A$.
This also shows that $d(\xij,x) < 2A \rho_i$.

For $i>i_x$, \ref{second} and \ref{third} are obvious.
\end{proof}

\begin{cor} \label{cor-John-H<Mod}
Let $\Om$ be a John domain with a John center $x_0$ and a John constant $C_\Om$.
Let $\rho_0\le \de_\Om(x_0)/4\la$ and $B= B(x_0,\rho_0)$.
If $p>Q-1$, then there exists a constant $C>0$ depending only on
$C_\Om$, $B$, $p$, the doubling constant and the constants in the
\p-Poincar\'e inequality, such that for all open $E\subset \Om\setm B$,
\[
\mathcal{H}_\infty^1(E) \le
C \Modp(E,B,\Om).
\]
\end{cor}

\begin{proof}
This follows directly from Lemmas~\ref{lem-John-chain}
and~\ref{lem-chain-imp-length-est}.
\end{proof}

\begin{proof}[Proof of Theorem~\ref{thm-singleton-gen}]
Let $0<r<\diam\Om$ and $F$ be the set associated with $r$
as in Definition~\ref{def-almost-John}.
Given an end $[E_k]$,
set $E_k'=E_k\setm F$ and $\Om'=\Om\setm F$.
By Remark~\ref{rmk-open} we may assume that the sets $E_k$ are 
open, and hence so are $E_k'$.
Let $x_0$ be a John center of $\Om'$ and $B=B(x_0,\rho)\Subset\Om'$.
If $k$ is large enough, then $E_k \cap B =\emptyset$.
We consider only such $k$ in the rest of the proof.

Every curve connecting $B$ to $E_k'$ in $\Om'$
connects $B$ to $E_k\supset E_k'$
in $\Om$ and hence
\[
    \Modp(E_k',B,\Om') \le \Modp(E_k,B,\Om).
\]
As $\Om'$ is a John domain, this together with
Corollary~\ref{cor-John-H<Mod} implies that
\[
   \mathcal{H}_{\infty}^{1}(E_k')
   \le C \Modp( E_k',B, \Om')
   \le C \Modp(E_k,B,\Om),
\]
where $C$ depends on $r$ but not on $E_k$.
Since $E_k$ is connected, it follows that
\[
   \diam E_k \le \mathcal{H}_{\infty}^{1}(E_k)
   \le \mathcal{H}_{\infty}^{1}(F) + \mathcal{H}_{\infty}^{1}(E_k')
       \le r + C \Modp(E_k,B,\Om).
\]
As $[E_k]$ is a $\Modp$-end, we know that
$\lim_{k\to\infty} \Modp(E_k,B,\Om)=0$.
Hence,
\[
  \limsup_{k\to\infty} \diam E_k \le r.
\]
Letting $r\to 0$ shows that $\lim_{k \to \infty} \diam E_k=0$, and
an application of
Proposition~\ref{prop-end-single}
(and Lemma~\ref{lem-single-char}) completes the proof.
\end{proof}

\begin{prop}  \label{prop-almost-John-prime-singleton}
Let $\Om$ be an almost John domain.
Then the following are equivalent\/\textup{:}
\begin{enumerate}
\item \label{i4-single}
$[E_k]$ is a singleton end\textup{;}
\item \label{i4-prime}
\setcounter{saveenumi}{\value{enumi}}
$[E_k]$ is a prime end.
\end{enumerate}

If moreover, $p>Q-1$ and $1 \le p\in Q(x)\ne(0,1]$ for all $x\in\bdry\Om$,
then the following
statements are also equivalent to the statements above\/\textup{:}
\begin{enumerate}
\setcounter{enumi}{\value{saveenumi}}
\item \label{i4-Modp}
$[E_k]$ is a $\Modp$-end\textup{;}
\item \label{i4-Modp-prime}
$[E_k]$ is a $\Modp$-prime end.
\end{enumerate}
In particular, the $\Modp$-end boundary coincides with the prime end boundary
$\bdy_P \Om$.
\end{prop}

\begin{proof}
\ref{i4-single} $\imp$ \ref{i4-prime}
This follows from Proposition~\ref{prop-end-single}.

\ref{i4-prime} $\imp$ \ref{i4-single}
This follows from Corollary~\ref{cor-John-sec10}\,\ref{John-b}.

Assume finally that 
$p$ is as in the statement of the proposition.

\ref{i4-single} $\imp$ \ref{i4-Modp-prime}
This follows from Proposition~\ref{prop-single-Modp}

\ref{i4-Modp-prime} \imp \ref{i4-Modp}
This is trivial.

\ref{i4-Modp} \imp \ref{i4-single}
This follows from Theorem~\ref{thm-singleton-gen}.
\end{proof}

\section*{Appendix. Modulus and capacity estimates}

\setcounter{equation}{0}
\setcounter{section}{1}
\setcounter{theorem}{0}
\renewcommand{\thesection}{\textup{\Alph{section}}}%
\addcontentsline{toc}{section}{Appendix}

In this appendix, we will provide several estimates for the
modulus and capacity needed in our study of prime ends.
Recall that $p \ge 1$ and that $\la\ge 1$ denotes the dilation constant
in the \p-Poincar\'e inequality.

The following lemma will be important for our estimates.

\begin{lemma}\label{MP}
For any choice of  disjoint sets  $E,F\subset \Om$ we have
\begin{equation}  \label{eq-capp=modp}
   \Modp(E,F,\Om) = \capp(E,F,\Om),
\end{equation}
where $\capp(E,F,\Om)$ is the \emph{\p-capacity}
of the condenser $(E,F,\Om)$ defined by
 \begin{equation} \label{eq-capp-minimizer}
    \capp(E,F,\Om):=\inf_u \int_\Om g_u^p\,d\mu,
 \end{equation}
 with the infimum  taken over all $u\in N^{1,p}(\Om)$
 satisfying $0\le u\le 1$ on $\Om$, $u=1$ on $E$, and $u=0$ on $F$.
 \end{lemma}

Note that both $\Modp$ and $\capp$ are symmetric with respect to the
first two arguments.

For compact $E$ and $F$, equality~\eqref{eq-capp=modp} was obtained 
by  Kallunki--Shan\-mu\-ga\-lin\-gam~\cite{KaSh}, Theorem~1.1,
with a more involved proof,
whereas 
Heinonen--Koskela~\cite{hk}, Proposition~2.17,
obtained this result using a different definition of the capacity.
At that time it was not known if the two definitions give
the same capacity, and  it was the measurability result from
Theorem~1.11 in
J\"arvenp\"a\"a--J\"arvenp\"a\"a--Rogovin--Rogovin--Shanmugalingam~\cite{jjrrs}
that made it possible to prove the lemma in its present form.
To do so we need the following localization of 
Theorem~1.11 in~\cite{jjrrs}.

\begin{lem}  \label{lem-local-jjrrs}
Assume that $\rho\in L^p\loc(\Om)$ is an upper gradient in $\Om$ of
$u:\Om\to\oR$.
Then $u$ is measurable in $\Om$.
\end{lem}

\begin{proof}
For $k>0$, let $u_k=\min\{k,\max\{-k,u\}\}$ be the truncation of $u$ 
at levels $\pm k$.
Let $\Om'\Subset\Om$ be an arbitrary open set and find  
a nonnegative Lipschitz
function $\eta$ with $\spt \eta \subset \Om$ such 
that $\eta=1$ on $\Om'$ and $0 \le \eta \le 1$ on $X$.
Let $g\in L^p(X)$ be an upper gradient of $\eta$ in $X$.

We shall show that the function $\rho\chi_{\spt\eta} + kg$
is an upper gradient of 
$u_k\eta$ in $X$.
To do so, let $\ga:[0,l_\ga]\to X$ be a nonconstant rectifiable
curve in $X$. 
If $\ga\subset X\setminus\spt\eta$, then
$|(u_k\eta)(\ga(0))-(u_k\eta)(\ga(l_\ga))|=0 \le \int_\ga (\rho\chi_{\spt\eta}+kg) \,ds$.
Otherwise, by splitting $\ga$ into two parts and possibly
reversing the orientation, we can assume that 
$x:=\ga(0)\in\spt\eta$, and
that either $\ga \subset \supp \eta$ or that
$\ga(l_\ga) \notin \supp \eta$.
In the latter case, let $c=\inf\{t: \ga(t) \notin \supp \eta\}$,
so that 
$\ga([0,c])\subset\spt\eta$ and $\eta(\ga(c))=0$.
In the former case let $c=l_\ga$.
In both cases we
have, with $\ga'=\ga|_{[0,c]}$, $z=\ga(c)$ and $y=\ga(l_\ga)$, that
\begin{align*}
|(u_k\eta)(x)-(u_k\eta)(y)| &= |(u_k\eta)(x)-(u_k\eta)(z)| \\
&\le |u_k(x)| |\eta(x)-\eta(z)| + |\eta(z)| |u_k(x)-u_k(z)| \\
&\le k \int_{\ga'} g\,ds +  \int_{\ga'} \rho\,ds \\
&\le \int_\ga (\rho\chi_{\spt\eta} + kg)\,ds.
\end{align*}
As $\ga$ was arbitrary, $\rho\chi_{\spt\eta} + kg \in L^p(X)$ is an upper 
gradient of $u_k\eta$ in $X$.
Theorem~1.11 in 
\cite{jjrrs}
implies that $u_k\eta$ is measurable in $X$. 
Letting $k\to\infty$ implies that $u=\lim_{k\to \infty}u_k\eta$ 
is measurable in $\Om'$.
Since $\Om'\Subset\Om$ was arbitrary, the result follows.
\end{proof}

\begin{proof}[Proof of Lemma~\ref{MP}]
To see the validity of (\ref{eq-capp=modp}),
note that by (\ref{eq:upperGrad})
and the fact that for the minimal $p$-weak upper gradient $g_v$ of
an (everywhere defined) function
$v\in \Np(\Om)$ there are upper gradients $g_j$ of $v$
such that
$g_j \to g_v$ in $L^p(\Om)$ (see Koskela--MacManus~\cite{KoMc}),
we have
$\capp(E,F,\Om)\geq \Modp(E,F,\Om)$. On the other hand, if 
$\rho\in L^p(\Om)$ is
an admissible function used for
computing $\Modp(E,F,\Om)$, then we define a function $u$ on $\Om$ by
\[
 u(x)=\min\left\{1, \inf_{\gamma_{E, x}}\int_{\gamma_{E, x}}\rho \, ds\right\},
\]
where 
the infimum is taken over all rectifiable curves connecting $E$ to $x$ in $\Om$.
Observe that $u=0$ on $E$, $u=1$ on $F$ and
$\rho$ is an upper gradient of $u$,
by Lemma~3.1 in Bj\"orn--Bj\"orn--Shanmugalingam~\cite{BBS5}
(or Lemma~5.25 in Bj\"orn--Bj\"orn~\cite{BBbook}).
By Lemma~\ref{lem-local-jjrrs},
the function $u$ is measurable in $\Om$,
and since $|u|\leq 1$ and $\Om$ is bounded,
it follows that $u\in N^{1,p}(\Om)$ and
\[
    \capp(E,F,\Om)\leq \int_\Om g_u^p\,d\mu\le\int_\Om\rho^p\, d\mu.
\]
Hence,
by taking infimum over all such $\rho$ we conclude that
\[
     \capp(E,F,\Om)\leq \Modp(E,F,\Om).
     \qedhere
\]
\end{proof}

\begin{lem}   \label{lem-mod>0}
Let $E, F \subset\Om$ be disjoint and with nonempty interiors.
Then
\[
\Modp(E,F,\Om) = \capp(E,F,\Om) >0.
\]
\end{lem}

\begin{proof}
The equality follows from Lemma~\ref{MP}.
It is thus enough to show that
\[
\int_\Om g_u^p\,d\mu \ge c>0
\]
for every $u\in N^{1,p}(\Omega)$ such that
$u=1$ on $E$ and $u=0$ on $F$.
Note that if there are no such  functions $u$, then the theorem
holds trivially since then $\capp(E,F,\Om)=\infty>0$.

Let $x$ and $y$ be points in the interiors of $E$ and $F$, respectively.
Since $X$ is quasiconvex,
Lemma~4.38 in Bj\"orn--Bj\"orn~\cite{BBbook} 
implies that $\Om$ is
rectifiably connected and we can thus find
a rectifiable curve $\gamma:[0,l_\ga]\to\Omega$
connecting $x$ to $y$.
Let $0<r<\dist(\ga,X\setm\Om)/3\la$
be such that both $B(x,r)\subset E$ and
$B(y,r)\subset F$.
Cover $\ga$ by balls $B_j=B(x_j,r)$, $j=0,1,\ldots,n$, such that
$B_0=B(x,r)$, $B_n=B(y,r)$ and $B_j\cap B_{j+1}$ is nonempty for
$j=0,1,\ldots,n-1$.
Then $B_{j+1}\subset 3B_j$ and $B_j \subset 3B_{j+1}$ for $j=0,1,\ldots,n-1$.

Let $u\in N^{1,p}(\Omega)$ be such that
$u=1$ on $E$ and $u=0$ on $F$.
Then, since $\mu$ is doubling,
\[
|u_{B_j}-u_{B_{j+1}}| \le |u_{B_j}-u_{3B_{j}}| + |u_{B_{j+1}}-u_{3B_{j}}|
\le C\vint_{3B_j}|u-u_{3B_j}|\,d\mu.
\]
The \p-Poincar\'e inequality
then yields that
\begin{align*}
  1&=|u_{B_0}-u_{B_n}|
  \le \sum_{j=0}^{n-1}|u_{B_j}-u_{B_{j+1}}|
          \le C \sum_{j=0}^{n-1}\vint_{3B_j}|u-u_{3B_j}| \,d\mu\\
  &\le C \sum_{j=0}^{n-1} r \left(\vint_{3\lambda B_j}g_u^p\,d\mu\right)^{1/p}
           \le \Ct \left(\int_\Omega g_u^p\,d\mu\right)^{1/p},
\end{align*}
where $\Ct$ is independent of $u$
(but depends on $r$ and the sequence of balls $\{B_j\}_{j=0}^n$).
Taking infimum over all admissible functions $u$ yields the desired result.
\end{proof}

The following estimate is crucial for showing
that singleton ends are $\Modp$-ends in Proposition~\ref{prop-single-Modp}.
Recall that $Q(x)$ was defined in Definition~\ref{ptwise-dim}.

\begin{lem}  \label{lem-cap-0-mod-0}
Let $x\in \overline{\Om}$.
If $1 \le p\in Q(x)\ne(0,1]$, then
for every compact $K\subset\Om\setminus\{x\}$,
\begin{equation}   \label{eq-mod-to-0}
\lim_{r\to0} \Modp (B(x,r)\cap\Om,K,\Om) = 0.
\end{equation}
\end{lem}

The following example shows that we cannot allow $p=1$ and $Q(x)=(0,1]$
in Lemma~\ref{lem-cap-0-mod-0}.

\begin{example}\label{ex-Q-not01}
Let $X=\R$ (unweighted), $\Om=(-1,1)$, $K=\{0\}$,
$x=1$ and $0<r<1$.
Then every function $u$ admissible for $\capp (B(x,r)\cap\Om,K,\Om)$
satisfies $u(0)=0$ and  $u(1-r)=1$.
Hence
\[
\int_{-1}^1 |u'(t)|\,dt \ge u(1-r)-u(0) = 1,
\]
resulting in $\capp (B(x,r)\cap\Om,K,\Om) \ge 1$.
In view of Lemma~\ref{MP}, we see that
\eqref{eq-mod-to-0} fails.
\end{example}

\begin{proof}[Proof of Lemma~\ref{lem-cap-0-mod-0}]
By Lemma~\ref{MP} it suffices to show that
\begin{equation}  \label{eq-show-cap-to-0}
\lim_{r\to0} \capp (B(x,r)\cap\Om,K,\Om) = 0.
\end{equation}
Assume first that $1<p<q_0$, where $q_0$ is the right end point of $Q(x)$.
Choose $\eps>0$ such that $q:=p+\eps\in Q(x)$.
Theorem~3.3 in Garofalo--Marola~\cite{GaMa}
together with the upper mass bound estimate~\eqref{upper-mass-bound-q}
implies that for all sufficiently small $R>0$ there exists $C(R)>0$ such that
for all $0<r<R$ we have
\begin{align}  \label{eq-GM-p-le-q0}
\capp(B(x,r),X \setm B(x,R),X) &\le C(R) r^{-p} \mu(B(x,r)) \\
&\le C(R) r^{-p} C_q \mu(B(x,R)) \Bigl(\frac{r}{R}\Bigr)^q \nonumber\\
&= C(R)  C_q R^{-q} \mu(B(x,R)) r^\eps \to0 \quad \text{as } r\to0.
\end{align}
If instead $Q(x)=(0,q_0]\ne(0,1]$ and $p=q_0$, then $p>1$ and
the estimate from Theorem~3.3 in \cite{GaMa} becomes
\begin{align}  \label{eq-GM-p-q0}
\capp(B(x,r),X \setm B(x,R),X) &\le C(R) \Bigl( \log\frac{R}{r} \Bigr)^{1-p}\to0
\quad \text{as } r\to0.
\end{align}

Now, let $R>0$ be sufficiently small and such that $B(x,R)\subset X\setm K$.
Since every $u$ admissible in the definition of 
$\capp(B(x,r),X \setm B(x,R),X)$
is also admissible for $\capp (B(x,r)\cap\Om,K,\Om)$,
we conclude from
\eqref{eq-GM-p-le-q0} and~\eqref{eq-GM-p-q0} that
\eqref{eq-show-cap-to-0} holds for all $1<p\in Q(x)$.

Finally, if $p=1\in Q(x)\ne(0,1]$, then we have by above that
\eqref{eq-mod-to-0} holds for some $q>1$.
As $\Om$ is bounded, the H\"older inequality implies that
for every $\rho$ admissible in the definition of $\Mod_q (B(x,r)\cap\Om,K,\Om)$
(and thus also for $\Mod_1 (B(x,r)\cap\Om,K,\Om))$
we have
\[
\int_\Om \rho\,d\mu \leq \mu(\Om)^{1-1/q} \bigg(\int_\Om \rho^{q} \,d\mu\bigg)^{1/q}.
\]
Taking infimum over all such $u$ shows that
\[
\Mod_1 (B(x,r)\cap\Om,K,\Om) \leq \mu(\Om)^{1-1/q} \Mod_q (B(x,r)\cap\Om,K,\Om)^{1/q}
\]
and \eqref{eq-mod-to-0} holds also for $p=1$ in this case.
\end{proof}

Next, we shall relate the modulus to the Hausdorff content.

\begin{lem}  \label{lem-chain-imp-length-est}
Let $E\subset\Om$ be open
and $B(x_0,r)\Subset\Om\setm E$.
Assume that there exist $M>0$ and $0<\rho_0\le r$
such that $E$ 
can be chain-connected to the ball  
$\Boo=B(x_0,\rho_0)$ as in
Definition~\ref{chainB}.
Let $s>0$ and $p>Q-s$.
Then  there exists a constant $C$ depending only on
$M$, $p$, $s$, $Q$, $r$,
the doubling  constant $C_\mu$ and on the constants in the Poincar\'e inequality
such that
\[
\mathcal{H}_\infty^s(E) \le C \capp(E,B(x_0,r),\Om)
= C \Modp(E,B(x_0,r),\Om).
\]
\end{lem}

The proof of this lemma is based on a technique 
introduced in Haj\l asz--Koskela~\cite{hak95}.

\begin{proof}
In view of Lemma~\ref{MP} it suffices to estimate
$\mathcal{H}_\infty^s(E)$ using $\capp(E,B(x_0,r),\Om)$.
Let
$u\in\Np(\Om)$ be such that $u=0$ on $B(x_0,r)$ and $u=1$ on $E$.
Consider $x\in E$ and let
$\mathcal{C}_x=\{\Bij:i=0,1,\ldots \text{ and } j=0,1,\ldots,m_i\}$
be the corresponding chain.
For each ball $B$ in the chain let $B^*$ be its immediate successor.
Then $B\cap B^*$ is nonempty and $\tfrac{1}{2}r(B) \le r(B^*) \le r(B)$,
where $r(B)$ is the radius of $B$.
Thus $B^*\subset 3B$ and $B \subset 5B^*$.
Note also that properties \ref{second}, \ref{third} and \ref{last}
imply that for all
$i=0,1,\ldots$ and $j=0,1,\ldots,m_i$, we have
\begin{equation}
d(\xij,x_0) \le (M+1) \sum_{k=0}^i 2\rho_k < 4(M+1)\rho_0.
\label{eq-Bij-subset-Boo}
\end{equation}
Since $u=1$ in the open set $E$ containing $x$, a telescopic argument together with assumption~\ref{fourth} implies that
\begin{align}
 1 
&= \lim_{i\to\infty} |u_{\Bio} - u_{\Boo}|
 \le \sum_{B\in \mathcal{C}_x} |u_{B}-u_{B^*}|
\le \sum_{B\in \mathcal{C}_x} (|u_{B}-u_{3B}| + |u_{B^*}-u_{3B}|).
\label{eq-telescopic}
\end{align}
The doubling property and the \p-Poincar\'e inequality yield
\[
|u_{B^*}-u_{3B}| \le C \vint_{3B} |u-u_{3B}|\,d\mu
\le Cr(B) \biggl( \vint_{3\la B} g_u^p\,d\mu \biggr)^{1/p}.
\]
The difference $|u_{B}-u_{3B}|$ is estimated similarly and inserting
both estimates into \eqref{eq-telescopic} implies that
\begin{align*}
  1 &\le C \sum_{B\in \mathcal{C}_x} \frac{r(B)}{\mu(3\la B)^{1/p}}
   \biggl( \int_{3\la B} g_u^p\,d\mu \biggr)^{1/p}.
\end{align*}
For each $B\in \mathcal{C}_x$,
\eqref{eq-Bij-subset-Boo} together with  \eqref{lower-mass-bound} gives
\[
\mu(3\la B) \ge \mu(B)
   \ge C \biggl( \frac{r(B)}{4(M+1)\rho_0}\biggr)^Q \mu(4(M+1)\Boo)
   \ge C \biggl( \frac{r(B)}{4(M+1)\rho_0}\biggr)^Q \mu(\Boo).
\]
The last estimate then becomes
\begin{align*}
  1  &\le \frac{C \rho_0^{{Q}/{p}}}{\mu(\Boo)^{1/p}}
       \sum_{B\in \mathcal{C}_x} r(B)^{1-Q/p}
           \biggl( \int_{3\la B} g_u^p\,d\mu \biggr)^{1/p},
\end{align*}
where $C$ depends only on $M$, $p$, $Q$,
the doubling  constant $C_\mu$ and on the constants in the Poincar\'e inequality,
but not on $u$.

Since $p>Q-s$, we have $p-Q+s>0$ and hence
\[
1 = C \sum_{i=1}^\infty 2^{-i(p-Q+s)/p}
\ge \frac{C}{M} \sum_{B\in \mathcal{C}_x} \biggl( \frac{r(B)}{\rho_0} \biggr)^{(p-Q+s)/p},
\]
where $C$ depends only on $p$, $Q$ and $s$.
Comparing the last two estimates we see that there exists a ball
$B_x\in\mathcal{C}_x$ such that
\begin{align*}
\biggl( \frac{r(B_x)}{r} \biggr)^{(p-Q+s)/p}
  &\le \biggl( \frac{r(B_x)}{\rho_0} \biggr)^{(p-Q+s)/p}\\
  &\le \frac{C r^{Q/p}}{\mu(\Boo)^{1/p}} r(B_x)^{1-Q/p}
\biggl( \int_{3\la B_x} g_u^p\,d\mu \biggr)^{1/p},
\end{align*}
where $C$ depends only on $M$, $p$, $Q$, $s$, $C_\mu$ and
the constants in the Poincar\'e inequality,
but not on $u$ or $x$.

Repeating this argument for every $x\in E$, we obtain balls $B_x$,
such that
\begin{equation}  \label{eq-est-rBx-gu}
r(B_x)^s \le \frac{C r^{p+s}}{\mu(\Boo)} \int_{3\la B_x} g_u^p\,d\mu.
\end{equation}
Note that $x\in 2MB_x$ for all $x\in E$ by \ref{second}.
Hence, the balls $\{2M B_x\}_{x\in E}$ cover $E$, as
do the balls $\{3M\la B_x\}_{x \in E}$. The 5-covering lemma
(Theorem~1.2 in Heinonen~\cite{he})
allows us to choose pairwise disjoint balls $3M\la B_{x_i}$,
$i=1,2\ldots,$ so that
$E\subset \bigcup_{i=1}^{\infty} 15M\la B_{x_i}$.
In particular, the balls $3\la B_{x_i}$, $i=1,2\ldots,$ are pairwise disjoint.
Thus we get from~\eqref{eq-est-rBx-gu}
that
\begin{align*}
\mathcal{H}^s_\infty(E) &\le \sum_{i=1}^\infty r(15M\la B_{x_i})^s
   = 15^s M^s\la^s \sum_{i=1}^\infty r(B_{x_i})^s \\
   &\le \frac{C r^{p+s}}{\mu(\Boo)} \sum_{i=1}^\infty
            \int_{3\la B_{x_i}} g_u^p\,d\mu
   \le \frac{C r^{p+s}}{\mu(\Boo)} \int_\Omega g_u^p.
\end{align*}
Taking infimum over all admissible functions $u$ completes the proof.
\end{proof}

\begin{lem}  \label{lem-ex-chain}
Let $E\Subset \Omega$ and $B=B(x_0,r)\Subset \Om\setm E$.
Then there exists $0<\rho_0<r$ such that $E$ can be 
chain-connected to
the ball $\Boo=B(x_0,\rho_0)$ 
as in Definition~\ref{chainB}.
\end{lem}

\begin{proof}
Since $\Omega$ is connected, there exists $0<\eps<r$ such that
both $\itoverline{B}$ and $\itoverline{E}$ belong to
the same component $G$ of
\[
\Omega_\eps:=\{x\in\Omega\, :\,  \dist(x,X\bs\Omega)>\eps\},
\]
see Lemma~4.49 in Bj\"orn--Bj\"orn~\cite{BBbook}.
Choose $0<\rho_0\le\eps/6\la$ and let
$B_i=B(x_i,\rho_0/2)$, $i=1,\ldots,N,$ be a maximal pairwise disjoint
collection of balls with centers in $\Om_\eps$.
By the doubling property, there are only finitely many such balls
and their number $N$ depends only on $\eps$, $\rho_0$
and the doubling constant $C_\mu$.
 The balls $\{2B_i\}_{i=1}^N$ cover $\Om_\eps$ and
$6\lambda B_i\subset\Omega$ for all $i=1,2,\ldots,N.$

Let $x\in E$ be arbitrary.
By pathconnectedness of the component $G$,
there exists a curve $\ga$ in $G$ from $x_0$ to $x$.
We can therefore among the balls $2B_i$, $i=1,2,\ldots,N,$ choose a
minimal chain of balls
covering $\ga$.
Number these balls in the direction from $x_0$ to $x$ and
call them $\Boj$, $j=1,2,\ldots, m_0$.
Clearly, $m_0\le N$ and neighboring balls in the chain have nonempty
intersection.
Complete the chain by the balls $\Bio=B(x,\rho_i)$,
where $\rho_i=2^{-i}\rho_0$, $i=1,2,\ldots.$

It remains to verify that the conditions \ref{first}--\ref{last}
of Lemma~\ref{lem-chain-imp-length-est} are satisfied.
The only property that needs some justification is that
$d(\xij,x)\le M\rho_i$ with $M=\max\{N,2/\eps \rho_0\}$.
For $i\ge1$, this is trivial and for $i=0$ we have
$d(\xoj,x)\le \diam\Om_\eps\le2/\eps$.
The other properties follow by construction.
\end{proof}

\begin{rem}
The proof of Lemma~\ref{lem-ex-chain} shows that $M=\max\{N,2/\eps \rho_0\}$.
It follows that $M$ (and hence also $C$ in
Lemma~\ref{lem-chain-imp-length-est}) depends on $\dist(E,X\setm\Om)$.
The estimate in Lemma~\ref{lem-chain-imp-length-est}
therefore does not apply if we only know that $E\subset{\Om}$.
Indeed, in the topologist's comb in Example~\ref{comb} we have
for $p \le 2$, every compact $K\subset\Om$, 
\[
\Modp(E_k,K,\Om) = \Modp(F_k,K,\Om)
\to 0, \quad\text{as } k\to\infty,
\]
by Proposition~\ref{prop-single-Modp},
where $F_k=\bigl(\bigl(\tfrac{1}{2}-2^{-k},\tfrac{1}{2}+2^{-k}\bigr)
 \times (0,2^{-k})\bigr)\cap\Om$.
On the other hand, $\mathcal{H}^1_\infty(E_k)\ge\tfrac{1}{2}$ for all $k=1,2,\ldots$.
See, however, Lemma~\ref{lem-John-chain}.
\end{rem}

\begin{cor}    \label{cor-modp-Q-1}
 Let $E\Subset \Omega$ be open and $B\Subset \Om\setm E$ be a ball.
If $p>Q-1$, then there exists $C>0$ depending on $\dist(E,X\setm \Om)$ such that
\[
\mathcal{H}_\infty^1(E) \le C \Modp(E,B,\Om).
\]
\end{cor}

\begin{proof}
This follows directly from
Lemmas~\ref{lem-chain-imp-length-est}
and~\ref{lem-ex-chain}.
\end{proof}

\begin{lem}     \label{lem-imp-subset-bdry}
Let $\{E_k\}_{k=1}^\infty$ be a sequence of open acceptable sets
satisfying $\itoverline{E}_{k+1}\cap\Om\subset E_k$ for each $k$. If
$\lim_{k\to\infty}\Modp(E_k,B,\Om)=0$ for some ball $B\subset\Om\setm E_1$ and 
$p>Q-1$, then 
$I:=\bigcap_{k=1}^\infty \itoverline{E}_k\subset\bdry\Om$.
\end{lem}

\begin{proof}
Clearly, $I$ is a compact connected set and $I\cap\bdry\Om$ is nonempty.
Let 
\[
\Om_\de=\{x\in\Om:\dist(x,X\setm\Om)>\de\}.
\]
Corollary~\ref{cor-modp-Q-1} applied to $E_k\cap\Om_{\de}$ shows that
\(
\mathcal{H}_\infty^1(I\cap\Om_\de)=0
\)
for each $\de>0$. It follows that $\mathcal{H}^1(I\cap\Om_\de)=0$ and hence
also $\mathcal{H}^1(I\cap\Om)=0$,
where $\mathcal{H}^1$ denotes the one-dimensional Hausdorff measure.

Assume that $I\cap\Om$ is nonempty and find $\de>0$ such that
$I\cap\Om_{2\de}\ne\emptyset$.
Fix $x\in I\cap\Om_{2\de}$ and let $U$ be the component of $I\cap\Om$
containing $x$.
Note that $U$ is connected but not necessarily pathconnected.
Let $0<\eps<\de/2,$ 
and
cover $U$ by (finitely or countably many) balls $B_j=B(x_j,r_j)$ so that
$\sum_{j} r_j <\eps$.
Let $V$ consist of all points $y\in U$ for which there exists a chain
$\{B_{j_k}\}_{k=1}^{N_y}$ (depending on $y$) of balls from this cover
such that $x\in B_{j_1}$,
$y\in B_{N_y}$ and $B_{j_k}\cap B_{j_{k+1}}\ne\emptyset$ for all $k$.
If $y\in V\cap B_j$ for some $j$, then $U\cap B_j\subset V$ and hence
$V$ is open in $U$.
Similarly, $U\setm V$ is open in $U$ and must therefore be empty, since
$U$ is connected.

Assume next that $U\cap\bdry\Om_\de=\emptyset$.
Then the connected set $I$ could be written as a disjoint union of the
nonempty relatively open sets $U\cap\Om_\de$ and
$I\setm\overline{U\cap\Om_\de}$, which is a contradiction.
Thus, there exists $z\in U\cap\bdry\Om_\de\subset V$. 
Hence there is a chain $\{B_{j_k}\}_{k=1}^{N_z}$ of balls from the cover
of $U$ above satisfying $x\in B_{j_1}$, $z\in B_{j_{N_z}}$, and $B_{j_k}\cap B_{j_{k+1}}$
nonempty, and so
\[
  \delta\le d(z,x)\le \sum_{k=1}^{N_z}2r_{j_k}\le 2\sum_j r_j<2\eps<\delta,
\]
which is not possible. 
This contradiction shows that $I\cap\Om$ must be empty,
i.e.\ $I\subset\bdry\Om$.
\end{proof}

\begin{lem}\label{lem-cpt-equiv}
Let $[E_k]$ be an end and $p>1$.
Then $\lim_{j\to\infty}\Modp(E_j,K,\Om)=0$
for every compact $K\subset\Om$ if and only if
$\lim_{j\to\infty}\Modp(E_j,K_0,\Om)=0$
for some compact  $K_0\subset\Om$ with $\Cp(K_0)>0$.
\end{lem}

Here
\begin{equation*} 
  \Cp (E) :=\inf    \|u\|_{\Np(X)}^p,
\end{equation*}
where the infimum is taken over all everywhere defined 
functions $u\in \Np (X) $ such that
$u \ge1$ on $E$.
We say that a property 
holds \emph{quasieverywhere} (q.e.)
if the set of points  for which it fails
has $\Cp$ capacity zero.

An alternative proof of this result,
using superharmonic functions, 
is given as Proposition~5.14 in 
the forthcoming paper Bj\"orn--Bj\"orn~\cite{BBnonopen}
(which needs to be combined with Lemma~\ref{MP} to give 
Lemma~\ref{lem-cpt-equiv}).

\begin{proof}
Assume that $\lim_{j\to\infty}\Modp(E_j,K_0, \Om)=0$ for some compact set
$K_0\subset\Om$ with positive capacity,
and let $K\subset\Om$ be compact.

By Lemma~\ref{MP}, we can equivalently work with capacities.
We can thus find $u_j\in\Np(\Om)$ admissible in the definition of
$\capp(E_j,K_0,\Om)$, such that
$g_{u_j}\to0$ in $L^p(\Om)$, as $j\to\infty$.
Lemma~3.2 in Bj\"orn--Bj\"orn--Parviainen~\cite{BBP} 
or Lemma~6.2 in Bj\"orn--Bj\"orn~\cite{BBbook} provides us with
$u\in\Np(\Om)$ and convex combinations $v_j=\sum_{i=j}^{N_j} a_{j,i} u_i$ and
$g_j=\sum_{i=j}^{N_j} a_{j,i} g_{u_i}$ such that $v_j\to u$ both in $L^p(\Om)$
and q.e., and $g_j\to g$ in $L^p(\Om)$, where $g$ is a \p-weak upper gradient of $u$.
Note that $v_j=0$ on $K_0$ and $v_j=1$ on $E_{N_j}$, $j=1,2,\ldots$.
Since $g_{u_j}\to0$ in $L^p(\Om)$, we see that $g=0$ a.e.\ in $\Om$.
As $\Om$ is connected, the Poincar\'e inequality implies that $u$
is constant q.e.\ in $\Om$.
Since $u=0$ q.e.\ on the set $K_0$ with positive capacity, we must have
$u=0$ q.e.\ in $\Om$,  and hence $v_j\to0$ in $\Np(\Om)$.

Corollary~3.9 in Shanmugalingam~\cite{Sh-rev}
or Corollary~1.72 in~\cite{BBbook}  implies that $v_j\to 0$
locally quasi-uniformly, i.e.\
for every $\eps>0$ we can find $A_\eps\subset\Om$  such that
$\Cp(A_\eps)<\eps$ and $v_j\to 0$ uniformly in $K\setminus A_\eps$.
It follows that for sufficiently large $j$,
we have $v_j<\tfrac12$ on $K\setminus A_\eps$,
and thus the function $w_j=\max\bigl\{0,2\bigl(v_j-\tfrac12\bigr)\bigr\}$
is admissible in the definition of $\capp(E_{N_j},K\setm A_\eps,\Om)$, i.e.\
\begin{equation} \label{eq-wj}
\capp(E_{N_j},K\setm A_\eps,\Om) \le \int_\Om g_{w_j}^p \, d\mu
    \le 2^p \int_\Om g_{v_j}^p \, d\mu
    \le 2^p \int_\Om g_{j}^p \, d\mu.
\end{equation}

Next, if $v\in\Np(X)$ is admissible in the definition of $\Cp(A_\eps)$,
then
for some nonnegative Lipschitz function $\eta$ with compact support
in $\Om$ and such that $\eta=1$ on $K$, the function
$1-v\eta$ is admissible in the definition of $\capp(E_j,K\cap A_\eps,\Om)$,
for sufficiently large $j$.
Hence by the Leibniz rule (Theorem~2.15 in \cite{BBbook}) we obtain
\begin{align*}
\capp(E_j,K\cap A_\eps,\Om) &\le \int_\Om g_{1-v\eta}^p\,d\mu
\le \int_\Om (v g_\eta + \eta g_v)^p\,d\mu \le C\|v\|_{\Np(X)}
\end{align*}
for sufficiently large $j$.
Taking infimum over all such $v$ shows that
$\capp(E_j,K\cap A_\eps,\Om)\le C\eps$ for sufficiently large $j$.
Combining this with \eqref{eq-wj} we obtain that, for sufficiently large $j$,
\[
\capp(E_{N_j},K,\Om)
     \le 2^p \int_\Om g_{j}^p \, d\mu + C\eps
    \to C\eps, \quad \text{as } j\to\infty.
\]
Letting $\eps\to0$ finishes the proof,
since the converse implication is trivial.
\end{proof}

\end{document}